\documentclass[11pt]{article}
\usepackage{hyperref}
\usepackage{amsmath,amssymb,amsthm}
\usepackage[margin=1in]{geometry}
\usepackage{graphicx}
\usepackage{color}
\usepackage{bbm}
\usepackage{booktabs}
\usepackage{breakcites}
\usepackage{bm}
\usepackage[numbers,sort&compress]{natbib}
\usepackage{thm-restate}
\usepackage{color-edits}
\usepackage{thmtools}
\usepackage{mathtools}
\usepackage[inline, shortlabels]{enumitem}
\usepackage{enumitem}
\usepackage{dsfont}
\usepackage{xspace}
\usepackage{hyperref}
\usepackage[algo2e,linesnumbered,ruled,vlined,algosection]{algorithm2e}
\usepackage{algorithmic}
\usepackage{subcaption}
\usepackage{array}
\usepackage[dvipsnames]{xcolor}
\usepackage{cleveref}
\usepackage{makecell}

\theoremstyle{plain}
\declaretheorem[name=Theorem,numberwithin=section]{theorem}
\newtheorem{lemma}[theorem]{Lemma}
\newtheorem{proposition}[theorem]{Proposition}

\newtheorem{corollary}[theorem]{Corollary}

\theoremstyle{definition}
\newtheorem{definition}[theorem]{Definition}

\theoremstyle{remark}
\newtheorem{remark}[theorem]{Remark}

\newcommand{\paren}[1]{\left(#1 \right )}

\newcommand{\Brac}[1]{\left[#1\right]}

\newcommand{\ceil}[1]{\lceil #1 \rceil}

\newcommand{\norm}[1]{\left\lVert#1\right\rVert}

\newcommand{\defeq}{\coloneqq}

\newcommand{\normInline}[1]{\lVert#1\rVert}

\newcommand{\R}{\mathbb R}

\newcommand{\cB}{\mathcal B}

\newcommand{\cS}{\mathcal S}

\newcommand{\cX}{\mathcal X}
\newcommand{\cY}{\mathcal Y}
\newcommand{\cZ}{\mathcal Z}

\DeclareMathOperator*{\argmin}{argmin} 
\DeclareMathOperator*{\argmax}{argmax}

\allowdisplaybreaks
\definecolor{violet}{RGB}{148, 0, 211}
\hypersetup{
    colorlinks=true,
    linkcolor=blue!70!black, %
    citecolor=violet, %
    filecolor=blue!70!black, %
    urlcolor=magenta
}

\newcommand{\otilde}{\tilde{O}}
\newcommand{\Otilde}{\otilde}

\newcommand{\code}[1]{\textnormal{\texttt{#1}}}
\SetKwInput{KwInput}{Input}
\SetKwInput{KwOutput}{Output}
\SetKwInput{KwParameter}{Parameter}
\SetKwProg{Function}{function}{}{}
\SetKwIF{IfNot}{ElseIfNot}{}{if not}{then}{else if not}{}{} %

\SetCommentSty{mycommfont}
\SetKwFor{Repeat}{repeat}{}{}

\usepackage{algorithmic}

\usepackage{booktabs} %
\usepackage{array} %

\newcommand{\ellOneEllOne}{\ell_1\text{-}\ell_1}
\newcommand{\ellTwoEllOne}{\ell_2\text{-}\ell_1}
\newcommand{\ellTwoEllTwo}{\ell_2\text{-}\ell_2}

\newcommand{\gap}{\mathrm{gap}}

\newcommand{\tr}{\mathrm{tr}}

\newcommand{\prox}{\mathrm{prox}}
\newcommand{\proxStep}{\mathsf{SUGMStep}}
\newcommand{\modelUpdateStep}{\code{flag}}
\newcommand{\true}{\code{True}}
\newcommand{\false}{\code{False}}
\newcommand{\guilty}{\code{guilty}}
\newcommand{\smooth}{\code{smooth}}
\newcommand{\flag}{\code{flag}}
\newcommand{\none}{\code{None}}
\newcommand{\B}{\mathbb{B}}
\newcommand{\KL}{\mathrm{KL}}

\newcommand{\xset}{\cX}
\newcommand{\yset}{\cY}
\newcommand{\zset}{\cZ}
\newcommand{\simplex}{\Delta}
\usepackage{etoolbox}
\usepackage{xparse}
\usepackage{mathtools}
\DeclarePairedDelimiterXPP{\inbraces}[1]{}{\{}{\}}{}{#1}
\DeclarePairedDelimiterXPP{\inparen}[1]{}{(}{)}{}{#1} %
\DeclarePairedDelimiterXPP{\insquare}[1]{}{[}{]}{}{#1}
\DeclarePairedDelimiter{\innorm}{\|}{\|}
\DeclarePairedDelimiter{\inabs}{\lvert}{\rvert}

\NewDocumentCommand\breg{s m m }{ %
  \IfBooleanTF{#1} %
    { V^{r}_{#2} \left( #3 \right) }
    { V^{r}_{#2} ( #3 ) }
}
\NewDocumentCommand\bregwr{s O{} m m }{ %
  \IfBooleanTF{#1} %
    { V^{#2}_{#3} \left( #4 \right) }
    { V^{#2}_{#3} ( #4 ) }
}

\newcommand{\rx}{r_\mathsf{x}}
\newcommand{\ry}{r_\mathsf{y}}

\NewDocumentCommand\xbreg{s m m }{
  \IfBooleanTF{#1} %
    { \bregwr*[\rx]{#2}{#3} }
    { \bregwr[\rx]{#2}{#3} }
}
\NewDocumentCommand\ybreg{s m m }{
  \IfBooleanTF{#1} %
    { \bregwr*[\ry]{#2}{#3} }
    { \bregwr[\ry]{#2}{#3} }
}

\newcommand{\gm}{\nabla_{\pm}}

\newcommand{\x}{_\mathsf{x}}
\newcommand{\y}{_\mathsf{y}}

\newcommand{\zopt}{z^\star}

\newcommand{\grad}{\nabla}

\DeclarePairedDelimiterXPP{\dualnorm}[1]{}{\|}{\|}{_{*}}{#1}

\newcommand{\DAPO}{\textsc{DAPO}}
\newcommand{\CWF}{\textsc{CWF}}
\newcommand{\GWF}{\textsc{GWF}}
\newcommand{\ztilde}{\tilde{z}}

\newcommand{\wbar}{\bar{w}}

\newcommand{\overle}[1]{\overset{#1}{\le}}
\newcommand{\overge}[1]{\overset{#1}{\ge}}

\newcommand{\Holder}{H\"{o}lder}

\newcommand{\proxStepFullZ}[6]{\prox^{#3, #4}_{#1 , #2}(#5; #6)}

\newcommand{\proxStepSimple}[3]{\prox^{#2}_{#1}(#3)}
\newcommand{\proxStepSimpleZ}[4]{\prox^{#2}_{#1}(#3; #4)}

\newcommand{\diag}{\mathrm{diag}}

\newcommand{\judge}{\mathsf{Judge}}
\newcommand{\normalize}{\mathrm{unit}}

\newcommand{\xtrunc}{\xset_\nu}
\newcommand{\ytrunc}{\yset_\nu}
\newcommand{\ztrunc}{\zset_\nu}

\newcommand{\gaptrunc}{\gap_{\nu}}

\newcommand{\localize}[2]{(#1)_{#2}}
\newcommand{\ground}[2]{{(#1)}_{#2}}
\newcommand{\unground}[2]{{(#1)}_{#2, *}}
\newcommand{\SUGSMStep}{\mathsf{SUG}\text{-}\mathsf{SM}\text{-}\mathsf{Step}} %
\newcommand{\SUGStronglyMonotoneMirrorProx}{\mathsf{SUG}\text{-}\mathsf{SM}\text{-}\mathsf{MP}}

\newcommand{\trunc}{_{\nu}}
\newcommand{\htrunc}{h_{\nu}}
\newcommand{\gtrunc}{g_{\nu}}
\newcommand{\ball}{\B}
\DeclarePairedDelimiter{\inmaxnorm}{\|}{\|_{\mathrm{max}}}

\newcommand{\regret}{\mathrm{regret}}

\newcommand{\walpha}{w_\alpha}
\newcommand{\wbeta}{w_\beta}
\newcommand{\halpha}{h_\alpha}
\newcommand{\hbeta}{h_\beta}

\newcommand{\hess}{\nabla^2}
\newcommand{\showdiag}{\mathrm{diag}}

\newcommand{\inKL}[2]{\KL(#2||#1)} %

\newcommand{\bilinear}[1]{f_{#1}}

\newcommand{\xtilde}{\tilde{x}}
\newcommand{\ytilde}{\tilde{y}}

\newcommand{\zcenter}{z_c}
\newcommand{\zcenterx}{z_{c \mathsf{x}}}
\newcommand{\zcentery}{z_{c \mathsf{y}}}

\newcommand{\zground}{z_\mathsf{n}}

\newcommand{\xground}{x_\mathsf{n}}

\newcommand{\alphaopt}{\tilde{\alpha}}

\newcommand{\epsprim}{\epsilon'}

\newcommand{\sball}{\cB}

\newcommand{\gammapw}{\gamma_{\mathrm{pw}}}
\newcommand{\gammab}{\gamma_{\mathrm{b}}}
\newcommand{\gammagb}{\gamma_{\mathrm{gb}}}
\newcommand{\gammav}{\gamma_{\mathrm{v}}}

\newcommand{\passModel}{\code{passModel}}

\newcommand{\checkdiv}{\mathsf{CheckDiv}}
\newcommand{\bsearch}{\mathsf{B}\text{-}\mathsf{Search}}

\newcommand{\checkcoords}{\mathsf{CheckCoords}}
\newcommand{\coordsInRange}{\code{coordsInRange}}
\newcommand{\tooBig}{\code{tooBig}}
\newcommand{\tooSmall}{\code{tooSmall}}
\newcommand{\justRight}{\code{justRight}}
\newcommand{\betaDivFlag}{\code{betaDivFlag}}
\newcommand{\alphaDivFlag}{\code{alphaDivFlag}}

\newcommand{\cZint}{\cZ_{\mathrm{int}}}
\newcommand{\zoptbeta}{\zopt_\beta}
\newcommand{\zoptalpha}{\zopt_\alpha}

\newcommand{\zoptlocal}{\zopt_\ell}

\newcommand{\zoptlocaly}{\zopt_{\ell \mathsf{y}}}
\newcommand{\zoptalphaopt}{\zopt_{\alphaopt}}

\newcommand{\Uq}[2]{\innorm{#1}_{(#2)}} %

\newcommand{\zbar}{\bar{z}}

\usepackage{tikz}
\newcommand*\circled[1]{\tikz[baseline=(char.base)]{
    \node[shape=circle,draw,inner sep=1pt] (char) {#1};}}

\newcommand{\oneC}{\circled{1}}
\newcommand{\twoC}{\circled{2}}
\newcommand{\threeC}{\circled{3}}

\newcommand{\Range}{\Gamma}

\newcommand{\entropy}{e}

\newcommand{\project}{\mathrm{proj}}

\newcommand{\smax}{\mathrm{smax}}

\newcommand{\spaceeq}{\qquad} \usepackage{microtype}
\usepackage{times}
  \usepackage{nth}
  \usepackage{intcalc}

  \newcommand{\cSTOC}[1]{\nth{\intcalcSub{#1}{1968}}\ Annual\ ACM\ Symposium\ on\ Theory\ of\ Computing\ (STOC)}

  \newcommand{\cFOCS}[1]{\nth{\intcalcSub{#1}{1959}}\ Annual\ IEEE\ Symposium\ on\ Foundations\ of\ Computer\ Science\ (FOCS)}

  \newcommand{\cCOLT}[1]{\nth{\intcalcSub{#1}{1987}}\ Annual\ Conference\ on\ Computational\ Learning\ Theory\ (COLT)}

  \newcommand{\cSODA}[1]{\nth{\intcalcSub{#1}{1989}}\ Annual\ ACM-SIAM\ Symposium\ on\ Discrete\ Algorithms\ (SODA)}
  \newcommand{\cNIPS}[1]{Advances\ in\ Neural\ Information\ Processing\ Systems\ \intcalcSub{#1}{1987} (NeurIPS)}

  \newcommand{\cICML}[1]{\nth{\intcalcSub{#1}{1983}}\ International\ Conference\ on\ Machine\ Learning\ (ICML)}
  
  \newcommand{\cITCS}[1]{\nth{\intcalcSub{#1}{2009}}\ Conference\ on\ Innovations\ in\ Theoretical\ Computer\ Science\ (ITCS)}

  \newcommand{\cICLR}[1]{\nth{\intcalcSub{#1}{2012}}\ International\ Conference\ on\ Learning\ Representations\ (ICLR)}

  \newcommand{\cEC}[1]{\nth{\intcalcSub{#1}{1999}}\ ACM\ Conference\ on\ Economics\ and\ Computation\ (EC)}

  \newcommand{\cAAAI}[1]{AAAI\ Conference\ on\ Artificial (AAAI)}

\title{Solving {Zero-Sum} Games with Fewer Matrix-Vector Products}
\author{%
    Ishani Karmarkar\thanks{Stanford University, \texttt{\string{ishanik,ocarroll,sidford\string}@stanford.edu}} 
    \and
    Liam O'Carroll\footnotemark[1] 
    \and
    Aaron Sidford\footnotemark[1] 
}
\begin{document}

\pagenumbering{gobble}

\maketitle

\begin{abstract}
In this paper we consider the problem of computing an $\epsilon$-approximate Nash Equilibrium of a zero-sum game in a payoff matrix $A \in \R^{m \times n}$ with $O(1)$-bounded entries given access to a matrix-vector product oracle for $A$ and its transpose $A^\top$. We provide a deterministic algorithm that solves the problem using 
$\tilde{O}(\epsilon^{-8/9})$-oracle queries, where $\tilde{O}(\cdot)$ hides factors polylogarithmic in $m$, $n$, and $\epsilon^{-1}$. Our result improves upon the state-of-the-art query complexity of $\tilde{O}(\epsilon^{-1})$ established by [Nemirovski, 2004] and [Nesterov, 2005]. We obtain this result through a general framework that yields improved deterministic query complexities for solving a broader class of minimax optimization problems which includes computing a linear classifier (hard-margin support vector machine) as well as linear regression. 
\end{abstract}

\setcounter{tocdepth}{2}  %
\tableofcontents

\newpage

\pagenumbering{arabic}

\section{Introduction}\label{sec:intro}

In this paper, we consider the foundational problem of approximately solving zero-sum games:
\begin{align}\label{eq:zero-sum-game}
    \min_{x \in \simplex^n} \max_{y \in \simplex^m} y^\top A x,
\end{align}
where we use $\Delta^k := \{u \in \R^k_{\geq 0} : \sum_{i \in [k]} [u]_i = 1\}$ to denote the $k$-dimensional probability simplex and each entry of $A\in \R^{m \times n}$ is in $[-1, 1]$. Zero-sum games are foundational in theoretical computer science, machine learning, algorithmic game theory, and economics \citep{goodfellow2014generative, mkadry2017towards, nemirovskij1983problem, mcculloch1943logical, rosenblatt1958perceptron} as they provide a basis for understanding strategic decision-making in adversarial settings where one agent's gain is  another agent's loss. 

In the standard form \eqref{eq:zero-sum-game}, there are two players, a minimization player with $n$ actions available and a maximization player with $m$ actions available. The minimization player selects a (mixed) \emph{strategy} $x \in \Delta^n$, where $[x]_i$ denotes the minimization player's probability of playing action $i \in [n]$, and the maximization player picks a strategy $y \in \Delta^m$ where $[y]_j$ is the maximization player's probability of playing $j \in [m]$. The expected payoff of the game is given by $y^\top A x$. The minimization player's goal is to pick a strategy to minimize this expected payoff, while the maximization player's goal is to maximize the expected payoff.

A common approximate solution concept for zero-sum games is that of an $\epsilon$-Nash equilibrium. Intuitively, at an $\epsilon$-Nash equilibrium, the minimization and maximization players pick strategies $x \in \Delta^n, y \in \Delta^m$ such that neither player can unilaterally modify their strategy to improve their objective by more than $\epsilon$. 

In this paper, we consider the problem of computing an $\epsilon$-Nash equilibrium of a zero-sum game in a standard, simple setting where the players have restricted access to the payoff matrix $A$ through a \emph{matrix-vector oracle} that, when queried at $x \in \Delta^n$ and $y \in \Delta^m$, outputs $(Ax, A^\top y)$ (Definition~\ref{def:mat-vec-oracle}). We call these queries \emph{matrix-vector queries}; implementing a query is equivalent to fixing a pair of strategies and computing the expected payoff of each action when the other player's strategy is fixed.

This setting is extensively studied. Since a matrix-vector oracle for $A$ yields access to the gradients of $\bilinear{A}(x,y) \defeq y^\top A x$ with respect to $x$ and $y$, a variety of gradient-based optimization methods apply in this setting. The well-known multiplicative weight update (MWU) method computes an $\epsilon$-Nash equilibrium with $\tilde{O}(\epsilon^{-2})$ queries \citep{mwis, freund1999adaptive, littlestone1994weighted}.\footnote{Throughout, we use $\tilde{O}(\cdot)$ and $\tilde{\Omega}(\cdot)$ to hide polylogarithmic factors in $\epsilon^{-1}$, $m$, and $n$.} In 2004, \citet{nem04} introduced the mirror prox method, which (when applied with the negative entropy function as the regularizer) solves the problem in $\tilde{O}(\epsilon^{-1})$ matrix-vector queries. Later, \citet{omdrakhlin2013optimization} showed optimistic mirror descent also solves the problem with $\tilde{O}(\epsilon^{-1})$ queries. Moreover, as Nesterov observed in 2003 (although the journal version \cite{nesterov2005smooth} was published two years later), this problem is a prototypical example of the power of acceleration: to obtain an $\epsilon$-approximately optimal strategy for the $x$-player,\footnote{Namely, $x \in \simplex^n$ such that $\max_{y \in \simplex^m} y^\top A x - \max_{y \in \simplex^m} y^\top A x' \le \epsilon$ for all $x' \in \simplex^n$.} one can replace the $\max$ in \eqref{eq:zero-sum-game} with an $\epsilon'$-\emph{softmax}, denoted $\smax_{\epsilon'}$ (see \eqref{eq:smax} for further discussion), for $\epsilon'^{-1} = \tilde{O}(\epsilon^{-1})$ and minimize the resulting objective to $\Omega(\epsilon)$-accuracy over $\simplex^n$. An $\epsilon$-approximately optimal strategy for the $y$-player can be obtained analogously, thereby yielding a $2 \epsilon$-Nash equilibrium. As $\smax_{\epsilon'}$ is $\tilde{O}(\epsilon^{-1})$-smooth, accelerated gradient descent also obtains an $\epsilon$-Nash equilibrium in $\tilde{O}(\epsilon^{-1})$ queries \cite{nesterov2005smooth}. This $\tilde{O}(\epsilon^{-1})$-matrix-vector query complexity seen in mirror prox, optimistic mirror descent, as well as accelerated gradient descent, raises a fundamental question: 
\begin{center}
    \emph{Is \emph{$\tilde{O}(\epsilon^{-1})$} queries for computing an $\epsilon$-Nash equilibrium asymptotically optimal?}
\end{center}

Given the well-studied nature of zero-sum games across optimization theory and algorithmic game theory and the lack of improvements in the past twenty years, query complexity improvements would perhaps be surprising. Indeed, related results indicate potential challenges toward improving. First, \citet{ daskalakis2011near} essentially showed that any \emph{no regret} online learning algorithm for zero-sum games must have regret that scales as $\Omega(1/T)$ in the number of rounds (see Section~\ref{sec:related-work} for further discussion). Hence, if one approaches the problem of computing an $\epsilon$-Nash equilibrium of a zero-sum game by designing an online algorithm for minimizing the regret of a zero-sum game, it is not possible to achieve a query complexity better than $O(\epsilon^{-1})$ with this approach. Second, if one takes the aforementioned approach of reducing the problem to minimizing the $\tilde{O}(\epsilon^{-1})$-smooth (in the $\ell_1$-norm) $\smax_{\epsilon'}$ function to $O (\epsilon)$-accuracy, then it is known that $\tilde{\Omega}(\epsilon^{-1})$-gradient queries are required to minimize a $\tilde{O}(\epsilon^{-1})$-smooth function in the $\ell_1$-norm to $O(\epsilon)$-accuracy over the simplex in general \cite{guzman2015lower}.\footnote{\citet{guzman2015lower} prove this lower bound 
for optimization over the $\ell_1$-ball as opposed to the probability simplex, but a standard reduction can be used to translate their lower bound to the simplex.} Finally, it is known that $\Omega(\epsilon^{-1})$ queries are necessary in the worst case to approximately solve Lipschitz-continuous variational inequalities, a problem which generalizes \eqref{eq:zero-sum-game} \cite{ouyang2021lowercomplexitybilinear,zhang2022loweriterationcomplexityconvexconcave,lin2025perseus,adil2022optimalmethodshigherordersmooth}. All of these potential obstacles suggest improvements may be challenging.

Naturally, given the fundamental nature of the problem and the lack of improved methods or tight lower bounds, information-theoretic lower bounds for zero-sum games were recently studied in \cite{kornowski2024oracle}. However, \citet{kornowski2024oracle} proves a lower bound of only $\tilde{\Omega}(\epsilon^{-2/5})$-matrix-vector queries against deterministic algorithms. Thus, despite the barriers in existing algorithms for the problem, the gap between upper and lower bounds points to the exciting possibility that algorithmic improvements may be possible.

The central result of this paper is that indeed, improved query complexities are obtainable! We present a \emph{deterministic}, $\tilde{O}(\epsilon^{-8/9})$-matrix-vector query method that computes an $\epsilon$-Nash equilibrium. This gives the first query complexity improvement to the problem since \cite{nem04,nesterov2005smooth} (even among randomized algorithms). We obtain this result through a broader framework which yields algorithmic insights and improvements for a wider range of problems, known as \emph{matrix-vector games}, as we discuss in the next section. 

\subsection{Our results}\label{sec:matrix-vector-games}

We obtain our result through a more general framework that has broader implications. Generalizing the problem of computing an $\epsilon$-Nash equilibrium of zero-sum games, we consider the problem of \emph{computing an $\epsilon$-solution of a matrix-vector game}, which is a minimax optimization problem of the form $\min_{x \in \cX} \max_{y \in \cY} \bilinear{A}(x,y)$, where $\cX \subset \R^n, \cY \subset \R^m$ are either the probability simplex or the unit Euclidean ball, $A \in \R^{m \times n}$ is an (appropriately normalized) matrix, and $\bilinear{A}: \cX \times \cY \to \R$ denotes the bilinear form in $A$ where $\bilinear{A}(x,y) \defeq y^\top A x$. We define an \emph{$\epsilon$-solution} of this problem as follows:

\begin{definition}[$\epsilon$-solution and gap function]\label{def:epsilon-solution} Let $\epsilon \geq 0$ and $f: \cX \times \cY \to \R$ be convex over $\cX \subset \R^n$ and concave over $\cY \subset \R^m$. We say $z = (x, y) \in \cX \times \cY$ is an \emph{$\epsilon$-solution} of $\min_{x \in \cX} \max_{y \in \cY} f(x,y)$ if it is an \emph{$\epsilon$-saddle point}, i.e.,
    \begin{align*}
        \gap(z) \defeq \max_{y' \in \cY} f(x, y') - \min_{x' \in \cX} f(x', y) \leq \epsilon.
    \end{align*}
We say $z$ is an \emph{exact solution} if it is a $0$-solution. 
\end{definition}

We assume throughout that we have restricted access to $A$ via a matrix-vector oracle, defined as follows. 

\begin{definition}[Matrix-vector oracle]\label{def:mat-vec-oracle} A \emph{matrix-vector oracle for $A \in \R^{m \times n}$}, when queried at $(x, y) \in \R^n \times \R^m$, returns $(Ax, A^\top y) \in \R^{m} \times \R^n$. 
\end{definition}

Next, we go over the specific matrix-vector games setups we consider and our results for each setup.

\paragraph{$\ell_1$-$\ell_1$ games (zero-sum games).} In the problem of computing an $\epsilon$-solution of an $\ell_1$-$\ell_1$ game, we are given a matrix-vector oracle for $A \in \R^{m\times n}$ with $\normInline{A}_{\max} \defeq \max_{i,j} |A_{ij}| \leq 1$ and $\epsilon > 0$, and must compute an $\epsilon$-solution of $\min_{x \in \Delta^n} \max_{y \in \Delta^m} \bilinear{A}(x,y)$. This is precisely the problem of computing an $\epsilon$-Nash equilibrium of a zero-sum game. We prove the following main result for this problem. 

\begin{restatable}{theorem}{zerosummain}\label{intro:l1l1} There is a \emph{deterministic} algorithm that computes an $\epsilon$-solution of an $\ell_1$-$\ell_1$ game (i.e., a zero-sum game) using $\tilde{O}(\epsilon^{-8/9})$-matrix-vector queries. 
\end{restatable}

\paragraph{$\ell_2$-$\ell_1$ games.} Letting $\B^k$ denote the $k$-dimensional Euclidean unit ball, in the problem of computing $\epsilon$-solution of an $\ell_2$-$\ell_1$ game, we are given a matrix-vector oracle for $A \in \R^{m\times n}$ with $\normInline{A}_{2\rightarrow \infty} \defeq \smash{\max_{i \in [m]} (\sum_{j \in [n]} A_{ij}^2)^{1/2} \leq 1}$ and $\epsilon > 0$, and must compute an $\epsilon$-solution of $\min_{x \in \B^n} \max_{y \in \Delta^m} \bilinear{A}(x,y)$. $\ell_2$-$\ell_1$ games arise, for instance, in linear separability problems and support vector machines \cite{soheili2012smoothperceptron,yu2014saddlepointsacceleratedperceptron,kornowski2024oracle}. We prove the following result.

\begin{restatable}{theorem}{svm}\label{intro:l2l1} There is a \emph{deterministic} algorithm that computes an $\epsilon$-solution of an $\ell_2$-$\ell_1$ game using $\tilde{O}(\epsilon^{-7/9})$-matrix-vector queries.  
\end{restatable}

For both $\ell_2$-$\ell_1$ and $\ell_1$-$\ell_1$ games, no previous algorithm (neither deterministic nor randomized) was known to achieve a rate better than the $\tilde{O}(\epsilon^{-1})$-matrix-vector query rate of mirror prox \cite{nem04}.

\paragraph{$\ell_2$-$\ell_2$ composite games:} Finally, we consider the setting where $\cX$ and $\cY$ denote the $n$- and $m$-dimensional Euclidean unit balls \cite{carmon2019variance, carmon2020coordinate, grigoriadis1995sublinear}. However, note that an exact solution to $ \min_{x \in \B^n} \max_{y \in \B^m} y^\top A x $ is trivially $(0, 0) \in \R^n \times \R^m$. Thus, the problem has been previously studied in the \emph{composite} setting, where the minimax objective is a bilinear function plus a (possibly nonlinear) \emph{composite function} $\phi: \B^n \times \B^m \to \R$ \cite{carmon2019variance, he2015mirror, carmon2020coordinate, woodworth2016tight, duchi2010composite}, i.e.,
\begin{align*}
    \min_{x \in \B^n}\max_{y \in \B^m} \bilinear{A}(x,y) + \phi(x,y).
\end{align*}
The composite function $\phi$ can be any differentiable \emph{convex-concave function}, where we say that a function $h: \cX \times \cY \to \R$ is convex-concave if $\cX \subset \R^n, \cY \subset \R^m$ and the restrictions of $h$ to the first $n$ and last $m$ inputs are convex and concave respectively. In applications of $\ell_2$-$\ell_2$ composite games, such as regression problems, the composite function is usually an appropriately selected regularizer and is explicitly known.

A composite variant of the mirror prox algorithm \cite{nem04} obtains a matrix-vector query complexity of $O(\normInline{A}_2 \cdot \epsilon^{-1})$ for $\ell_2$-$\ell_2$ composite games where $\normInline{A}_2$ denotes the spectral norm, which is known to be the near-optimal dependence on $\normInline{A}_2$ and $\epsilon$ (even for randomized algorithms) \cite{braverman2020gradient}.\footnote{The lower bounds in \citet{liu2023accelerated} and \cite{braverman2020gradient} are stated for approximate $\ell_2$-regression; however, there is a known reduction (see e.g., \cite{carmon2019variance}) to reduce approximate $\ell_2$-regression to computing an $\epsilon$-solution of an $\ell_2$-$\ell_2$ composite game.} However, it is perhaps natural to ask (see, e.g., \cite{liu2023accelerated}) whether better trade-offs can be obtained when $A$ satisfies better spectral decay properties, i.e., is approximately low rank. Indeed, we prove the following general result, which yields rates parameterized by the Schatten-$p$ norm of $A$ for any $p \in [1, \infty)$.

\begin{restatable}{theorem}{mainresultthree}\label{intro:l2l2} For any $p \in [1, \infty)$, there is a \emph{deterministic} algorithm that computes an $\epsilon$-solution of an $\ell_2$-$\ell_2$ composite game using $O( \normInline{A}_{\cS_p}^{p/(1+p)} \epsilon^{-p/(1 + p)})$-matrix-vector queries.  
\end{restatable}

When $p = 2$, Theorem~\ref{intro:l2l2} implies a \emph{deterministic} ${O}(\normInline{A}_F^{2/3} \epsilon^{-2/3})$-matrix-vector query algorithm. This rate is the near-optimal dependence on $\normInline{A}_F$ and $\epsilon$ in light of lower bounds of \cite{liu2023accelerated}. Previously, this rate was achieved using a \emph{randomized} method in \cite{jin2025reusingsamplesvariancereduction}.\footnote{We also note that this rate may be well-known for special cases of $\ell_2$-$\ell_2$ composite games such as regression (see, e.g., \cite{liu2023accelerated, spielman2009note}), however, our results hold for $\ell_2$-$\ell_2$ composite games more broadly.} For $p \neq 2$, Theorem~\ref{intro:l2l2} indicates rates that depend on the decay properties of the singular values of $A$. We also obtain related instance-dependent rates for $\ell_2$-$\ell_1$ and $\ell_1$-$\ell_1$ games; however, we defer a discussion of this to Section~\ref{sec:elltwoelloneandelloneellone} (Theorem~\ref{thm:general-norm-result}). 

\subsection{Additional related work}\label{sec:related-work} 

\paragraph{Lower bounds in the matrix-vector oracle model.}{ As mentioned, for $\ell_1$-$\ell_1$ games (zero-sum games) \citet{kornowski2024oracle} recently showed a lower bound of $\tilde{\Omega}(\epsilon^{-2/5})$ queries. In addition, for $\ell_2$-$\ell_1$ games, they showed a lower bound of $\tilde{\Omega}(\epsilon^{-2/3})$ queries. Their lower bounds hold for all deterministic algorithms. For $\ellOneEllOne$-games in particular, their work improved upon that of \cite{hadiji2024towards}, which showed a $\Omega(\log (1 / (n \epsilon)))$ lower bound for sufficiently small $\epsilon = \text{poly}(1/n)$ and $m = n$. For $\ell_2$-$\ell_2$ composite games, \citet{liu2023accelerated} showed a lower bound of $\tilde{\Omega} (\epsilon^{-2/3})$ against both deterministic and randomized algorithms when the payoff matrix $A$ satisfies $\norm{A}_F \leq 1$, and \citet{braverman2020gradient} showed a lower bound of $\tilde{\Omega} (\epsilon^{-1})$ against both deterministic and randomized algorithms when the payoff matrix $A$ satisfies $\norm{A}_2 \leq 1$.
}

\paragraph{Other oracle models.}{ Aside from the matrix-vector oracle model, other more restrictive models have also been studied in the optimization and algorithmic game theory literature. Some prior works \citep{carmon2019variance, grigoriadis1995sublinear, clarkson2012sublinear, palaniappan2016stochastic} consider a more restrictive setting where payoff matrix access is restricted to row/column queries. \citet{carmon2020coordinate} consider the case where payoff matrix access is restricted to entry queries. \citet{omdrakhlin2013optimization} and \citet{daskalakis2011near} consider a ``strongly-uncoupled'' online setting in which two players play a sequence of strategies in rounds, $\{(x^1, y^1), (x^2, y^2), \ldots\}$, and can only query a matrix-vector oracle for the strategies that were played in the history of the game; the goal is to minimize the regret of the played strategies. Hence, the strongly-uncoupled setting is more restrictive than the setting which we study in this paper. In this strongly-uncoupled setting, \citet{daskalakis2011near} showed a regret lower bound of $\Omega(1/T)$, where $T$ is the number of rounds of the game, which matches the rate achieved by optimistic mirror descent \cite{omdrakhlin2013optimization} up to polylogarithmic factors.
}

\subsection{Directions for future work}\label{sec:futurework}

The obvious problem left open by our work is that of closing the gap between our upper bounds (Theorems~\ref{intro:l1l1} and~\ref{intro:l2l1}) and the lower bounds of \citet{kornowski2024oracle}. In this paper, we prove that improvements over the $\tilde{O}(\epsilon^{-1})$-matrix-vector queries achieved by mirror prox is possible. The next question is to study the optimal information-theoretic complexity. We hope our work enables further research on this problem.

We note that our work focuses on studying the information-theoretic complexity of $\ell_1$-$\ell_1$ and $\ell_2$-$\ell_1$ games. We do not study the problem of obtaining algorithms with better runtime or improved parallelizability, and we pose these as interesting problems for future work.

Lastly,  our improved bounds on the number of matrix vector queries for solving matrix-vector games as well as our further improvements based on different bounds on the intput matrix $A$, indicate the potential to develop algorithms that leverage alternative structure to obtain better $\epsilon$-dependence for broader classes minimax optimization problems. This may be another avenue for future work. 

\subsection{Overview of approach}\label{sec:overview-of-approach}

In this section, we outline our algorithmic and analytical techniques. The starting point for our algorithms is the classic mirror prox algorithm \citep{nem04}, which we describe below.

\paragraph{The mirror prox algorithm.} Let $f, \phi$ be any differentiable, convex-concave functions over $\cX \times \cY$ where $\cX \subset \R^n$ and $\cY \subset \R^m$ are convex and compact. Mirror prox is a natural extension of mirror descent \cite{nemirovskij1983problem} for computing an $\epsilon$-solution of the minimax problem $\min_{x \in \cX} \max_{y \in \cY} f(x,y) + \phi(x,y)$ 
 given \emph{gradient oracle} access to $f$ (which maps $(x, y) \mapsto \nabla f(x,y)$) and explicit access to $\phi$. 

To describe the algorithm, we briefly introduce some notation. In the remainder of the paper, we use $\cZ$ as a shorthand for $\cX \times \cY$ and let $d \defeq m + n$. Further, for any $z \in \cZ$, we use $z\x \in \cX$ and $z\y \in \cY$ to denote the first $n$ and last $m$ coordinates of $z$ respectively. We denote the \emph{gradient mapping} of a differentiable function $h : \cZ \to \R$ by $\nabla_\pm h$. That is, $\nabla_\pm h: \cZ \to \R^d$ is given by $\nabla_\pm h(z) \defeq (\nabla\x h(z), - \nabla\y h(z))$ where $\nabla\x h$ and $\nabla\y h$ denote the gradient of $h$ with respect to the first $n$ and last $m$ coordinates, respectively. 

Now, mirror prox can be described for general geometries $\cZ$ via \emph{distance-generating function (dgf) setups}, defined as follows. The following definition is inspired by \cite{carmon2019variance}. 

\begin{definition}[dgf setups]\label{def:dgf-setup}
We say $(\zset, \normInline{\cdot}, r)$ is a \emph{dgf setup} if: (i) $\zset \subset \R^d$ is compact and convex; and (ii) $r : \zset \to \R$, referred to as the \emph{distance-generating function (dgf)}, is differentiable and 1-strongly convex over $\zset$ with respect to the norm $\normInline{\cdot} : \R^d \to \R$. For any $z, z' \in \zset$, $\breg{z}{z'} \defeq r(z') - r(z) - \grad r(z)^\top (z' - z)$ denotes the \emph{Bregman divergence} induced by the dgf $r$, and $\Range \defeq \max_{z, z' \in \cZ} r(z) - r(z')$ is the range of $r$. 
\end{definition}

Given a dgf setup $(\cZ, \normInline{\cdot}, r)$, mirror prox first initializes $z^0 \gets \argmin_{z \in \cZ} r(z)$. Then, at each iteration $t$, the algorithm performs a \emph{mirror prox step}, which consists of two \emph{proximal steps} from $z^t$ with a parameter $\lambda > 0$, as described below (see Definition~\ref{def:proximal-mappings} for the $\prox$ operator definition).
\begin{align}\label{eq:update}
    \emph{\text{Mirror prox step:~}} w^t \gets \prox_{z^t}^{\lambda}(\nabla_\pm f(z^t) + \nabla_\pm \phi), ~~~\text{followed by}~~~ z^{t+1} \gets \prox_{z^t}^{\lambda}(\nabla_\pm (f+\phi)(w^t)). 
\end{align} 
One can think of $1/\lambda$ as the step size of the proximal steps, analogously to step sizes in mirror descent or gradient descent. Note that each mirror prox step can be implemented with $O(1)$-gradient-oracle queries to $f$ (assuming that $\phi$, and hence $\nabla \phi$, are explicitly known).

Now, suppose the gradient mapping $\gm f$ is $L$-Lipschitz with respect to $\normInline{\cdot}$, i.e., that $\normInline{ \nabla_\pm f(z) - \nabla_\pm f(z')}_* \leq L \normInline{z - z'}$ for all $z, z' \in \zset$, where $\normInline{\cdot}_* : 
z \mapsto \sup_{z': \normInline{z'} \leq 1} z^\top z'$ denotes the \emph{dual norm} of $\normInline{\cdot}$. Then, one can show that with step size $1 / \lambda = 1 / L$, the average of the iterates $w^t$ is an $\epsilon$-solution of $\min_{x \in \cX} \max_{y \in \cY} f(x,y) 
+ \phi(x,y)$ after $O(\Range L/\epsilon)$-mirror prox steps, implying a total of $O(\Range L/\epsilon)$ queries to a gradient oracle for $f$.

In the special case of matrix-games, the objective function of interest is $f = \bilinear{A}$, and a matrix-vector oracle yields a gradient oracle for the objective function $\bilinear{A}$. Indeed, when $A$ is normalized appropriately ($\normInline{A}_{\max} \leq 1$ for $\ell_1$-$\ell_1$ games, 
$\normInline{A}_{2\to\infty} \leq 1$ for $\ell_2$-$\ell_1$ games, and $\normInline{A}_{2} \leq 1$ for $\ell_2$-$\ell_2$ composite games), it is known how to configure the dgf setup $(\cZ, r, \normInline{\cdot})$ so that $L = 1$ and $\Range = \tilde{O}(1)$ (e.g., \cite{nem04, carmon2019variance}). Thus, mirror prox yields a deterministic $\tilde{O}(\epsilon^{-1})$-matrix-vector query algorithm for each of the matrix-vector games described in Section~\ref{sec:matrix-vector-games}. 

In fact, as we will describe in more detail below, mirror prox is known to attain the ${O}(\Range L/\epsilon)$-matrix-vector query complexity even if $\gm f$ is not $L$-Lipschitz so long as the mirror prox triples $(z^t, w^t, z^{t+1})$ in each iteration $t$ of \eqref{eq:update} satisfy a weaker condition called \emph{$L$-local relative Lipschitzness} with respect to $\gm f_A$ and $(\cZ, \normInline{\cdot}, r)$ \cite[Corollary 1]{cohen2020relative}. The first key idea in our approach is to carefully analyze---and capitalize on---conditions where the triple $(z^t, w^t, z^{t+1})$ does not satisfy $L$-relative Lipschitzness with respect to $\gm f_A$ and $(\cZ, \normInline{\cdot}, r)$ but still reveals useful information about the payoff matrix $A$. 

\paragraph{Smooth-until-proven-guilty mirror prox.} To illustrate this first idea, we restrict, for now, to the case of $\ell_2$-$\ell_2$ composite games, and sketch our approach for proving Theorem~\ref{intro:l2l2}. To this end, let $\cX = \B^n, \cY = \B^m$, and consider the problem of computing an $\epsilon$-solution of an $\ell_2$-$\ell_2$ composite game for $A \in \R^{m \times n}$ and a differentiable convex-concave function $\phi: \cZ \to \R$. For simplicitly, here we assume $\normInline{A}_F \leq 1$ and describe the techniques for proving Theorem~\ref{intro:l2l2} for $p = 2$. (The idea for other Schatten-$p$ norms is similar.)

Consider the dgf setup (Definition~\ref{def:dgf-setup}) $(\cZ, r, \normInline{\cdot}_2)$, where $r(z) \defeq \frac{1}{2} \norm{z}_2^2$. Since $\normInline{A}_F \leq 1$, $\gm f_A$ is $1$-Lipschitz with respect to $\normInline{\cdot}_2$ (the Lipschitz constant of $\gm f_A$ under $\norm{\cdot}_2$ is $\innorm{A}_2$, and of course $\innorm{A}_2 \le \innorm{A}_F$), and hence mirror prox computes an $\epsilon$-solution of the $\ell_2$-$\ell_2$ composite game in $O(\epsilon^{-1})$-matrix-vector queries by setting $\lambda = L = 1$ and performing $O(\epsilon^{-1})$ mirror prox steps \eqref{eq:update}.

To improve this rate, one hope might be to use the fact that $\normInline{A}_F \leq 1$ to show that one can simply set a larger fixed step size than $1 / L$ in the mirror prox steps \eqref{eq:update} and still converge to an $\epsilon$-solution in fewer iterations, and hence, fewer matrix-vector queries. However, unfortunately in the worst case $\normInline{A}_2 = \normInline{A}_F$, and hence the Lipschitz constant of $\gm f_A$ could be $\Omega (1)$. As with most gradient-based methods, it is difficult to prove faster convergence with step size larger than the inverse of the Lipschitz constant of the gradient.

Nonetheless, we prove that a slight variant of the standard mirror prox algorithm converges in fewer iterations and oracle queries by using prox steps with a fixed step size larger than $1/L$. Intuitively, our algorithm leverages the fact that if $\normInline{A}_F$ is on the order of $\normInline{A}_2$, then $A$ must be essentially low rank. 

To illustrate the idea, let $L' < L$ and consider running mirror prox with the more aggressive step size $1/\lambda = 1/L' > 1/L$ in each mirror prox step  \eqref{eq:update}. \citet{cohen2020relative} showed that at $t$-th iteration, mirror prox makes progress in solving the $\ell_2$-$\ell_2$ composite game so long as the following $L'$-relative-Lipschitzness condition is satisfied:
\begin{align}\label{eq:okay-condition}
    (\nabla_\pm \bilinear{A}(w^t)-\nabla_\pm \bilinear{A}(z^t))^\top(w^t-z^{t+1}) \leq L'(\breg{w^t}{z^{t+1}} + \breg{z^t}{w^t}). 
\end{align}
Steps where \eqref{eq:okay-condition} holds make progress in converging to an $\epsilon$-solution of the $\ell_2$-$\ell_2$ composite game, and we call such iterations \emph{progress iterations.} On the other hand, suppose \eqref{eq:okay-condition} does not hold at some iteration $t$, i.e., 
\begin{align}\label{eq:violate}
    (\nabla_\pm \bilinear{A}(w^t)-\nabla_\pm \bilinear{A}(z^t))^\top(w^t-z^{t+1}) > L'(\breg{w^t}{z^{t+1}} + \breg{z^t}{w^t}). 
\end{align}
Now, note that $\nabla_\pm \bilinear{A}$ is given by $(x,y) \mapsto (A^\top y, -Ax)$. Substituting this observation into \eqref{eq:violate}, we find
\begin{align}\label{eq:sug}
   & ({w^{t}\y} - {z^{t}\y})^\top A ({w^{t}\x} - {z^{t+1}\x}) + ({w^{t}}\y - {z^{t+1}\y})^\top A ({z^{t}\x} - {w^{t}\x})  > L' [V^r_{w^{t}}(z^{t+1}) - V^r_{z^t}({w^{t}})] \notag \\
   =& (L' / 2) \paren{\normInline{w^t\x - z^{t+1}\x}_2^2 + \normInline{w^t\y - z^{t+1}\y}_2^2 + \normInline{w^t\x - z^{t}\x}_2^2 + \normInline{w^t\y - z^{t}\y}_2^2} \\
   \geq& L' \paren{ \normInline{w^t\y- z^{t}\y}_2  \normInline{w^t\x - z^{t+1}\x}_2 + \normInline{w^t\y- z^{t+1}\y}_2  \normInline{w^t\x - z^{t}\x}_2  }. \notag
\end{align}
For the equality on the second line, we use the fact that the Bregman divergence induced by the specified dgf $r$ is simply $V^r_{a}(b) = \frac{1}{2}\norm{a-b}_2^2$ for all $a, b \in \cZ$. For the third inequality, we use that for $a, b \in \R$, we have $a^2 + b^2 \geq 2ab$. Rearranging the above, we see that we must have that either
\begin{align}\label{eq:model-update-condition}
\frac{({w^{t}\y} - {z^{t}\y})^\top A ({w^{t}\x} - {z^{t+1}\x})}{ \normInline{w^t\y- z^{t}\y}_2  \normInline{w^t\x - z^{t+1}\x}_2 } \geq L' \text{~~~or~~~} \frac{({w^{t}\y} - {z^{t+1}\y})^\top A ({z^{t}\x} - {w^{t}\x})}{\normInline{w^t\y- z^{t+1}\y}_2  \normInline{w^t\x - z^{t}\x}_2} \geq L'. 
\end{align}
Note that \eqref{eq:model-update-condition} ensures that the triple $(z^t, w^t, z^{t+1})$ is guaranteed to reveal a large component of the matrix $A$! Although this is not directly valuable for making progress in computing a solution to $f_A + \phi$, it \emph{is} valuable in learning the payoff matrix $A$!

To capitalize on the information revealed in these steps where \eqref{eq:violate} does not hold, we modify such iterations to implement what we call \emph{model-update steps}: we use $O(1)$ matrix-vector queries to detect when \eqref{eq:model-update-condition} holds and when it does, use an additional $O(1)$ matrix-vector queries to project off and store a large component of the matrix $A$. (We do this by subtracting from the matrix $A$ in the objective and then adding to an extra composite term---see Algorithm~\ref{alg:mirror-prox-sug-l2l2}.)

This projection reduces the Frobenius norm of $A$ by an amount depending on $L'$, and therefore we can bound the total number of model-update steps. Said differently, in the remainder of the algorithm, we no longer need to query the matrix-vector oracle over the subspace spanned by the stored directions, and thus, the total number of model-update steps can be bounded. By carefully selecting $L'=\epsilon^{1/3}$ and setting the step size of our prox steps to $1/\lambda = 1/L'$, we optimally trade-off the number of progress steps with the number of model-update steps, thereby obtaining a query complexity of $\tilde{O}(\epsilon^{-2/3})$.

To summarize: in \emph{progress iterations}, our algorithm makes large progress by treating the function $\bilinear{A}$ as smoother than it might actually be, but performs a \emph{model-update iteration} every time it discovers a large component certifying otherwise.  Inspired by related work in nonconvex optimization \cite{carmon2017convex}, we term this principle ``smooth-until-proven-guilty'' and call our method ``smooth-until-proven-guilty mirror prox.'' We note that alternatively, one could instead try to find model-update directions using techniques from randomized numerical linear algebra (see e.g., \cite{musco2015randomized, bakshi2022low} and references therein); however, the key advantage of our approach is that it directly enables deterministic query complexities.

\paragraph{Extending to $\ell_1$-$\ell_1$ and $\ell_2$-$\ell_1$ games.} Unfortunately, the technique for $\ell_2$-$\ell_2$  composite games is tailored to the $\ell_2$-$\ell_2$ setting, and in particular the choice of regularizer $r(z) = \frac{1}{2} \normInline{z}_2^2$ is crucial to the simplification in Line~\ref{eq:sug}. Hence, extending the above approach to $\ell_1$-$\ell_1$ and $\ell_2$-$\ell_1$ games requires additional techniques. 

Said differently, one way of interpreting our result for $\ell_2$-$\ell_2$ composite games is the following: The smoothness of $\ell_2$-$\ell_2$ composite games depends on the operator norm $\normInline{A}_{2}$. If we further know that the Frobenius norm is bounded $\normInline{A}_F \leq 1$, then even though $\normInline{A}_{2}$ could be as large as 1 in the worst-case, we can still bound the number of low-rank projections necessary in order to reduce the operator norm. Once the operator norm has been sufficiently reduced, the problem is smoother, and mirror prox is guaranteed to converge in fewer iterations. Concretely, our method for $\ell_2$-$\ell_2$ composite games performs $O(\normInline{A}_F^2/L'^2)$-model-update iterations and at most $O(L'/\epsilon)$-progress iterations to achieve an optimal trade-off for $L' = \normInline{A}_F^{2/3} \epsilon^{1/3}$. This is consistent with trade-offs obtained in  numerical linear algebra settings, where one can compute low-rank projections to improve the smoothness of a matrix (see, e.g., \cite{spielman2009note, musco2015randomized}).

However, in non-Euclidean geometries, such as those arising in $\ell_1$-$\ell_1$ and $\ell_2$-$\ell_1$ games, the smoothness of the problem depends on $\normInline{A}_{\max}$ (the max entry) and $\normInline{A}_{2\to\infty}$ (the max row $\ell_2$-norm) respectively, and it is less clear how to leverage low-rank projections to identify low-rank updates which would significantly reduce these norms and enable mirror prox to converge in fewer iterations. Hence, we cannot directly extend our smooth-until-proven-guilty idea from above to non-Euclidean geometries, because it is not clear that mirror prox iterates which violate the relative Lipschitzness condition will directly allow us to lower the smoothness of problems in these non-Euclidean geometries.

Therefore, we take an alternative approach where we reduce to solving a series of subproblems which more closely resemble regularized $\ellTwoEllTwo$ games. Specifically, letting $\yset = \simplex^m$, as well as $\cX = \B^n$ for $\ell_2$-$\ell_1$ games and $\cX = \Delta^n$ for $\ell_1$-$\ell_1$ games, a standard optimization framework known as the \emph{proximal point method} solves the original matrix game $\min_{x \in \xset} \max_{y \in \yset} f_A(x, y)$ by iteratively solving subproblems of the form 
\begin{align}\label{eq:reduction-subproblem}
    \min_{x \in \xset} \max_{y \in \yset} \bilinear{A}(x,y) + \alpha \xbreg{{z_c}\x}{x} - \alpha \ybreg{{z_c}\y}{y} 
\end{align}
for $\alpha > 0$ and center point $\zcenter \in \zset$ (both of which may vary from subproblem to subproblem), and where $\xbreg{{z_c}\x}{\cdot}$ and $\ybreg{{z_c}\y}{\cdot}$ are appropriate Bregman divergences on $\xset$ and $\yset$ respectively (see Definition~\ref{def:matrix-vector-games-setups} for further details). Letting $\zopt_\alpha$ denote the exact solution to \eqref{eq:reduction-subproblem}, our key starting observation is that if $\breg{\zcenter}{\zopt_\alpha} = O(\alpha^2)$, then the subproblem \eqref{eq:reduction-subproblem} is \emph{stable} in the following sense: It can be shown (see Lemma~\ref{lem:stability-wrt-best-response}) that a linearized proximal step from $\zcenter$, formally $\ztilde = \prox^\alpha_{\zcenter}(\gm f_A(\zcenter))$ (see Definition~\ref{def:proximal-mappings}) or equivalently $\ztilde = (\xtilde, \ytilde)$ where
\begin{align}\label{eq:ztilde}
    \tilde{x} \gets \min_{x \in \cX} {{z_c}\y}^\top A x + \alpha \xbreg{{z_c}\x}{x} \text{~~and~~} \tilde{y} \gets \max_{y \in \cY} y^\top A {{z_c}\x} - \alpha \ybreg{{z_c}\y}{y},
\end{align}
is such that $\ztilde$ and $\zopt_\alpha$ are \emph{multiplicatively close} in all of their coordinates which lie in a probability simplex. Therefore, to solve the subproblem \eqref{eq:reduction-subproblem}, it suffices to restrict to a \emph{multiplicative ball} (see Definition~\ref{def:stable-ball} for a formal treatment) around $\ztilde$. Furthermore, within this multiplicative ball the Hessian of the KL divergence is stable up to a constant multiplicative factor, in the sense that the KL divergence is approximated by a rescaled Euclidean norm (see Lemma~\ref{corr:stable-balls}) which, as is standard in the literature for referring to this phenomenon (e.g., \cite{carmon2019variance,clarkson2012sublinear,shwartz2012onlinelearning,anuran2015studyoflocalapproximationsininfotheory}), we refer to as a \emph{local norm}. Upon replacing the KL divergence terms in \eqref{eq:reduction-subproblem} with these local norms, \eqref{eq:reduction-subproblem} becomes a regularized $\ellTwoEllTwo$ game in an appropriate change of basis/rescaling of the coordinates. 

We show in Section~\ref{sec:innerloops} that our smooth-until-proven-guilty approach can be extended to these regularized and rescaled $\ellTwoEllTwo$ games to obtain a $\Otilde(\alpha^{-2/3})$ matrix-vector query bound for solving the subproblem \eqref{eq:reduction-subproblem} to high accuracy.\footnote{This $\Otilde(\alpha^{-2/3})$ matrix-vector query bound comes from optimizing over $\tau$ in the $\Otilde(\tau \alpha^{-1} + \zeta^p \tau^{-p})$ matrix-vector query bound in Corollary~\ref{cor:SM-mirror-prox-complexity-when-passmodel-is-false} with $p \gets 2$, noting $\zeta \gets 1$ is a valid choice in this case due to Remark~\ref{remark:general-conversion}.} This notion of muliplicative stability has been used in other context such as \emph{Hessian stability} and \emph{quasi-self-concordance} from the ball acceleration and trust-region methods literature \cite{carmon2020acceleration,karimireddy2018globallinearconvergencenewtons}.

To summarize so far, we have an ``inner loop'' method given in Section~\ref{sec:innerloops} which can solve subproblems of the form \eqref{eq:reduction-subproblem} to high accuracy with $\Otilde(\alpha^{-2/3})$ matrix-vector queries as long as $\breg{\zcenter}{\zopt_\alpha} = O(\alpha^2)$, but have yet to describe our ``outer loop'' algorithm based on the proximal point method for reducing to a series of such subproblems. The standard proximal point method is given a sequence of positive numbers $\alpha_1, \alpha_2, \dots$ and iterates \eqref{eq:reduction-subproblem} with $\alpha \gets \alpha_k$ at iteration $k$ (i.e., the solution to \eqref{eq:reduction-subproblem} in iteration $k$ becomes the center point $\zcenter$ in iteration $k + 1$). Omitting dependence on the domain size for brevity (which is $\Otilde(1)$ in our matrix games applications), it is straightforward to show that this method achieves error $O (1 / \sum_{k \in [K]} {\alpha_k^{-1}})$ after $K$ iterations. Then, to obtain fast convergence, we should dynamically set $\alpha_k$ as small as possible at each iteration $k$ while still ensuring the condition $\breg{\zcenter}{\zopt_{\alpha_k}} = O(\alpha_k^2)$ is met. Since $\breg{\zcenter}{\zopt_{\alpha_k}}$ monotonically increases as $\alpha_k$ decreases (less regularization leads to more movement; see Lemma~\ref{lem:monotonicity}), we aim for the sweet spot where $\breg{\zcenter}{\zopt_{\alpha_k}} = \Theta(\alpha_k^2)$.

We formalize this intuition of dynamically searching for an $\alpha_k$ which satisfies a movement bound by defining a \emph{dynamic approximate proximal oracle} or $\DAPO$ in Definition~\ref{def:DAPO} in Section~\ref{sec:prox-point-regret}, and then prove an iteration bound for the proximal point method with a $\DAPO$ in Lemma~\ref{lem:prox-point-iteration-bound}. As we apply it to $\ellOneEllOne$ and $\ellTwoEllOne$ matrix games, the $\DAPO$ oracle requires $\alpha_k$ to either satisfy $\breg{\zcenter}{\zopt_{\alpha_k}} \ge \alpha_k^2$ or else $\alpha_k = \beta$ for a tunable parameter $\beta > 0$,\footnote{More generally, Definition~\ref{def:DAPO} requires either $\breg{\zcenter}{\zopt_{\alpha_k}} \ge \alpha_k^c$ or else $\alpha = \beta$ for some $c, \beta > 0$, but in our applications to $\ellOneEllOne$ and $\ellTwoEllOne$ matrix games we always instantiate $c \gets 2$.} and we prove a $\Otilde(\beta / \epsilon + \epsilon^{-2/3})$ iteration bound for the resulting proximal point method to obtain an $\epsilon$-solution to the original matrix game (see Lemma~\ref{lem:prox-point-iteration-bound}).

In the context of $\ellOneEllOne$ and $\ellTwoEllOne$ matrix games, to instantiate the requirements of the $\DAPO$ oracle, namely $\breg{\zcenter}{\zopt_{\alpha_k}} \ge \alpha_k^2$ or else $\alpha_k = \beta$, as well as the $\breg{\zcenter}{\zopt_{\alpha_k}} = O(\alpha_k^2)$ requirement of our subproblem solver for \eqref{eq:reduction-subproblem}, we give a binary search procedure in Section~\ref{subsec:binary-search} which first checks whether $\alpha_k = \beta$ satisfies the $\breg{\zcenter}{\zopt_{\alpha_k}} = O(\alpha_k^2)$ requirement, and otherwise searches for $\alpha_k > \beta$ such that $\breg{\zcenter}{\zopt_{\alpha_k}} = \Theta(\alpha_k^2)$. (In particular, the $\alpha_k = \beta$ condition can be viewed as a ``minimum regularization level'' which we are always willing to use as long as the $\breg{\zcenter}{\zopt_{\alpha_k}} = O(\alpha_k^2)$ requirement is satisfied.) Combining the $\Otilde(\beta / \epsilon + \epsilon^{-2/3})$ iteration bound of the proximal point with the $\Otilde(\alpha_k^{-2/3}) \le \Otilde(\beta^{-2/3})$ complexity of solving the subproblem \eqref{eq:reduction-subproblem} at the $k$-th iteration to high accuracy (since we ensure $\alpha_k \ge \beta$), we obtain an overall matrix-vector query complexity of $\Otilde(\beta^{1/3} \epsilon^{-1} + \beta^{-2/3} \epsilon^{-2/3})$. Optimizing over $\beta$ (namely $\beta \gets \epsilon^{1/3}$) yields $\Otilde(\epsilon^{-8/9})$.

We remark that our proximal point outer loop, including the $\DAPO$ oracle and accompanying binary search procedure, can be viewed as a primal-dual or variational-inequality analog of existing acceleration frameworks including Monteiro-Svaiter acceleration \cite{renato2013monteirosvaiteroriginalpaper,bubeck2019highlysmooth,bubeck2019highlyparallel,carmon2022optimalandadaptivemonteirosvaiter} and ball acceleration \cite{carmon2020acceleration,jambulapati2024closingcomputationalquerygap,carmon2021thinking,hilal2021stochasticbiasreduced,carmon2022distributionallyrobustoptimizationball,carmon2023resqueing,carmon2024whole}. These methods also reduce to subproblems involving optimization in a ball or an approximate proximal step with a careful choice of center point and/or regularization, and we hope our primal-dual analog of these frameworks will enable future work.

Finally, for $\ell_2$-$\ell_1$ games, we obtain a further improvement by demonstrating that information about the matrix $A$ discovered in one subproblem \eqref{eq:reduction-subproblem} can be effectively leveraged in order to solve subsequent subproblems in our proximal point procedure. Performing an  amortized analysis of this method, we obtain an overall query complexity of $\tilde{O}(\epsilon^{-7/9})$-matrix-vector queries for $\ell_2$-$\ell_1$ games (see Section~\ref{sec:four-fifths} for details).

\paragraph{Alternative approaches.} We conclude this introduction by briefly mentioning alternative approaches to obtain query complexities that might improve over the prior state-of-the-art $\tilde{O}(\epsilon^{-1})$-matrix-vector query complexity for $\ell_2$-$\ell_1$ and $\ell_1$-$\ell_1$ games. 

As proof of concept that such improvements could be possible, an illustrative example is the setting of $\ell_2$-$\ell_1$ games. Computing an $\epsilon$-solution to an $\ell_2$-$\ell_1$ game can be reduced \cite{nesterov2005smooth,carmon2021thinking} to computing an $O(\epsilon)$-approximate minimizer of the following softmax objective for $\epsilon' = \tilde{O}(\epsilon)$:
\begin{align}\label{eq:smax}
    \min_{x \in \B^n} \smax_{\epsilon'}(x), \text{~~where~~} \smax_{\epsilon'} \defeq \epsilon' \log\Big(\sum_{i \in [n]} \exp\Big(\frac{[Ax]_i}{\epsilon'}\Big)\Big).
\end{align}
The Hessian of $\smax_{\epsilon'}$ has trace $\tilde{O}(1/\epsilon)$ and recent work of \cite{liu2023accelerated} showed techniques to accelerate optimization of functions whose Hessians satisfy trace-bounds using cubic-regularized Newton. However, the algorithms of \citet{liu2023accelerated} do not directly handle constrained problems such as \eqref{eq:smax} and are fundamentally randomized. Nonetheless, by replacing the ball-constraint in \eqref{eq:smax} with $\ell_2$-regularization, we believe one might be able to apply the techniques of \citet{liu2023accelerated} to obtain \emph{randomized} algorithms for $\ell_2$-$\ell_1$ games with $\tilde{O}(\epsilon^{-1+c})$-query complexity, for some $c > 0$. 

We note that there is another approach one could take by leveraging a recent technique of \emph{ball-accelerated Newton} \cite{carmon2020acceleration}, which improves on cubic-regularized Newton in certain cases where the Hessian satisfies additional structure. In the case of $\ell_2$-$\ell_1$ games, the Hessian of the softmax objective in \eqref{eq:smax} is trace-bounded and locally stable \cite{carmon2020acceleration, carmon2024whole, carmon2023resqueing}. Hence, one could hope that quadratic optimization techniques \cite{liu2023accelerated, spielman2009note} might apply to develop a more efficient ball acceleration oracle for \eqref{eq:smax}. If so, by combining this with the ball-acceleration methods of \cite{carmon2020acceleration}, one might obtain an alternative algorithm for $\ell_2$-$\ell_1$ games which requires $\tilde{O}(\epsilon^{-8/9})$-matrix-vector queries. Then, it might be possible to further extend this alternative algorithm to obtain an $\tilde{O}(\epsilon^{-8/9})$-matrix-vector query algorithm for $\ell_1$-$\ell_1$ games by combining ball acceleration with the techniques of \citet{carmon2024whole}.

In this paper, we do not directly present this ball-acceleration approach but instead present a primal-dual proximal point version as described previously. We take this approach because we believe it provides a straightforward and unified framework to simultaneously obtain deterministic guarantees for $\ell_2$-$\ell_1$ games, $\ell_1$-$\ell_1$ games, and $\ell_2$-$\ell_2$ composite games. An additional advantage of our proximal-point approach is that it allows us to cleanly apply an amortized analysis to obtain a further improvement for $\ell_2$-$\ell_1$ games, leading to our \emph{deterministic} $\tilde{O}(\epsilon^{-7/9})$-matrix-vector query algorithm for $\ell_2$-$\ell_1$ games.

\paragraph{Paper organization.} In Section~\ref{sec:prelim} we discuss notation. In Section~\ref{sec:l2l2}, we present our proof of Theorem~\ref{intro:l2l2} and introduce primitives which aid in the proof of Theorem~\ref{intro:l1l1} and~\ref{intro:l2l1}. In Section~\ref{sec:prox-point-regret}, we present our proximal point approach to reduce solving $\ell_1$-$\ell_1$ and $\ell_2$-$\ell_1$ games to solving a sequence of subproblems of the form \eqref{eq:reduction-subproblem}. In Section~\ref{sec:innerloops}, we show how to adapt our algorithms for $\ell_2$-$\ell_2$ composite games from Section~\ref{sec:l2l2} to solve subproblems of the form \eqref{eq:reduction-subproblem}. Finally, in Section~\ref{sec:elltwoelloneandelloneellone} we show how to implement the reduction from Section~\ref{sec:prox-point-regret} using the subproblem solvers derived in Section~\ref{sec:innerloops} to obtain Theorems~\ref{intro:l1l1} and~\ref{intro:l2l1}. Omitted proofs of well-known results and technical linear-algebraic lemmas are included in the appendix, for completeness.  

\section{Preliminaries}\label{sec:prelim}

\paragraph{General notation.} For a vector $x \in \R^n$, $[x]_i$ denotes the $i$-th entry of $x$ and $\normInline{x}_p$ denotes its $\ell_p$-norm.  Given $x \in \R^n$, we let $\diag(x)$ denote the diagonal matrix in $\R^{n \times n}$ where the $(i,i)$-th entry is given by $[x]_i$ and the off-diagonal entries are 0. We use the term \emph{unit vectors} to refer to vectors in the $\ell_2$-unit ball, and  $\normalize(x) \defeq x / \normInline{x}_2$ to denote the unit vector in the direction of $x \in \R^d$. We let $\simplex^n_\nu \defeq \inbraces{ x \in \simplex^n : [x]_i \ge \nu, \, \forall i \in [n] }$ denote the truncated simplex, and $\smash{\Delta^n_{ > 0} \defeq \{x \in \Delta^n : [x]_i > 0, \, \forall i \in [n] \} }$ denote the relative interior of the simplex. For $x \in \simplex_{>0}^n$ and $x' \in \Delta^n$, we define $\normInline{x'}_{x^{-1}} \defeq \normInline{\diag (x)^{-1/2} x'}_2$. For $x\in \Delta^n$ and $x' \in \Delta^n_{>0}$, we denote the KL divergence $\KL(x || x') \defeq \sum_{i \in [n]} [x]_i \log([x]_i/[x']_i)$, where we let $\smash{0 \log 0 \defeq 0}$. For sequences (of numbers, vectors, etc.) $\smash{u_1, u_2, \dots}$ or $\smash{u_1, w_1, u_2, w_2, \dots}$ we may use the notation $\{u_k\}$ and $\{u_k, w_k\}$ respectively. For a matrix $A$, we let $\smash{A_{i,:}}$ and $\smash{A_{:,i}}$ denote its $i$-th row and $i$-th column, respectively. If a function returns several variables and only a subset of them are pertinent to a certain lemma or statement, we may use a dot $(\cdot)$ in place of irrelevant outputs. For $u, v \in \R^n$ and $c > 0$, we write $u \approx_c v$ if and only if $[u]_i / c \le [v_i] \le c [u]_i$ for all $i \in [n]$ (note that we may use this notation with $n = 1$). For a matrix $B \in \R^{m \times n}$, we let $f_B(x, y) \defeq y^\top B x$ denote the bilinear form in $B$.

\paragraph{Matrix norms.} For matrices $A \in \R^{m \times n}$ of rank $r$, we let $\sigma(A) \in \R^r_{> 0}$ denote the vector of singular values of $A$. For $p \geq 0$, we use $\normInline{A}_{\cS_p} \defeq \normInline{\sigma}_p$ to denote the Schatten-$p$ norm of $A$. When $p$ is $1$, $2$, or $\infty$, we write $\normInline{A}_* \defeq \normInline{A}_{\cS_1}$, $\normInline{A}_F \defeq \normInline{A}_{\cS_2}$, and $\normInline{A}_2 \defeq \normInline{A}_{\cS_\infty}$ for the \emph{nuclear}, \emph{Frobenius}, and \emph{spectral} or \emph{operator} norms, respectively. We define $\inmaxnorm{A} \defeq \max_{i, j} |A_{ij}|$ and $\norm{A}_{p \to q} \defeq \max_{\innorm{x}_p \le 1} \innorm{Ax}_q$. For $s, t \in [1, \infty]$, we define $\innorm{A}_{(s, t)} \defeq \innorm{a}_t$ where $[a]_i = \innorm{A_{i, :}}_s$ for $i \in [m]$, i.e., $\innorm{A}_{(s, t)}$ is the $\ell_t$-norm of the vector $a$ whose $i$-th entry is given by the $\ell_s$-th norm of the $i$-th row of $A$. We show this is indeed a norm for completeness in Proposition~\ref{prop:instance-dependent-is-norm}.

\paragraph{Problem setups.} We encapsulate the dgf setups (Definition~\ref{def:dgf-setup}) that we consider in the next definition.

\begin{definition}[$\ellTwoEllTwo$, $\ellTwoEllOne$, and $\ellOneEllOne$ setups]
    \label{def:matrix-vector-games-setups}
With $d \defeq n + m$, we refer to the tuples $(\xset, \yset, \xtrunc, \ytrunc, \cZint, \rx : \xset \to \R, \ry : \yset \to \R, \norm{\cdot} : \R^d \to \R, \Range)$ defined in Table \ref{table:setups} as the \emph{$\ellTwoEllTwo$, $\ellTwoEllOne$, and $\ellOneEllOne$ setups} respectively (in the case of $\ellTwoEllTwo$ setup, we do not need $\cZint$ or the truncated sets $\xtrunc$ and $\ytrunc$). In the context of these setups, we further define $\zset \defeq \xset \times \yset$, $\ztrunc \defeq \xtrunc \times \ytrunc$, and $r : \zset \to \R$ via $r(z) \defeq \rx(z\x) + \ry(z\y)$. Finally, we make the standard normalization assumptions $\inmaxnorm{A} \le 1$ in the $\ellOneEllOne$ setup and $\innorm{A}_{2 \to \infty} \le 1$ in the $\ellTwoEllOne$ setup. (Note that $\inmaxnorm{A}$ and $\innorm{A}_{2 \to \infty}$ are the respective Lipschitz constants of $\gm f_A$ under $\norm{\cdot}$.)
\end{definition}

\begin{table}[ht]
    \centering
    \begin{tabular}{ c c c c }
    \hline
    & \textbf{$\ell_2$-$\ell_2$} & \textbf{$\ell_2$-$\ell_1$} & \textbf{$\ell_1$-$\ell_1$} \\ \hline
    $\cX$     & $\mathbb{B}^n$     & $\mathbb{B}^n$    & $\Delta^n$    \\ 
    $\cY$     & $\mathbb{B}^m$     & $\Delta^m$  & $\Delta^m$    \\ 
    $\cX_\nu$     & -     & $\mathbb{B}^n$    & $\Delta_\nu^n$    \\ 
    $\cY_\nu$     & -     & $\Delta_\nu^m$  & $\Delta_\nu^m$   \\
    $\cZint$     & -     & $\mathbb{B}^n \times \Delta_{>0}^n $    & $\Delta_{>0}^n \times \Delta_{>0}^m$    \\ 
    $\rx(x)$     & $\frac{1}{2}\norm{x}_2^2$    & $\frac{1}{2}\norm{x}_2^2$    & $\sum_{i \in [n]} [x]_i \log([x]_i)$    \\
    $\ry(y)$     & $\frac{1}{2}\norm{y}_2^2$    & $\sum_{i \in [m]} [y]_i \log([y]_i)$    & $\sum_{i \in [m]} [y]_i \log([y]_i)$    \\
    $\norm{z}^2$     & $\norm{z}^2_2$    & $\norm{z\x}_2^2 + \norm{z\y}_1^2$     & $\norm{z\x}_1^2 + \norm{z\y}_1^2$  \\ 
    $\Range$     & 1    & $\frac{1}{2} + \log(m)$     & $\log(mn)$  \\ 
    $V^r_{z'}(z)$ & $\frac{1}{2} \norm{z-z'}_2^2$ & $\frac{1}{2} \norm{z\x - z'\x}^2_2 + \KL(z\y||z'\y)$ & $\KL(z || z')$\\
    \hline 
    \end{tabular}
    \caption{$\ell_2$-$\ell_2$, $\ell_2$-$\ell_1$, and $\ell_1$-$\ell_1$ setups (Definition~\ref{def:matrix-vector-games-setups}) and additional associated notation.}
    \end{table}\label{table:setups}

    For the $\ell_p$-$\ell_q$ setups for $p, q \in \inbraces{1, 2}$, we have that $(\xset, \norm{\cdot}_p, \rx)$, $(\yset, \norm{\cdot}_q, \ry)$, and $(\zset, \norm{\cdot}, r)$ are all dgf setups per Definition~\ref{def:dgf-setup}. Furthermore, the induced Bregman divergence over the latter (which we also give explicitly in Table \ref{table:setups} for ease of reference) is given by the summation of the Bregman divergences over the former, i.e., $\breg{z}{z'} = \xbreg{z\x}{z\x'} + \ybreg{z\y}{z\y'}$. Finally, $\Range$ in the setups is the range of $r$, implying in particular $\sup_{z \in \zset} \breg{z'}{z} \le \Range$ for $z' \defeq \argmin_{z \in \zset} r(z)$. (We will not need notation for the ranges of $\rx$ and $\ry$.)

    \paragraph{Monotone operators and proximal mappings.} For a dgf setup $(\zset, \norm{\cdot}, r)$ per Definition~\ref{def:dgf-setup}, we say an operator $g : \zset \to \R^d$ is \emph{$\alpha$-strongly monotone relative to $r$} if $(g(z') - g(z))^\top (z' - z) \ge \alpha \breg{z'}{z}$ for all $z, z' \in \zset$. We say an operator is \emph{monotone} if it is 0-strongly monotone. Next, we define the proximal mappings we will use in the following definition:

    \begin{definition}[Proximal mappings]
        \label{def:proximal-mappings}
        For a given dgf setup $(\zset, \norm{\cdot}, r)$, continuous monotone operator $g : \zset \to \R^d$, points $z, w \in \zset$, regularization levels $\lambda > 0, \mu \ge 0$, and compact, convex $\zset' \subseteq \zset$, we let $\prox_{z, w}^{\lambda, \mu}(g; \zset')$ denote the unique $z' \in \zset'$ such that
        \begin{align*}
             g(z')^\top (z' - u) \le \lambda \insquare{\breg{z}{u} - \breg{z'}{u} - \breg{z}{z'}} + \mu \insquare{\breg{w}{u} - \breg{z'}{u} - \breg{w}{z'}} ~~\text{for all $u \in \zset'$},
        \end{align*}
        and similarly let $\prox_{z}^{\lambda}(g; \zset')$ denote $\prox_{z,z}^{\lambda,0}(g; \zset')$, i.e., the unique $z' \in \zset'$ such that
        \begin{align}
            \label{eq:prox-single}
            g(z')^\top (z' - u) \le \lambda \insquare{\breg{z}{u} - \breg{z'}{u} - \breg{z}{z'}} ~~\text{for all $u \in \zset'$}.
        \end{align}
        We drop $\zset'$ (e.g., writing $\prox_{z, w}^{\lambda, \mu}(g)$) when $\zset' = \zset$ for brevity. Furthermore, in the context of the input to a proximal mapping, we may write a vector $v \in \R^d$ as a stand-in for the associated constant operator $z \mapsto v$.
        
        As an example (which, e.g., is used in \eqref{eq:update}), supposing $g : \zset \to \R^d$ is a continuous monotone operator and $v \in \R^d$, then $\prox_{z}^{\lambda}(v + g; Z')$ denotes the unique $z' \in \zset'$ such that
        \begin{align*}
            (v + g(z'))^\top (z' - u) \le \lambda \insquare{\breg{z}{u} - \breg{z'}{u} - \breg{z}{z'}} ~~\text{for all $u \in \zset'$}.
        \end{align*}
    \end{definition}

    Note that the proximal mappings above are all solutions to continuous, strongly monotone variational inequalities, thereby guaranteeing existence and uniqueness (e.g., \cite{facchinei2003finitevariational}). Indeed,
    since Bregman divergences satisfy the following standard identity (e.g., \cite[Sec. 3.1]{carmon2019variance}):
    \begin{align*}
        - \grad \breg{z}{z'}^\top  (z' - u) = \breg{z}{u} - \breg{z'}{u} - \breg{z}{z'}\text{ for all }
        z, z', u \in \zset\,,
    \end{align*}
    the condition \eqref{eq:prox-single} is for example equivalent to 
    \begin{align}
        \label{eq:equivalent-prox-condition}
        (g(z') + \lambda \grad \breg{z}{z'})^\top (z' - u) \le 0 ~~\text{for all $u \in \zset'$},
    \end{align}
    where in general $\grad \breg{z}{z'} = \grad r(z') - \grad r(z)$ denotes the gradient of $u \mapsto \breg{z}{u}$ evaluated at $z'$. We define the proximal mappings as in \Cref{def:proximal-mappings} to enable more direct use in our applications.

    \paragraph{Convex-concave functions, gradient mappings, and regret.}  The gradient mapping $\gm f$ of a differentiable convex-concave function $f : \xset \times \yset \to \R$ is a monotone operator. 
    Furthermore, a fact we will use throughout this paper is that if $\rx : \xset \to \R$ and $\ry : \yset \to \R$ are distance-generating functions with $r(z) \defeq r\x(z\x) + r\y(z\y)$ a distance-generating function on $\zset = \xset \times \yset$ (namely, as in the setups of Definition~\ref{def:matrix-vector-games-setups}), then $\prox_z^\alpha(\gm f)$ is the exact solution of the following (see, e.g., Section 3.1 in \cite{carmon2019variance}):
    \begin{align*}
        \min_{x \in \xset} \max_{y \in \yset} f(x, y) + \alpha \xbreg{z\x}{x} - \alpha \ybreg{z\y}{y}.
    \end{align*}

    Finally, the following standard reduction (proven in Appendix~\ref{apx:regret-lemmas} for completeness) reduces computing an $\epsilon$-solution (Definition~\ref{def:epsilon-solution}) to the problem of computing a sequence of points $w^1, \ldots, w^T$ with $\epsilon$-average regret with respect to the gradient mapping. Our algorithms will compute such sequences to obtain $\epsilon$-solutions.

    \begin{restatable}{lemma}{regretbounderror}\label{lemma:regret-bounds-error} Let $f: \cX \times \cY \to \R$ be a differentiable convex-concave function over compact, convex sets $\xset \subset \R^n$ and $\yset \subset \R^m$, with $\zset \defeq \xset \times \yset$. Then for any $w^1, \ldots, w^T \in \cZ$ and $\rho_1, \dots, \rho_T > 0$, letting $\Lambda \defeq \sum_{t \in [T]} \rho_t$ and $\wbar \defeq \frac{1}{\Lambda} \sum_{t \in [T]} \rho_t w^t$, we have
        \begin{align}\label{eq:regret}
            \gap(\wbar) \leq \regret_{\gm f}(\inbraces{w^t, \rho_t}) \defeq \sup_{u \in \cZ} \inbraces*{ \frac{1}{\Lambda} \sum_{t \in [T]} \rho_t \nabla_\pm f(w^t)^\top (w^t - u)}, 
        \end{align}
        where $\regret_{\gm f}(\inbraces{w^t, \rho_t})$ is called the \emph{regret of the sequence $w^1, \ldots, w^T$ (with respect to the operator $\gm f$ and the weights $\inbraces{\rho_t}$).}
    \end{restatable}

\section{Warmup: algorithms for $\ell_2$-$\ell_2$ composite games}\label{sec:l2l2}

In this section, we consider the problem computing an $\epsilon$-solution of an $\ell_2$-$\ell_2$ composite game
\begin{align*}
    \min_{x \in \ball^n} \max_{y \in \ball^m} f_A(x, y) + \phi(x, y)
\end{align*}
for differentiable, convex-concave $\phi : \ball^n \times \ball^m \to \R$. Throughout this section (Section~\ref{sec:l2l2}), we operate in the $\ellTwoEllTwo$ setup (Definition~\ref{def:matrix-vector-games-setups}).
In Section~\ref{subsec:mirror-prox-step}, we describe how to use proximal mappings (Definition~\ref{def:proximal-mappings})
to construct a smooth-until-guilty composite mirror prox step. In Section~\ref{subsec:smooth-until-proven-guilty}, we leverage this as a subroutine to ultimately prove Theorem~\ref{intro:l2l2}. 

\subsection{Smooth-until-guilty composite mirror prox steps}\label{subsec:mirror-prox-step}

We first introduce a single smooth-until-guilty mirror prox step for the monotone operator $\nabla_\pm (\bilinear{B} + \psi)$, where $B \in \R^{m \times n}$ and $\psi$ is any differentiable convex-concave function over $\cZ$ (recall the notation $f_B(x, y) \defeq y^\top B x$). We assume access to a matrix-vector oracle for $B$ (Definition~\ref{def:mat-vec-oracle}) and explicit access to $\psi$. 

The following algorithm ($\proxStep(z, \tau, B, \psi; \judge)$, Algorithm~\ref{alg:mirror-prox-iteration}) implements a smooth-until-guilty mirror prox step (which consists of two proximal steps) and outputs the results along with a flag (which takes value $\guilty$ or $\smooth$) and a matrix $D \in \R^{m \times n}$ that is computed using a $\judge$~subroutine. Intuitively, this $\judge$~subroutine is responsible for identifying whether or not the prox steps for computing $w, z'$ (Lines~\ref{line:first-step} and~\ref{line:second-step}) meet the \emph{smoothness requirements} to make progress in reducing regret (see Lemma~\ref{lemma:regret-bounds-error} for the definition of regret). If they do, $\judge$~returns \smooth, $D = 0$, and $L=\none$. If not, $\judge$~returns \guilty~along with a matrix $D$ and $L > \tau$. This matrix $D$ will have the property that $\normInline{B - D}_{\cS_p}^p \leq \normInline{B}_{\cS_p}^p - L^p$, for an appropriate Schatten-$p$ norm, which may depend on the implementation of $\judge$.

We discuss implementation details of $\judge$~later, in Section~\ref{subsec:smooth-until-proven-guilty}. However, for now, we focus on proving general guarantees about Algorithm~\ref{alg:mirror-prox-iteration}. The following \Cref{lemma:mirror-prox-step-guarantee} shows that $\proxStep(z, \tau, B, \psi; \judge)$ either makes progress in bounding the regret of $w$ or else discovers a pair of unit vectors with large bilinear form in $B$, capturing the intuition laid out in our discussion of the smooth-until-proven-guilty approach from \Cref{sec:overview-of-approach}. 

\RestyleAlgo{ruled}\label{alg:mirror-prox-iteration}
\SetKwComment{Comment}{/* }{ */}
\begin{algorithm2e}[ht]
    \DontPrintSemicolon
\caption{Smooth-until-guilty mirror prox step $\proxStep(z, \tau, B, \psi; \judge)$}
\KwInput{Center $z \in \zset$, smoothness threshold $\tau > 0$}
\KwInput{Matrix-vector oracle for $B \in \R^{m \times n}$, differentiable convex-concave $\psi: \cZ \to \R$.} 
\KwParameter{Judge function $\judge$} 
\tcp{Perform a composite mirror prox step, consisting of two composite proximal steps }
$w \gets \proxStepSimple{z}{\tau}{\nabla_\pm \bilinear{B}(z) + \nabla_\pm \psi}$ \label{line:first-step}\;
$z' \gets \proxStepSimple{z}{\tau}{\nabla_\pm (\bilinear{B}+\psi)(w)}$ \label{line:second-step}\; 
$z^1 \gets ({w\x - z'\x}, {w\y - z\y},)$ and $z^2 \gets ({z\x - w\x}, {w\y - z'\y})$\; 
$(\flag, D) \gets \judge(z^1, z^2, B, \tau)$ \label{line:judgement} \tcp*{Use a judge to evaluate smoothness} 
\Return{$(w, z', \flag, D)$}
\end{algorithm2e}

\begin{lemma}\label{lemma:mirror-prox-step-guarantee} For $z \in \zset$, $\tau > 0$, $B \in \R^{m \times n}$, and differentiable convex-concave $\psi: \zset \to \R$, letting $(w, z', \flag, D) \defeq \proxStep(z, \tau, B, \psi; \judge)$ (Algorithm~\ref{alg:mirror-prox-iteration}), we have either 
\begin{align}\label{eq:smooth-equation}
    \nabla_\pm (\bilinear{B}+\psi)(w)^\top (w-u) \leq \tau \cdot [\breg{z}{u} -\breg{z'}{u}], ~~~\text{for all } u \in \cZ
\end{align}
or else
\begin{align}\label{eq:alternative}
    {z^1\y}^\top B z^1\x \geq \tau \normInline{z^1\y}_2 \normInline{z^1\x}_2
    ~~~\text{or}~~~
    {z^2\y}^\top B z^2\x\geq \tau \normInline{z^2\y}_2 \normInline{z^2\x}_2\,. 
\end{align} 
Furthermore, the algorithm can be implemented with $O(1)$-matrix-vector queries, plus the additional matrix-vector queries made by the $\judge$ subroutine in Line~\ref{line:judgement}.
\end{lemma}
\begin{proof} Applying Definition~\ref{def:proximal-mappings} to each composite proximal step (Lines~\ref{line:first-step} and~\ref{line:second-step}) individually, yields that that for all $u, u' \in \cZ$,
\begin{align}
    \frac{1}{\tau} \nabla_\pm \bilinear{B}(z)^\top (w-u') + \frac{1}{\tau}\nabla_\pm \psi(w)^\top (w-u') &\leq \breg{z}{u'} - \breg{w}{u'} - \breg{z}{w}, \label{eq:first-inequality-1}\\ \notag
    \frac{1}{\tau} \nabla_\pm \bilinear{B}(w)^\top (z'-u) + \frac{1}{\tau} \nabla_\pm \psi(w)^\top (z' - u) &\leq \breg{z}{u} - \breg{z'}{u} - \breg{z}{z'}.
\end{align}
Setting $u' = z'$ in \eqref{eq:first-inequality-1} and summing yields,
\begin{align*}
    &\frac{1}{\tau} \Brac{\nabla_\pm \bilinear{B}(z)^\top (w-z') + \nabla_\pm \psi(w)^\top (w-u) + \nabla_\pm \bilinear{B}(w)^\top (z'-u)} \\
    \leq & \breg{z}{u} - \breg{z'}{u} 
        - ( \breg{w}{z'} + \breg{z}{w} ). 
\end{align*}
Now, consider two cases. First, suppose that
\begin{align}\label{eq:successful-step-1}
    (\nabla_\pm \bilinear{B}(w) - \nabla_\pm \bilinear{B} (z))^\top(w - z') \leq \tau(\breg{w}{z'} + \breg{z}{w}). 
\end{align}
Then, combining the two preceding displays, we obtain 
\begin{align*}
    &\frac{1}{\tau} \Brac{\nabla_\pm \bilinear{B}(z)^\top (w-z') + \nabla_\pm \psi(w)^\top (w-u) + \nabla_\pm \bilinear{B}(w)^\top (z'-u)} \\
    \leq &\breg{z}{u} - \breg{z'}{u} - ( \breg{w}{z'} + \breg{z}{w} )
    \\
    \leq &\breg{z}{u} - \breg{z'}{u} -\frac{1}{\tau} (\nabla_\pm \bilinear{B}(w) - \nabla_\pm \bilinear{B}(z))^\top(w - z'). 
\end{align*}
Consequently, we have
\begin{align*}
       \frac{1}{\tau} \Brac{\nabla_\pm (\bilinear{B}+\psi)(w)^\top (w-u)} 
       =&\frac{1}{\tau} \Brac{\nabla_\pm \bilinear{B}(z)^\top (w-z') + \nabla_\pm \psi(w)^\top (w-u) + \nabla_\pm \bilinear{B}(w)^\top (z'-u)} \\
       & \spaceeq +  \frac{1}{\tau} \Brac{(\nabla_\pm \bilinear{B}(w) - \nabla_\pm \bilinear{B}(z))^\top(w - z')} \\
       \leq &\breg{z}{u} - \breg{z'}{u}. 
\end{align*}
Multiplying through by $\tau$, we obtain
\begin{align*}
    \nabla_\pm (\bilinear{B}+\psi)(w)^\top (w-u) \leq \tau \cdot [\breg{z}{u} - \breg{z'}{u}]. 
\end{align*}

On the other hand, suppose instead that \eqref{eq:successful-step-1} does not hold. Then, it must be the case that 
\begin{align}\label{eq:unsuccessful-step}
    \frac{1}{\tau} (\nabla_\pm \bilinear{B}({w}) - \nabla_\pm \bilinear{B}(z))^\top({w} - z') &> \breg{w}{z'} + \breg{z}{w} \\ \label{eq:unsuccessful-step-2}
    &= \frac{1}{2}\normInline{z'-{w}}_2^2 + \frac{1}{2}\normInline{{w}-z}_2^2. 
\end{align}
where \eqref{eq:unsuccessful-step-2} is by the definition of the $\ell_2$-$\ell_2$ setup (Definition~\ref{def:matrix-vector-games-setups}). 

Now, recall that for any $z' \in \cZ$, $\nabla_\pm \bilinear{B}(z') = (B^\top z'\y, -B z'\x)$. Thus, we observe that 
\begin{align*}
    (\nabla_\pm \bilinear{B}({w}) - \nabla_\pm \bilinear{B}(z))^\top ({w} - {z'}) &= \begin{pmatrix}
        B^\top (w-z)\y \\
        -B (w-z)\x
    \end{pmatrix}^\top \begin{pmatrix}
        w\x - z'\x \\
        w\y - z'\y
    \end{pmatrix} \\
    &= 
    ({{w}}\y - {z}\y)^\top B ({{w}}\x - {z'\x}) + ({{w}}\y - {z'\y})^\top B ({z}\x - {{w}}\x). 
\end{align*}
Thus, expanding out \eqref{eq:unsuccessful-step} implies that 
\begin{align*}
    &\frac{1}{\tau}({{w}}\y - {z\y})^\top B ({{w}}\x - {z'\x}) + \frac{1}{\tau}
    ({{w}}\y - {z'\y})^\top B ({z}\x - {{w}}\x) \\
    > &\frac{1}{2} \norm{{w}\x-{z'\x}}_2^2 + \frac{1}{2} \norm{{w}\y-{z'\y}}_2^2 + \frac{1}{2} \norm{{w}\x-z\x}_2^2 + \frac{1}{2}\norm{{w}\y-z\y}_2^2 \\
    > &\norm{{{w}}\y - {z}\y}_2 \norm{{{w}}\x - {{z'\x}}}_2 + \norm{{{w}}\y - {{z'\y}}}_2 \norm{{z}\x - {{w}}\x}_2, 
\end{align*}
where in the last line, we used that $\frac{1}{2}(a^2+b^2) \geq ab$ for any $a, b \in \R$. Therefore, either
\begin{align*}
    \frac{({{w}}\y - {z}\y)^\top B ({{w}}\x - {{z'\x}})}{\norm{{{w}}\y - {z}\y}_2 \norm{{{w}}\x - {{z'\x}}}_2} \geq \tau, \quad \text{ or } \quad \frac{({{w}}\y - {z'\y})^\top B ({z}\x - {{w}}\x)}{\norm{{{w}}\y - {{z'\x}}}_2 \norm{{w}\x - {{z}}\x}_2} \geq \tau. 
\end{align*}

Finally, the matrix-vector query complexity follows from noting that each of the prox steps in Lines~\ref{line:first-step} and~\ref{line:second-step} simply requires one matrix-vector query to $B$ (to evaluate $\nabla_\pm f_B(z), \nabla_\pm f_B(w)$).
\end{proof}

Note that when $\tau \leq \normInline{A}_2$, it is easy to see that \eqref{eq:smooth-equation} always holds (because \eqref{eq:alternative} \emph{cannot} hold) and in this case, the guarantees of Lemma~\ref{lemma:mirror-prox-step-guarantee} match the standard analysis of mirror prox \cite{nem04}. However, Lemma~\ref{lemma:mirror-prox-step-guarantee} differs from the standard analysis of mirror prox for $\ell_2$-$\ell_2$ games in that it provides a fine-grained guarantee, \emph{even} when $\tau > \normInline{A}_2$.

\subsection{Smooth-until-guilty composite mirror prox}\label{subsec:smooth-until-proven-guilty}

In this section, we show how to use the smooth-until-guilty composite mirror prox steps from Section~\ref{subsec:mirror-prox-step} to obtain $\epsilon$-solutions for $\ell_2$-$\ell_2$ composite games. First, we will formalize our requirements of the $\judge$~subroutine in Algorithm~\ref{alg:mirror-prox-iteration} as follows.  

\begin{definition}[Smooth-guilty judge]\label{def:smooth-guilty-judge} For $p \in [1, \infty)$, an algorithm $\judge$~is a $p$-smooth-guilty judge if for any $z, z' \in \R^d$, $B > 0$, $\tau > 0$, we have that $(\flag, D) = \judge(z, z', B, \tau)$ makes $O(1)$ queries to a matrix-vector oracle for $B$ and satisfies the following: If
\begin{align*}
    \max\paren{ \frac{{z}\y^\top B {z}\x}{\normInline{{z}\y}_2\normInline{{z}\x}_2}, \frac{{z'\y}^\top B z'\x}{\normInline{{z'\x}}_2\normInline{{z'\y}}_2} } > \tau, 
\end{align*}
then $\flag = \guilty$, and $\normInline{B-D}_{\cS_p}^p \leq \normInline{B}_{\cS_p}^p - \tau^p$. Otherwise, $\flag = \smooth$ and $D = 0$. 
\end{definition} 

The following Algorithm~\ref{alg:judge-l2l2} provides a simple implementation of a $p$-smooth-guilty judge for any $p \in [1, \infty)$. The algorithm takes in two vectors $z, z' \in \cZ$, a matrix $B$, and a threshold $\tau$. It checks if components of $z, z'$ reveal a $\tau$-large component of $B$. If so, the algorithm projects a large component off of $B$, stores this low-rank component in a new matrix $D$, and returns a flag $\guilty$ and along with $D$. Otherwise, it returns $\smooth$ and $D = 0$.

\RestyleAlgo{ruled}\label{alg:judge-l2l2}
\SetKwComment{Comment}{/* }{ */}
\begin{algorithm2e}[ht]
\caption{$\judge_{\cS}(z, z', B, \tau)$}
\KwInput{Vectors $z, z' \in \R^d$, smoothness threshold $\tau > 0$.}
\KwInput{Matrix-vector oracle for a matrix $B \in \R^{m \times n}$.} 

\tcp{Check if pair of unit vectors has large bilinear form in $B$}

\lIf{$z\y^\top B z\x > \tau \normInline{z\y}_2\normInline{z\x}_2 \label{line:if1}$}{
    $v \gets \normalize(z\y)$ and  $u \gets \normalize(z\x)$
}
\lElseIf{${z'\y}^\top B z'\x > \tau \normInline{z'\y}_2\normInline{z'\x}_2$ \label{line:if2}}{
    $v \gets \normalize(z'\y)$ and $u \gets \normalize(z'\x)$
}
\lElse{
    \Return{$(\smooth, 0)$}
}
\tcp{ Implemented with 3 matrix-vector queries as $D \gets vv^\top B + Buu^\top  - (u^\top B v) uv^\top$ }
$D \gets  B - (I-vv^\top) B (I - uu^\top)$\label{line:two-sided}\; 
\Return{$({\guilty}, D)$}
\end{algorithm2e}

Lemma~\ref{lemma:smooth-guilty-judge} proves that the algorithm implements a $p$-smooth guilty judge (Definition~\ref{def:smooth-guilty-judge}). The proof utilizes the following Pythagorean-type identity for Schatten norms from \cite{bakshi2022low}. We remark that for the special case of $p = 2$, other (simpler) implementations of a $2$-smooth-guilty judge are possible. We discuss this further in Section~\ref{sec:four-fifths} as well as Appendix~\ref{apx:other-judge}. 

\begin{lemma}[Lemma 5.5 of \cite{bakshi2022low}, restated]\label{lemma:pythagorean-theorem} If $x \in \B^n, y \in \B^m$ and $B \in \R^{m \times n}$, then
\begin{align*}
    \normInline{(I - yy^\top) B (I - xx^\top)}_{\cS_p}^p \leq \normInline{B}_{\cS_p}^p - \normInline{yy^\top B xx^\top}_{\cS_p}^p. 
\end{align*}
\end{lemma}

\begin{restatable}{lemma}{smoothguiltyjudgep} \label{lemma:smooth-guilty-judge} For any $p \in [1, \infty)$, $\judge_{\cS}$ (Algorithm~\ref{alg:judge-l2l2}) is a $p$-smooth-guilty-judge. 
\end{restatable}

\begin{proof} If neither of the if statements (Lines~\ref{line:if1} or \ref{line:if2}) execute, then the algorithm outputs $(\smooth, 0)$ as required. Otherwise, we have that $v^\top B u > \tau$, and we need to prove that $\norm{ B - D }_{\cS_p}^p \leq \norm{B}_{\cS_p}^p - \tau^p.$ First, note that $B - D = (I - vv^\top) B (I - uu^\top)$. Now, by Lemma~\ref{lemma:pythagorean-theorem}, we have 
\begin{align*}
    \norm{ B - D }_{\cS_p}^p &\leq \norm{B}_{\cS_p}^p - \norm{vv^\top B uu^\top}_{\cS_p}^p = \norm{B}_{\cS_p}^p - (v^\top B u)^p \norm{vu^\top}_{\cS_p}^p = \norm{B}_{\cS_p}^p - (v^\top B u)^p \\
    &\leq \norm{B}_{\cS_p}^p - \tau^p
\end{align*}
where the last step used that $\normInline{vu^\top}_{\cS_p} = 1$ because $u, v$ are unit vectors. Finally, the algorithm clearly requires only $O(1)$ matrix-vector queries to $B$ to compute $D = vv^\top B - Buu^\top - vv^\top B u u^\top.$ 
\end{proof}

Next, we introduce Algorithm~\ref{alg:mirror-prox-sug-l2l2}, our main algorithm for solving $\ell_2$-$\ell_2$ matrix-vector games. The algorithm initializes what we call a \emph{model} $M_0 = 0 \in \R^{m \times n}$ and sets $k = 0$, and initially works with the input matrix $A_0 = A$ and composite function $\psi_1 = \phi$. The algorithm then runs multiple iterations of a while loop. In each iteration, the algorithm performs a $\proxStep$ step (Algorithm~\ref{alg:mirror-prox-iteration}) with the current $A_k$ and composite part $\psi_k$. If the step returns $\guilty$, the algorithm \emph{updates the model} by $M_{k+1} \gets M_k + D$ and $A_{k+1} \gets A - M_{k+1}$ and the composite function by $\psi_{k+1} \gets \phi + f_{M_k}$. Then, the number of \emph{model-update iterations} (stored in the variable $k$) is incremented. Otherwise, if the step returns $\smooth$, the number of progress iterations (stored in the variable $j$) is incremented. 

\RestyleAlgo{ruled}
\DontPrintSemicolon
\SetKwComment{Comment}{/* }{ */}
\begin{algorithm2e}[h!]
\caption{Smooth-until-guilty mirror prox for $\ell_2$-$\ell_2$ composite matrix-vector games}\label{alg:mirror-prox-sug-l2l2}
\KwInput{Smoothness threshold $\tau$} 
\KwInput{Matrix-vector oracle for $A$ and differentiable convex-concave function $\phi: \cX \to \R$}
\KwParameter{Judge function $\judge$.} 
$j \gets 0, ~k \gets 0$\; 
$z^1 \gets \argmin_{z \in \cZ}r(z)$\;
$\psi_0 \gets \phi, ~A_0 \gets A, ~M_0 \gets 0$ \tcp*{$M_0 \in \R^{m \times n}$}
\While{$j \leq J$ where $J =  \ceil{\tau/\epsilon}$}{
    \BlankLine
    $(w^{j}, z^{j+1}, \modelUpdateStep, D) \gets \proxStep(z^{j}, \tau, A_k, \psi_k; \judge)$     \tcp*{Smooth-until-guilty mirror prox step}
    
    \BlankLine
    \If{$\flag = \guilty$}
    {\label{line:model-update}
        $M_{k+1} \gets M_k + D$ \tcp*{Addition implemented explicitly, since $M_k, D$ are explicit matrices}
        $A_{k+1} \gets A - M_{k+1}$ \tcp*{Subtraction implemented implicitly, since $A$ is not known explicitly}
        $\psi_{k+1} \defeq \phi + f_{M_{k+1}}$ \tcp*{Addition implemented explicitly, since $\phi, M_k$ are known explicitly} 
        $k \gets k + 1$ \tcp*{Increment the number of model-update iterations (for analysis only)}
    }
    
    \lElse(\tcp*[f]{Otherwise, increment the number of progress iterations}){
        $j \gets j + 1$ 
    }
}
\Return{$\frac{1}{J}\sum_{j \in [J]} w^j$}
\end{algorithm2e}

\begin{lemma}\label{thm:main-general-l2l2} Consider an $\ell_2$-$\ell_2$ composite game and let $p \in [1, \infty)$. Let $\judge$~be a $p$-smooth-guilty judge (e.g., Algorithm~\ref{alg:judge-l2l2}). Then, Algorithm~\ref{alg:mirror-prox-sug-l2l2} terminates after making $O(\normInline{A}^p_{\cS_p} / \tau^p + \tau/\epsilon)$-matrix-vector queries and returns an $\epsilon$-solution of the $\ell_2$-$\ell_2$ composite game. 
\end{lemma}
\begin{proof} First, it is easy to see that for any $k \geq 0$ and $x \in \cX, y \in \cY$, we maintain that $M_k + A_k = A$, and hence 
\begin{align*}
    f_A + \phi = f_{A_k + M_k} + \phi = f_{A_k} + \psi_k. 
\end{align*}
Thus, by Lemma~\ref{lemma:mirror-prox-step-guarantee} and the fact that $\judge$ is a $p$-smooth-guilty judge, we see that for all $u \in \cZ$,
\begin{align*}
    \frac{1}{J}\sum_{j \in [J]} \nabla_\pm (\bilinear{A} + \phi)(w^j)^\top (w^j - u) &\leq \frac{1}{J} \sum_{j \in [J]} \tau \cdot \Brac{\breg{z^j}{u} - \breg{z^{j+1}}{u}} \leq \frac{\tau}{J} \breg{z^1}{u}  \leq \frac{\tau\Range}{J} \\
    &\leq \tau/J \leq \epsilon.
\end{align*}
Consequently, by Lemma~\ref{lemma:regret-bounds-error}, the algorithm outputs an $\epsilon$-solution for the $\ell_2$-$\ell_2$ composite game. 

Now, note that for each iteration of the while loop, the algorithm either performs a progress iteration and increments $j$, or, it performs a model-update iteration and updates $k$. Thus, to bound the total number of iterations of the algorithm, since the number of progress iterations is at most $J$, it only remains to bound $k$. To this end, note that because $\judge$~is a $p$-smooth-guilty-judge, $\normInline{A_{k+1}}_{\cS_p}^p \leq \normInline{A_k}_{\cS_p}^p - \tau^p$ for each $k\geq 0$. Hence, by induction,
$ 0 \leq \normInline{A_{k+1}}_{\cS_p}^p \leq \normInline{A_k}_{\cS_p}^p - \tau^p k$. 
Thus, $k \leq \normInline{A}_{\cS_p}^p/\tau^p$. Thus, the algorithm \emph{must} terminate after $\ceil{\normInline{A}_{\cS_p}^p/\tau^p} + \ceil{\tau/\epsilon}$ iterations of the while loop. Each iteration requires $O(1)$ matrix-vector queries, in light of Definition~\ref{def:smooth-guilty-judge} and Lemma~\ref{lemma:mirror-prox-step-guarantee}, which yields the query complexity guarantee. 
\end{proof}

We are now ready to prove our main result for the $\ell_2$-$\ell_2$ setup.

\mainresultthree*
\begin{proof} Consider Algorithm~\ref{alg:mirror-prox-sug-l2l2} with $\judge = \judge_p$ and $\tau = \normInline{A}_{\cS_p}^{p/(1+p)}\epsilon^{1/(1+p)}$. By Lemma~\ref{lemma:smooth-guilty-judge}, $\judge_p$ is a $p$-smooth-guilty judge. Thus, \Cref{thm:main-general-l2l2} ensures that the algorithm makes at most
\begin{align*}
    O\paren{\frac{\normInline{A}_{\cS_p}^{p/(1+p)} \epsilon^{1/(1+p)}}{\epsilon}} =  O\paren{\normInline{A}_{\cS_p}^{p/(1+p)} \epsilon^{-1 + 1/(1+p)}}
\end{align*}
matrix-vector queries and returns an $\epsilon$-solution of the $\ell_2$-$\ell_2$ composite game.
\end{proof}

When $p = 2$, Theorem~\ref{intro:l2l2} implies an oracle complexity of $O(\normInline{A}_{F}^{2/3} \epsilon^{-2/3})$, which is known to be optimal in light of lower bounds \cite{liu2023accelerated}. For $p = 1$, Theorem~\ref{intro:l2l2} implies an oracle complexity of $O(\normInline{A}_{*}^{1/2} \epsilon^{-1/2})$. For $p \in [1, 2)$, Schatten-$p$ norms could be considered more robust than the Frobenius norm, in the sense that they \emph{dampen} the effect of large singular values; meanwhile, for $p > 2$, Schatten-$p$ norms are \emph{more} sensitive to large singular values \citep{bakshi2022low}. Theorem~\ref{intro:l2l2} demonstrates that one can achieve oracle complexities tailored to the singular value decay of the underlying utility matrix $A$. 
\section{Proximal point method for non-Euclidean geometries}
\label{sec:prox-point-regret}

Toward generalizing the approach from Section~\ref{sec:l2l2} to non-Euclidean geometries (namely, the $\ellTwoEllOne$ and $\ellOneEllOne$ setups of Definition~\ref{def:matrix-vector-games-setups}), in this section we give a general proximal point method (Algorithm~\ref{alg:proximal-point-regret}) which reduces the problem of obtaining $\epsilon$-regret with respect to a continuous monotone operator (which in turn can be used to solve minimax problems per Lemma~\ref{lemma:regret-bounds-error}) to approximately solving a sequence of variational inequalities with respect to strongly monotone operators. In particular, we instantiate Algorithm~\ref{alg:proximal-point-regret} in the context of the $\ellTwoEllOne$ and $\ellOneEllOne$ setups in Section~\ref{sec:elltwoelloneandelloneellone}, where it forms the outer loop of our ultimate algorithm for these setups.

Formally, we fix a dgf setup $(\zset, \norm{\cdot}, r)$ per Definition~\ref{def:dgf-setup} with $\Range \defeq \max_{z, z' \in \zset} r(z) - r(z')$ 
and a continuous monotone operator $g : \zset \to \R^d$. Our goal is to obtain a sequence $z^1, \dots, z^K \in \zset$ and $\rho_1, \dots, \rho_K > 0$ with $\Lambda \defeq \sum_{k \in [K]} \rho_k$ such that 
\begin{align*}
    \regret_g(\inbraces{z^k, \rho_k}) \defeq \sup_{u \in \zset} \inbraces*{\frac{1}{\Lambda} \sum_{k \in [K]} \rho_k g(z^k)^\top (z^k - u)} \le \epsilon.
\end{align*}

Before stating Algorithm~\ref{alg:proximal-point-regret}, we define a type of relaxed proximal oracle called a \emph{dynamic approximate proximal oracle ($\DAPO$)} in Definition~\ref{def:DAPO} below, which approximately solves a strongly monotone variational inequality. We additionally say the oracle is \emph{kinetic} if its outputs satisfy either a movement lower bound or have a certain level of regularization.

\begin{definition}[$\DAPO$]
    \label{def:DAPO}
A \emph{$(z \in \zset, \epsilon' > 0)$-dynamic approximate proximal oracle}, $\DAPO(z, \epsilon')$, returns a pair $(z' \in \zset, \alpha > 0)$ such that
\begin{align}
    \label{eq:DAPO-variational-cond}
    g(z')^\top (z' - u) \le \alpha \insquare{ \breg{z}{u} -  \breg{z'}{u} -  \breg{z}{z'} } + \epsilon' \text{~~for all $u \in \zset$}.
\end{align}
We say a $\DAPO$ oracle is \emph{$(\beta > 0, c > 0)$-kinetic} if for any input $(z, \epsilon')$, the resulting output $(z', \alpha)$ additionally satisfies at least one of
\begin{enumerate*}[series = tobecont, itemjoin =, label=(\alph*)]
    \item $\alpha = \beta$ or \label{item:third}
    \item $\breg{z}{z'} \ge \alpha^c$. \label{item:gthird}
\end{enumerate*}
\end{definition}

We give our proximal point method in Algorithm~\ref{alg:proximal-point-regret} below, where we define the condition in Line~\ref{algline:while-loop} to be $\true$ when $k = 0$. What distinguishes Algorithm~\ref{alg:proximal-point-regret} from the standard proximal point method is precisely the kinetic condition in Definition~\ref{def:DAPO}, which was also discussed in Section~\ref{sec:overview-of-approach}. The correctness of Algorithm~\ref{alg:proximal-point-regret}, given in Lemma~\ref{lem:prox-point-correctness} below, does not require the kinetic condition, but our iteration bound given in Lemma~\ref{lem:prox-point-iteration-bound} critically does. 
In particular, the condition $\breg{z}{z'} \ge \alpha^c$ in Definition~\ref{def:DAPO} makes use of the $- \breg{z}{z'}$ term on the right-hand side of \eqref{eq:DAPO-variational-cond}, which is typically dropped in standard proximal point analyses, to certify faster progress.
This dynamic setting of the regularizer $\alpha$ to satisfy a movement is reminiscent of MS oracles (e.g., Definitions 1 and 2 in \cite{carmon2022optimalandadaptivemonteirosvaiter}) from the Monteiro-Svaiter acceleration \cite{renato2013monteirosvaiteroriginalpaper,bubeck2019highlysmooth,bubeck2019highlyparallel,carmon2022optimalandadaptivemonteirosvaiter} and ball acceleration \cite{carmon2020acceleration,jambulapati2024closingcomputationalquerygap,carmon2021thinking,hilal2021stochasticbiasreduced,carmon2022distributionallyrobustoptimizationball,carmon2023resqueing,carmon2024whole} literature, and Definition~\ref{def:DAPO} can be viewed as a variational-inequality variant of these. Furthermore, the use of \Holder's inequality in the proof of Lemma~\ref{lem:prox-point-iteration-bound} to bound the number of iterations which satisfy $\breg{z}{z'} \ge \alpha^c$ is similar to proof techniques from the aforementioned literature.

The underlying primitive subroutine used to implement a $\DAPO$ oracle in the context of the $\ellTwoEllOne$ and $\ellOneEllOne$ setups will be given in the next section (Section~\ref{sec:innerloops}). When applying Algorithm~\ref{alg:proximal-point-regret} to $\ellOneEllOne$ and $\ellTwoEllOne$ matrix games in Section~\ref{subsec:ell2ell1-ell1ell1-putting-all-together}, we ultimately implement Definition~\ref{def:DAPO} with $c \gets 2$ and $\beta \gets \epsilon^{1/3}$ so as to obtain an $\Otilde(\epsilon^{-2/3})$ iteration bound due to Lemma~\ref{lem:prox-point-iteration-bound}. As discussed in Section~\ref{sec:overview-of-approach}, the $\alpha = \beta$ condition in this context serves as a ``minimum regularization level''; we first check whether \eqref{eq:DAPO-variational-cond} is implementable using the techniques of Section~\ref{sec:innerloops} with $\alpha = \beta = \epsilon^{1/3}$, and if not, we binary search for $\alpha > \beta$ which satisfies the second condition $\breg{z}{z'} \ge \alpha^c = \alpha^2$ (and for which \eqref{eq:DAPO-variational-cond} is implementable via the techniques of Section~\ref{sec:innerloops}). Note that the $\beta / \epsilon$ term in Lemma~\ref{lem:prox-point-iteration-bound} reflects standard iteration bound for the proximal point method when using a fixed level of regularization $\beta$.

\RestyleAlgo{ruled}
\DontPrintSemicolon
\SetKwComment{Comment}{/* }{ */}
\begin{algorithm2e}[h!]
\caption{Proximal point method}
\label{alg:proximal-point-regret}
\KwInput{Precision $\epsilon > 0$, $\DAPO$ oracle (Definition~\ref{def:DAPO})}

$z^0 \gets \argmin_{z \in \zset} r(z)$ ~and~ $k \gets 0$ \;

\While(\tcp*[f]{Recall $\Range \defeq \max_{z, z' \in \zset} r(z) - r(z')$}){$\sum_{j \in [k]} \alpha_j^{-1} < \Range / \epsilon$}{ \label{algline:while-loop}

    $k \gets k + 1$ \;

    $(z^k, \alpha_k) \gets \DAPO(z^{k - 1}, \epsilon)$ \label{algline:DAPO} \tcp*{$g(z^k)^\top (z^k - u) \le \alpha_k \insquare{ \breg{z^{k - 1}}{u} - \breg{z^{k}}{u}     -  \breg{z^{k - 1}}{z^k} } + \epsilon, \, \forall u \in \zset$}

}

\Return{$\inbraces{z^k, \alpha_k^{-1}}$ and $K \defeq k$ \tcp*{$K$ is used in the analysis to refer to the final iteration count}
}\label{algline:return}

\end{algorithm2e}

Then moving on to our formal guarantees, we prove in Lemma~\ref{lem:prox-point-correctness} that Algorithm~\ref{alg:proximal-point-regret} achieves our desired regret guarantee, and also bound the sum of the divergences between consecutive iterates up to the penultimate iterate (note that this framing unites the movement bounds in the separate cases in Lemma~\ref{lem:prox-point-correctness}). The exclusion of the final movement term $\breg{z^{K - 1}}{z^K}$ when the algorithm terminates is due to the use of approximate proximal oracle calls (namely, the fact that the variational inequality in \Cref{algline:DAPO} allows for some additive $\epsilon$ error). This combined with the fact that the final step may significantly overshoot the stopping threshold in \Cref{algline:while-loop} (in other words, $\sum_{j \in [K]} \alpha_j^{-1}$ may be much larger than $\Range / \epsilon$) prevents us from controlling $\breg{z^{K - 1}}{z^K}$; see \Cref{eq:bounding-gap} in particular. However, in our ultimate application to matrix games we are able to control $\breg{z^{K - 1}}{z^K} = O(1)$ directly (see the proof of \Cref{lem:movement-bound-local-norm-points}).

\begin{lemma}[Algorithm~\ref{alg:proximal-point-regret} correctness]
    \label{lem:prox-point-correctness}
    If Algorithm~\ref{alg:proximal-point-regret} terminates, letting $S \defeq \sum_{k \in [K]} \alpha_k^{-1}$, the iterates satisfy
    \begin{align*}
        \regret_g(\inbraces{z^k, \alpha_k^{-1}}) = \sup_{u \in \zset} \inbraces*{\frac{1}{S} \sum_{k \in [K]} \alpha_k^{-1} g(z^k)^\top (z^k - u)} \le 2 \epsilon ~~\text{and}~~ \sum_{k \in [K - 1]} \breg{z^{k - 1}}{z^k} \le 2 \Range.
    \end{align*}
    If \Cref{alg:proximal-point-regret} does not terminate, then $\sum_{k \ge 1} \breg{z^{k - 1}}{z^k} \le 2 \Range$.
\end{lemma}

\begin{proof}
Line~\ref{algline:DAPO} in Algorithm~\ref{alg:proximal-point-regret} and Definition~\ref{def:DAPO} yield for all $k \ge 1$,
\begin{align*}
    g(z^k)^\top (z^k - u) \le \alpha_k \insquare{ \breg{z^{k - 1}}{u} - \breg{z^{k}}{u}     -  \breg{z^{k - 1}}{z^k} } + \epsilon, ~~~\text{for all $u \in \zset$}.
\end{align*}
Letting $S_{t'} \defeq \sum_{k \in [t']} \alpha_k^{-1}$ for $t' \ge 1$, multiplying both sides of the above by $\alpha_k^{-1} / S_{t'}$, summing, and using the nonnegativity of Bregman divergences yields
\begin{align}
    \label{eq:bounding-gap}
    0 \overle{(i)} \sup_{u \in \zset} \inbraces*{\frac{1}{S_{t'}} \sum_{k \in [t']} \alpha_k^{-1} g(z^k)^\top (z^k - u)} \le \frac{\Range - \sum_{k \in [t']} \breg{z^{k - 1}}{z^k}}{S_{t'}} + \epsilon,
\end{align}
where $(i)$ follows since regret with respect to a monotone operator is nonnegative (for completeness, we prove this in Proposition~\ref{prop:regret-wrt-monotone-operator-nonnegative} in Appendix~\ref{apx:regret-lemmas}). Here, we also use the fact that $\breg{z^0}{u} \le \Range$ for all $u \in \zset$ since $z^0 = \argmin_{z \in \zset} r(z)$.

Then supposing \Cref{alg:proximal-point-regret} terminates, the first claim follows by instantiating $t' \gets K$ in \eqref{eq:bounding-gap} and noting $S = S_K \ge \Range / \epsilon$ due to the termination condition in \Cref{algline:while-loop}. As for the second claim, the case $K = 1$ is trivial. Otherwise, taking $t' = K - 1$ and using the termination condition in \Cref{algline:while-loop}, we get
\begin{align*}
    \sum_{k \in [K - 1]} \breg{z^{k - 1}}{z^k} \le \Range + \epsilon \cdot S_{K - 1} \le \Range + \epsilon \cdot \Range / \epsilon \le 2 \Range.
\end{align*}
Similarly, supposing \Cref{alg:proximal-point-regret} does not terminate, the termination condition in \Cref{algline:while-loop} and \eqref{eq:bounding-gap} imply that for all $t' \ge 1$,
\begin{align*}
    \sum_{k \in [t']} \breg{z^{k - 1}}{z^k} \le \Range + \epsilon \cdot S_{t'} \le \Range + \epsilon \cdot \Range / \epsilon \le 2 \Range.
\end{align*}
\end{proof}

We now give an iteration bound under the additional condition of Definition~\ref{def:DAPO}:

\begin{lemma}[Algorithm~\ref{alg:proximal-point-regret} iteration bound]
    \label{lem:prox-point-iteration-bound}
If the $\DAPO$ oracle given as input to Algorithm~\ref{alg:proximal-point-regret} is $(\beta, c)$-kinetic, then the number of iterations $K$ is at most $\inparen{\beta / \epsilon + 2^{\frac{1}{c + 1}} \epsilon^{- \frac{c}{c + 1}}} \Range + 2$.
\end{lemma}

\begin{proof}
Let $J_a \defeq \inbraces{k \ge 1 : \alpha_k = \beta}$ and $J_b \defeq \inbraces{k \ge 1 : \alpha_k \ne \beta}$, where we restrict to values of $k$ such that $\alpha_k$ is well-defined. With this definition, every $k \in J_b$ satisfies the movement bound $\breg{z^{k - 1}}{z^k} \ge \alpha_k^c$ by \Cref{def:DAPO}.

We first prove termination. Note $|J_a| \le \beta \Range / \epsilon + 1$ by the termination condition in Line~\ref{algline:while-loop}, and therefore it suffices to show $|J_b|$ is finite. Supposing for the sake of contradiction that $|J_b|$ is infinite, we have for every $t' \ge 1$,
\begin{align*}
    \sum_{k \in J_b \cap [t']} \alpha_k^c \le \sum_{k \in J_b \cap [t']} \breg{z^{k - 1}}{z^k} \le 2 \Range
\end{align*}
by \Cref{lem:prox-point-correctness}. Thus, $\lim_{k \to \infty, k \in J_b} \alpha_k = 0$, contradicting the termination condition in \Cref{algline:while-loop}.

Having shown that \Cref{alg:proximal-point-regret} terminates, we now prove the bound on $K$. Let $J_b' \defeq J_b \setminus \inbraces{K}$, and note that the nonnegativity of Bregman divergences implies
\begin{align}
    \label{eq:sum-c-exps}
    \sum_{k \in J_b'} \alpha_k^c \le \sum_{k \in J_b'} \breg{z^{k - 1}}{z^k} \le \sum_{k \in [K - 1]} \breg{z^{k - 1}}{z^k} \le  2 \Range,
\end{align}
where the last inequality followed from Lemma~\ref{lem:prox-point-correctness}.
Then
\begin{align*}
    \inabs{J_b'} = \sum_{k \in J_b'} \alpha_k^{\frac{c}{c + 1}} \alpha_k^{- \frac{c}{c + 1}} \overle{(i)} \inparen*{\sum_{k \in J_b'} \alpha_k^c}^{\frac{1}{c + 1}} \inparen*{\sum_{k \in J_b'} \alpha_k^{-1}}^{\frac{c}{c + 1}} \overle{(ii)} (2 \Range)^{\frac{1}{c + 1}} (\Range / \epsilon)^{\frac{c}{c + 1}} = 2^{\frac{1}{c + 1}} \Range \epsilon^{- \frac{c}{c + 1}},
\end{align*}
by $(i)$ \Holder's inequality and $(ii)$ Equation~\ref{eq:sum-c-exps} and the fact that $\sum_{k \in J_b'} \alpha_k^{-1} \le \sum_{k \in [K - 1]} \alpha_k^{-1} < \Range / \epsilon$ due to the termination condition in Line~\ref{algline:while-loop}. Then
\begin{align*}
    K = \inabs{J_a} + \inabs{J_b} \le \inabs{J_a} + \inabs{J_b'} + 1 \le  \inparen{\beta / \epsilon + 2^{\frac{1}{c + 1}} \epsilon^{- \frac{c}{c + 1}}} \Range + 2.
\end{align*}
\end{proof}

\section{Subproblem solvers for non-Euclidean geometries}\label{sec:innerloops}

Motivated by the proximal point method derived in Section~\ref{sec:prox-point-regret}, our eventual goal is to implement a DAPO (Definition~\ref{def:DAPO}) for the $\ell_2$-$\ell_1$ and $\ell_1$-$\ell_1$ geometries. As a first step towards implementing a DAPO, in this section we will show how to approximate a special type of  proximal mapping (recall Definition~\ref{def:proximal-mappings}), as described in \eqref{eq:reduction-subproblem} and which we describe in the following paragraphs in more detail. Later, in Section~\ref{sec:elltwoelloneandelloneellone}, we discuss how to use the proximal mapping subroutines developed in the remainder of this section to implement a DAPO and solve $\ell_1$-$\ell_1$ and $\ell_2$-$\ell_1$ games. Throughout this section, we operate in the $\ellTwoEllOne$ and $\ellOneEllOne$ setups of Definition~\ref{def:matrix-vector-games-setups}; if a statement does not explicitly distinguish between them, it applies to both setups. 

Suppose that we are given matrix-vector access to $A \in \R^{m \times n}$ (Definition~\ref{def:mat-vec-oracle}), a center point $z_c \in \cZint$, and a regularization parameter $\alpha > 0$. Suppose further that we are supplied with a set $\cZ' \defeq \cX' \times \cY'$ which is a compact, convex subset of the $c_1$-stable region about some point $\Tilde{z} \in \cZint$, i.e., $\cZ' \subset \cB_{c_1, \Tilde{z}}$ where $\cB_{c_1, \Tilde{z}}$ is defined as follows. (Recall that for vectors $v, \tilde{v} \in \R^k$ and $c > 0$, we write $v \approx_c \tilde{v}$ if and only if $[\tilde{v}]_i/c \leq [v]_i \leq c [\tilde{v}]_i$ for all $i \in [k]$.)

\begin{definition}[$c$-stable region]\label{def:stable-ball} Let $\Tilde{z} \in \cZint$ and $c > 1$. The $c$-\emph{stable region} or \emph{$c$-multiplicative ball} about $\Tilde{z}$ is defined as follows 
\begin{align*}
    \cB_{c, \tilde{z}} \defeq \begin{cases}
        \{z \in \cZ: {{z}} \approx_c \tilde{z}\}, & \cX = \Delta^n, \\
         \{z \in \cZ: {{z}}\y \approx_c \tilde{z}\y \}, & \cX = \B^n.
    \end{cases} 
\end{align*}
\end{definition}
In other words, the $c_1$-stable region about $\tilde{z} \in \cZint$, $\cB_{c_1, \tilde{z}}$, is the subset of $\cZint$ where all \emph{simplex} coordinates are entrywise $c_1$-multiplicatively close to the corresponding \emph{simplex} coordinates of $\tilde{z}$. As we discuss later in this section, within a constant stable region about $\tilde{z}$, we can approximate KL divergences by an appropriate notion of a \emph{local} (or, reweighted) Euclidean norm. Equipped with this definition, the goal of this section is to describe how to approximately implement the proximal step $\proxStepSimpleZ{z_c}{\alpha}{\nabla_\pm f_A }{\cZ'}$ to high accuracy, as described in the following definition. 

\begin{definition}[Suproblem proximal step]\label{def:subproblem} In the $(A, z_c, \cZ', \alpha, \epsilon)$-subproblem proximal step, we are given a matrix $A \in \R^{m \times n}$, a center point $z_c \in \cZint$, a regularization parameter $\alpha > 0$, and $\cZ' = (\cX' \times \cY') \subset \cZint$ such that $\cZint \subset \cB_{c_1, \tilde{z}}$ for some $\tilde{z} \in \cZint$, and we must compute a point $z \in \mathcal{Z}$ such that
\begin{align}\label{eq:form-h-inner-loop}
       \breg{\proxStepSimpleZ{z_c}{\alpha}{\nabla_\pm f_A }{\cZ'}}{z} \leq \epsilon. 
\end{align}
\end{definition}

Recall from Section~\ref{sec:prelim} that $\proxStepSimpleZ{z_c}{\alpha}{\nabla_\pm f_A }{\cZ'}$ is the unique exact solution of $\min_{x \in \xset'} \max_{y \in \yset'} y^\top A x + \alpha \xbreg{\zcenterx}{x} - \alpha \ybreg{\zcentery}{y}$ (recall Definition~\ref{def:epsilon-solution}).
Note that unlike the proximal mappings considered earlier in this paper (e.g., Section~\ref{sec:l2l2}), the argument in the $\prox$ notation of \eqref{eq:form-h-inner-loop} is the \emph{full} gradient mapping $\nabla_\pm f$ (as opposed to being evaluated at a point, e.g., $\nabla_\pm f(z)$) and hence in general, \eqref{eq:form-h-inner-loop} cannot be implemented in $\Tilde{O}(1)$ matrix-vector queries exactly. Thus, in this section, our goal is to present an algorithm that computes $z \in \cZ'$ satisfying \eqref{eq:form-h-inner-loop} using few matrix-vector queries. In particular, our algorithm computes $\proxStepSimpleZ{z_c}{\alpha}{\nabla_\pm f_A }{\cZ'}$ to \emph{high accuracy} in the sense that its matrix-vector query complexity scales \emph{polylogarithmically} in the accuracy parameter $1/\epsilon$ (where $\epsilon$ is as in \eqref{eq:form-h-inner-loop}). 

Our approach builds on the smooth-until-proven-guilty techniques presented in Section~\ref{sec:l2l2}; however, the algorithms presented in this section require more care due to the non-Euclidean nature of the geometries in the $\ell_2$-$\ell_1$ and $\ell_1$-$\ell_1$ setup. In Section~\ref{sec:smsug-composite mirror prox steps}, we will discuss the smooth-until-proven-guilty mirror prox steps which are the main subroutine of our algorithm. Then, in Section~\ref{subsec:smooth-until-proven-guilty-strongly-monotone} we will use these steps to construct an iterative algorithm for solving \eqref{eq:form-h-inner-loop} to high accuracy. 

\subsection{Strongly monotone smooth-until-guilty composite mirror prox steps}\label{sec:smsug-composite mirror prox steps}

In order to obtain matrix-vector query complexities scaling polylogarithmically in $1 / \epsilon$, we leverage that solving subproblems of the form of Definition~\ref{def:subproblem} is equivalent to solving a variational inequality in a \emph{strongly monotone} operator (see Section~\ref{sec:prelim}). To this end, we first adapt the analysis from Section~\ref{subsec:mirror-prox-step} to derive a \emph{single} strongly monotone smooth-until guilty mirror prox step for the operator $\nabla_{\pm} (f_{B} + f_C) + \alpha \nabla \breg{z_c}{\cdot}$, where $B, C \in \R^{m \times n}$, $\alpha > 0$, and $z_c \in \cZint$. We assume access to $B$ via a matrix-vector oracle (Definition~\ref{def:mat-vec-oracle}). Consistent with Section~\ref{subsec:mirror-prox-step}, we assume explicit access to $z_c, \alpha,$ and $C$. Throughout this subsection (Section~\ref{sec:smsug-composite mirror prox steps}), we use $z^\star$ to denote the exact value of the following proximal mapping:
\begin{align}\label{eq:form-h-inner-loop-step-first}
       z^\star \defeq \proxStepSimpleZ{z_c}{\alpha}{\nabla_\pm f_{B+C} }{\cZ'}, 
\end{align}
where $\cZ'$ is as in Definition~\ref{def:subproblem}. 

Our strongly monotone smooth-until-guilty composite mirror prox steps leverage the fact that within $\cZint \subset \cB_{c_1, \tilde{z}}$, KL divergences between simplex vectors can be approximated by a \emph{local} (or, reweighted) norm. More concretely, for any constant $c_1 \geq 1$, $\tilde{x} \in \Delta^n_{>0}$, and  $x, x', \xground \in \{\bar{x} \in \Delta^n : \bar{x} \approx_{c_1} \tilde{x}\}$, the following lemma shows that 
\begin{align}\label{eq:local-norm}
    \KL(x || x') \approx_{q_{c_1}} \normInline{x - x'}_{\xground^{-1}}^2 \defeq (x-x')^\top \diag(\xground)^{-1} (x-x')
\end{align}
for an appropriate constant $q_{c_1}$, which depends on $c_1$. 

\begin{lemma}\label{corr:stable-balls} Let $c \geq 1$ be an absolute constant and $\tilde{x} \in \Delta^n_{>0}$. Let $x, x', \xground \in \{\bar{x} \in \Delta^n_{>0}: \bar{x} \approx_c \tilde{x}\}$. Then, $\normInline{{x-x'}}_{\xground^{-1}}^2 \approx_{q_c} \KL(x'|| x)$, 
where 
\begin{align}\label{eq:c2}
    q_c \defeq c^2 \cdot \max\paren{ \paren{\int_{0}^1 \frac{1-t}{1-t + t c^4} dt}^{-1},  \paren{\int_{0}^1 \frac{1-t}{1-t + t /c^4} dt }}. 
\end{align}
\end{lemma}
\begin{proof} We have that for all $i \in [n]$,
\begin{align*}
 [\tilde{x}]_i / c \leq [x]_i \leq c [\tilde{x}]_i, \text{~~~} 
[\tilde{x}]_i / c \leq [x']_i \leq c [\tilde{x}]_i, \text{~~and~~}
 [\tilde{x}]_i / c \leq [\xground]_i \leq c [\tilde{x}]_i. 
\end{align*}
Thus, 
\begin{align}\label{eq:previous}
     [\xground]_i / c^2 \leq [\tilde{x}]_i / c &\leq [x]_i \leq c [\tilde{x}]_i \leq c^2[\xground]_i, \\ 
    [\xground]_i / c^2 \leq [\tilde{x}]_i / c &\leq [x']_i \leq c [\tilde{x}]_i \leq c^2[\xground]_i. \notag
\end{align}
By a similar argument, 
\begin{align}\label{eq:bound-for-integral}
    [x]_i / c^4 \leq [\xground]_i / c^2 \leq [x']_i \leq c^2[\xground]_i \leq c^4 [x]_i. 
\end{align}

Consider the dgf $\entropy(x) = \sum_{i \in [n]} [x]_i \log [x]_i$ (i.e., the negative entropy function) and note that 
\begin{align*}
    \KL(x' || x) = V^e_x(x') = e(x') - e(x) + \nabla e(x)^\top (x'-x) = \int_0^1 (1-t) (x'-x)^\top \nabla^2 e(x_t) (x'-x) dt, 
\end{align*}
where $x_t = tx + (1-t) x'$ and for any $a \in \Delta^n$, the $(i,j)$-th entry of the Hessian $\nabla^2 e(a)$ is given by 
\begin{align*}
    \paren{\nabla^2 e(a)}_{i,j} = \begin{cases}
        [a]_i^{-1}, & i = j, \\
        0, & \text{ otherwise}.
    \end{cases}
\end{align*}
By \eqref{eq:bound-for-integral}, we have that 
\begin{align*}
    [(1-t) + t / c^4] [x]_i \leq [x_{t}]_i &= (1-t)[x]_i + t [x']_i\leq [(1-t) + t c^4] [x]_i,
\end{align*}
and consequently, for any $t \in (0, 1)$, we have that
\begin{align*}
    \frac{1}{[(1-t) + t / c^4] [x]_i} \geq \frac{1}{[x_{t}]_i } &\geq \frac{1}{[(1-t) + t c^4] [x]_i}.
\end{align*}
Thus,
\begin{align*}
    \normInline{x - x'}_{x^{-1}} \cdot \int_{0}^1 \frac{1-t}{1-t + t c^4} dt \leq V^e_x(x') \leq \normInline{x - x'}_{x^{-1}} \cdot \int_{0}^1 \frac{1-t}{1-t + t / c^4} dt. 
\end{align*}
Finally, by \eqref{eq:previous}, 
\begin{align*}
   \frac{1}{c^2[\xground]_i} \leq \frac{1}{[x]_i} \leq \frac{c^2}{ [\xground]_i}
\end{align*}
and consequently,
\begin{align*}
    \normInline{x - x'}_{\xground^{-1}} /c^2 &= \sum_{i \in [n]} [x-x']_i^2 \frac{1}{c^2[\xground]_i}\leq \sum_{i \in [n]} [x-x']_i^2 [x]_i^{-1} = \normInline{x - x'}_{x^{-1}} \\
    &\leq \sum_{i \in [n]} [x-x']_i^2 \frac{c^2}{[\xground]_i} = c^2 \normInline{x - x'}_{\xground^{-1}}. 
\end{align*}
Thus, $\normInline{x - x'}_{\xground^{-1}} / q_{c} \leq \KL(x'|| x) \leq q_c \normInline{x - x'}_{\xground^{-1}}$ as desired.
\end{proof}

\begin{restatable}{remark}{integral}\label{remark:integral}  We briefly remark that $q_c$ in \eqref{eq:c2} is always computable in closed form, and hence $q_c$ is always well-defined. Namely, for any $C \geq 1$,
\begin{align*}
    \int_{0}^1 \frac{1-t}{1-t + t C} dt = 
    \begin{cases}
        (C\log(C) - (C-1))/(C-1)^2,  & C \neq 1, \\
        1/2, & C = 1.
    \end{cases} 
\end{align*}
\end{restatable}
\begin{proof} For $C = 1$, the integral is trivially $\int_{0}^1 (1-t) dt = 1/2$. For $C \neq 1$, we can use the change of variables $x = 1 + t(C-1)$ and note that 
\begin{align*}
    \frac{dx}{dt} &= (C-1)\enspace\text{,}\enspace
    t = \frac{x-1}{C-1}\enspace\text{, and }\enspace
    1-t = \frac{C-x}{C-1}, 
\end{align*}
so that 
\begin{align*}
    \int_0^1 \frac{1-t}{1-t + tC} dt  = \int_0^1 \frac{1-t}{1+t(C-1)} dt  =  \int_1^C \frac{C-x}{(C-1) x} \frac{dx}{C-1} = \frac{1}{(1-C)^2} \paren{C \log(C) - (C-1)}. 
\end{align*}
\end{proof}

Thus, whenever $x, x', \xground \in \{\bar{x} \in \Delta^n : \bar{x} \approx_c \tilde{x}\}$, the KL divergence $\KL(x || x')$ can be approximated (to multiplicative $q_c$ accuracy) by the local norm $\normInline{x - x'}_{\xground^{-1}}$. In order to simplify this local norm notation, one can naturally consider the change of variables under which the diagonal reweighting in \eqref{eq:local-norm} becomes the identity. That is, 
\begin{align*}
\normInline{x - x'}_{\xground^{-1}} = \normInline{\diag(\xground)^{-1/2} x - \diag(\xground)^{-1/2} x'}_2. 
\end{align*}
Consequently, to simplify this notation, in the remainder of this section, we introduce the following reweighting matrices, which simplify the notation throughout our technical results. 
\begin{definition}[Additional setup details]\label{def:setup-details} For any $z' \in \cZint$ we define 
\begin{align*}
    N_{\cX,z'} = \begin{cases}
        I, & \cX = \B^n, \\
        \diag(z'\x)^{-1/2}, & \cX = \Delta^n,
    \end{cases} \text{~~and~~}
    N_{\cY,z'} = \diag(z'\y)^{-1/2}.
\end{align*}
Further, for any $z \in \cZ, B \in \R^{m \times n}$, we define the following shorthands,
\begin{align*}
    \localize{z}{z'} \defeq (N_{\cX,z'} z\x, N_{\cY,z'} z\y) \in \R^d, \text{~~}
    \ground{B}{z'} \defeq N_{\cY,z'}^{-1}  B  N_{\cX,z'}^{-1} \in \R^{m\times n}, \text{~and~}
    \unground{B}{z'} \defeq N_{\cY,z'}  B  N_{\cX,z'} \in \R^{m \times n}.
\end{align*}
\end{definition}

Now, the following Algorithm~\ref{alg:mirror-prox-iteration-sm} implements a strongly-montone smooth-until-guilty mirror prox step (which consists of two composite proximal steps) and outputs the results along with a flag (which takes value $\guilty$ or $\smooth$) and a matrix $D \in \R^{m \times n}$ computed using the $\judge$~subroutine as introduced in Definition~\ref{def:smooth-guilty-judge}. As in Section~\ref{sec:l2l2}, we use the $\judge$~subroutine to identify whether or not the prox steps for computing $w, z'$ (Lines~\ref{line:first-step-stronglymonotone} and \ref{line:second-step-stronglymonotone}) meet the smoothness requirements to make progress in converging towards $z^\star$. However, in this case the vectors $\zbar^1$ and $\zbar^2$ (Line~\ref{line:strongly-monotone-test-vectors}) which we use to check for a large matrix component have been rescaled using the local norm point $\zground$ due to the fact that we are now using the local norm \eqref{eq:local-norm} as opposed to the standard Euclidean norm.

\RestyleAlgo{ruled}
\SetKwComment{Comment}{/* }{ */}
\begin{algorithm2e}[ht]
    \DontPrintSemicolon
\caption{Smooth-until-guilty Composite Strongly Monotone Mirror Prox Step $\SUGSMStep(z, z_c, \zground, \tau, \alpha, B, C, c_1, \cZ'; \judge)$}
\label{alg:mirror-prox-iteration-sm}
\KwInput{Input points $z, \zground \in \cZ'$, center point $z_c \in \cZint$, smoothness threshold $\tau > 0$, regularization level $\alpha > 0$, constant $c_1 > 0$, and a convex and compact subset $\cZ' \subset \cB_{c_1, \Tilde{z}}$}
\KwInput{Matrix-vector oracle for a matrix $B \in \R^{m \times n}$.} 
\KwInput{Matrix $C \in \R^{m \times n}$.} 
\KwParameter{Judge function $\judge$} 
\tcp{Perform a strongly monotone mirror prox step, consisting of two composite proximal steps }

$c_2 \gets q_{c_1}$ where $q_{c_1}$ is as defined in \eqref{eq:c2} \label{line:c2-ref}\; 
$w \gets \proxStepSimpleZ{z}{\tau}{\nabla_\pm f_B(z) + \nabla_\pm f_C + \alpha \nabla \breg{z_c}{\cdot} }{\cZ'}$ \label{line:first-step-stronglymonotone}\;
$z' \gets \proxStepFullZ{z}{w}{\tau}{\alpha}{\nabla_\pm f_{B+C}(w) + \alpha \nabla \breg{z_c}{w} }{\cZ'}$ \label{line:second-step-stronglymonotone}\;
$z^1 \gets ({w\x - z'\x}, {w\y - z\y})$ and $z^2 \gets ({z\x - w\x}, {w\y - z'\y})$\; 
\tcp{ Localize all quantities to ${\zground}$  (Definition~\ref{def:setup-details})}
$\bar{z}^1 \gets \localize{z_1}{\zground}$ and $\bar{z}^2  \gets \localize{z_2}{\zground}$ \label{line:strongly-monotone-test-vectors}\; 
$(\flag, D) \gets \judge(\bar{z}^1, \bar{z}^2, \ground{B}{\zground}, c_2\tau)$ \tcp*{Use a judge to evaluate smoothness} 
\Return{$(w, z', \flag, D)$}
\end{algorithm2e}

We prove the following lemma regarding the strongly monotone smooth-until-guilty composite mirror prox steps in Algorithm~\ref{alg:mirror-prox-iteration-sm}.

\begin{lemma}\label{lemma:mirror-prox-step-guarantee-sm} Let $B, C, \alpha, \cZ'$ be as in \eqref{eq:form-h-inner-loop-step-first}, $\tau > 0$, $z_c \in \cZint$, $z, \zground \in \cZ'$ and $c_1 > 1$ be an absolute constant. Let 
\begin{align*}
    (w, z', \flag, D) = \SUGSMStep(z, z_c, \zground, \tau, \alpha, B, C, c_1, \cZ'; \judge),\,\text{(Algorithm~\ref{alg:mirror-prox-iteration-sm})}\,. 
\end{align*}
Then, for $c_2$ as defined in Line~\ref{line:c2-ref} we have that either, 
\begin{align*}
    \breg{z^\star}{z'} \leq \paren{1 + \frac{\alpha}{\tau}}^{-1} \breg{z^\star}{z}, 
\end{align*}
or one of the following must hold: 
\begin{align*}
    {\bar{z}^1\y}{}^\top \ground{B}{\zground} {\bar{z}^1\x} \geq c_2 \tau \normInline{\bar{z}^1\y}_2 \normInline{\bar{z}^1\x}_2
    \enspace\text{ or }\enspace
    {\bar{z}^2\y}{}^\top \ground{B}{\zground} {\bar{z}^2\x} \geq c_2 \tau \normInline{\bar{z}^2\y}_2 \normInline{\bar{z}^2\x}_2\,.
\end{align*}
\end{lemma}

\begin{proof} For notational convenience, let $\psi: z \mapsto  \nabla_\pm f_C(z) + \alpha \nabla \breg{z_c}{z}$. Applying Definition~\ref{def:proximal-mappings} to each composite proximal step (Lines~\ref{line:first-step-stronglymonotone} and~\ref{line:second-step-stronglymonotone}) individually, we have that for all $u, u' \in \cZ$,
\begin{align}
    \frac{1}{\tau} \nabla_\pm f_B(z)^\top (w-u') + \frac{1}{\tau}\psi(w)^\top (w-u') &\leq \breg{z}{u'} - \breg{w}{u'} - \breg{z}{w}\,\text{, and} \label{eq:first-inequality-new}\\ \notag
    \frac{1}{\tau} \nabla_\pm f_B(w)^\top (z'-u) + \frac{1}{\tau} \psi(w)^\top(z' - u) &\leq \breg{z}{u} - \breg{z'}{u} - \breg{z}{z'} + \frac{\alpha}{\tau} \paren{ \breg{w}{u} + \breg{z'}{u}}, 
\end{align}
where in the second line, we used that $\breg{w}{z'} \geq 0$. Setting $u' = z'$ in \eqref{eq:first-inequality-1} and summing yields,
\begin{align*}
    &\frac{1}{\tau} \Brac{\nabla_\pm f_B(z)^\top (w-z') + \psi(w)^\top (w-u) + \nabla_\pm f_B(w)^\top (z'-u)} \\
    \leq &\breg{z}{u} - \breg{z'}{u} - ( V_{w}^r(z') + \breg{z}{w} ) + \frac{\alpha}{\tau} \paren{ \breg{w}{u} + \breg{z'}{u}}. 
\end{align*}

Now, consider two cases. First, suppose that 
\begin{align}\label{eq:successful-step-new}
    (\nabla_\pm \bilinear{B}(w) - \nabla_\pm \bilinear{B}(z))^\top(w - z') \leq \tau(\breg{w}{z'} + \breg{z}{w}). 
\end{align}
Then, 
\begin{align*}
    &\frac{1}{\tau} \Brac{\nabla_\pm f_B(z)^\top (w-z') + \psi(w)^\top (w-u) + \nabla_\pm f_B(w)^\top (z'-u)} \\
    \leq& \breg{z}{u} - \breg{z'}{u} - ( V_{w}^r{z'} + \breg{z}{w} ) + \frac{\alpha}{\tau} \paren{ \breg{w}{u} + \breg{z'}{u}}. 
    \\
    \leq& \breg{z}{u} - \breg{z'}{u} -\frac{1}{\tau} (\nabla_\pm f_B(w) - \nabla_\pm f_B(z))^\top(w - z') + \frac{\alpha}{\tau} \paren{ \breg{w}{u} - \breg{z'}{u}}. 
\end{align*}
Thus, we can rearrange terms to obtain 
\begin{align*}
       &\frac{1}{\tau} \Brac{(\nabla_\pm \bilinear{B}(w) +\psi(w))^\top (w-u)} \\
       =&\frac{1}{\tau} \Brac{\nabla_\pm f_B(z)^\top (w-z') + \psi(w)^\top (w-u) + \nabla_\pm f_B(w)^\top (z'-u)}  \\ 
       &\spaceeq+  \frac{1}{\tau} \Brac{(\nabla_\pm f_B(w) - \nabla_\pm f_B(z))^\top(w - z')} + \frac{\alpha}{\tau} \paren{ \breg{w}{u} - \breg{z'}{u}} \\
       \leq & \breg{z}{u} - \breg{z'}{u} + \frac{\alpha}{\tau} \paren{ \breg{w}{u} - \breg{z'}{u}}. 
\end{align*}

Substituting $u = z^\star$ into the display above, we obtain 
\begin{align*}
   \frac{1}{\tau} \Brac{(\nabla_\pm \bilinear{B}(w)+\psi(w))^\top (w-z^\star)} \leq  \breg{z}{z^\star} - \breg{z'}{z^\star} + \frac{\alpha}{\tau} \paren{ \breg{w}{z^\star} - \breg{z'}{z^\star}}. 
\end{align*}
Rearranging terms implies, 
\begin{align}\label{eq:preceding_step_here}
    \frac{1}{\tau} \Brac{(\nabla_\pm \bilinear{B}(w)+\psi(w))^\top (w-z^\star)} - \frac{\alpha}{\tau} \breg{w}{z^\star} \leq \breg{z}{z^\star} - \paren{1 + \frac{\alpha}{\tau}} \breg{z'}{z^\star}. 
\end{align}
Now, by the definition of $z^\star$, we have that
\begin{align*}
     \frac{1}{\tau} \Brac{(\nabla_\pm \bilinear{B}(z^\star)+\psi(z^\star))^\top (w - z^\star)} \geq 0. 
\end{align*}
Thus, 
\begin{align*}
     &\frac{1}{\tau} \Brac{(\nabla_\pm \bilinear{B}(w)+\psi(w))^\top (w-z^\star)} - \frac{\alpha}{\tau} \breg{w}{z^\star} \\
     \geq& \frac{1}{\tau} \Brac{\paren{(\nabla_\pm \bilinear{B}(w)+\psi(w)) - (\nabla_\pm \bilinear{B}(z^\star)+\psi(z^\star))}^\top (w-z^\star)} - \frac{\alpha}{\tau} \breg{w}{z^\star} \\
     \geq& \frac{\alpha}{\tau} \breg{w}{z^\star}  - \frac{\alpha}{\tau} \breg{w}{z^\star}  = 0, 
\end{align*}
where the last line follows by linearity and strong monotonicity. Combining with \eqref{eq:preceding_step_here}, we can conclude that 
\begin{align*}
    \breg{z'}{z^\star} \leq \paren{1 + \frac{\alpha}{\tau}}^{-1} \breg{z}{z^\star}. 
\end{align*}

Now, on the other hand, suppose instead that \eqref{eq:successful-step-new} does not hold. Then, it must be the case that 
\begin{align}\label{eq:unsuccessful-step-new-here}
    \frac{1}{\tau} (\nabla_\pm f_B({w}) - \nabla_\pm f_B(z))^\top({w} - z') &> \breg{w}{z'} + \breg{z}{w}. 
\end{align}
Recall that for any $z' \in \cZ$, $\nabla_\pm f_B(z') = (B^\top z'\y, -B z'\x)$. Thus, we observe that 
\begin{align*}
    (\nabla_\pm f_B({w}) - \nabla_\pm f_B(z))^\top ({w} - {z'}) =({{w}}\y - {z}\y)^\top B ({{w}}\x - {z'\x}) + ({{w}}\y - {z'\y})^\top B ({z}\x - {{w}}\x). 
\end{align*}
Thus, expanding out \eqref{eq:unsuccessful-step-new-here} implies that 
\begin{align}\label{eq:unsuccessful-step-implication}
   ({{w}}\y - {z\y})^\top B ({{w}}\x - {z'\x}) + 
    ({{w}}\y - {z'\y})^\top B ({z}\x - {{w}}\x) > \tau \cdot \paren{ \breg{w}{z'} + \breg{z}{w} }, 
\end{align}
or equivalently, by Definition~\ref{def:setup-details}, we can write
\begin{align}\label{eq:convert-it}
   &\localize{{{w}}\y - {z\y}}{\zground}^\top \ground{B}{\zground} \localize{{{w}}\x - {z'\x}}{\zground} + 
    \localize{{{w}}\y - {z'\y}}{\zground}^\top \ground{B}{\zground} \localize{{z}\x - {{w}}\x}{\zground} \\
    =& {\bar{z}^2\y}{}^\top \ground{B}{\zground} {\bar{z}^2\x} + {\bar{z}^1\y}{}^\top \ground{B}{\zground} {\bar{z}^1\x}  > \tau \cdot \paren{ \breg{w}{z'} + \breg{z}{w} }. 
\end{align}
Note that $z, z', w \in \cZ' \subset \cB_{c_1, \zground}$. Consequently, by Lemma~\ref{corr:stable-balls} and the definition of $c_1$ in the statement of this lemma, for $c_2 \defeq c_1^2 q_{c_1^4}$, we have 
\begin{align*}
    \breg{w}{z'} + \breg{z}{w} &\geq \begin{cases}
        c_2 \paren { \normInline{{z^1\x}}_{{\zground}\x^{-1}}^2 + \normInline{{z^1\y}}_{{\zground}\y^{-1}}^2 + \normInline{{z^2\x}}_{{\zground}\x^{-1}}^2 + \normInline{{z^2\y}}_{{\zground}\y^{-1}}^2 }, & \cX = \Delta^n , \\
       \normInline{{z^1\x}}^2 + c_2 \normInline{{z^1\y}}_{{\zground}\y^{-1}}^2 + \normInline{{z^2\x}}^2 + c_2 \normInline{{z^2\y}}_{{\zground}\y^{-1}}^2, & \cX = \B^n .
    \end{cases}
\end{align*}
Consequently, noting that $q_c \geq 1 $ for all $c$ (recall equation \eqref{eq:c2}), we have
\begin{align*}
    \breg{w}{z'} + \breg{z}{w} &\geq 
    c_2 \paren { \normInline{\localize{z^1\x}{{\zground}\x}}_2^2 + \normInline{\localize{z^1\y}{{\zground}\y}}_2^2 + \normInline{\localize{z^2\x}{{\zground}\x}}_2^2 + \normInline{\localize{z^2\y}{{\zground}\y}}_2^2 } \\
    &\geq c_2 \paren { \normInline{\localize{z^1\x}{{\zground}\x}}_2\normInline{\localize{z^1\y}{{\zground}\y}}_2 + \normInline{\localize{z^2\x}{{\zground}\x}}_2 \normInline{\localize{z^2\y}{{\zground}\y}}_2 } \\
    &= c_2 \paren { \normInline{\bar{z}^1\y}_2 \normInline{\bar{z}^1\x}_2,  + \normInline{\bar{z}^2\y}_2 \normInline{\bar{z}^2\x}_2,  }, 
\end{align*}
where the second line used Definition~\ref{def:setup-details} and the third line used the usual property that $(a^2 + b^2) \geq 2ab \geq ab$ for any $a, b \in \R_{\geq 0}$. Combining with \eqref{eq:convert-it} completes the proof as one of the following holds,
\begin{align*}
    {\bar{z}^1\y}{}^\top \ground{B}{\zground} {\bar{z}^1\x} \geq c_2 \tau \normInline{\bar{z}^1\y}_2 \normInline{\bar{z}^1\x}_2,\enspace\text{ or }\enspace
    {\bar{z}^2\y}{}^\top \ground{B}{\zground} {\bar{z}^2\x} \geq c_2 \tau \normInline{\bar{z}^2\y}_2 \normInline{\bar{z}^2\x}_2\,. 
\end{align*}
\end{proof}

We also obtain the following immediate corollary of Lemma~\ref{lemma:mirror-prox-step-guarantee-sm}. 
\begin{corollary}\label{corr:lemma-sm-step} Let $B, C, \alpha, \cZ'$ be as in \eqref{eq:form-h-inner-loop-step-first}, $\tau > 0$, $p \in [1, \infty)$ and $\judge$ be a $p$-smooth-guilty judge. Let $z_c \in \cZint$, $z, \zground \in \cZ'$, $c_1 > 1$ be an absolute constant, and
\begin{align*}
    (w, z', \flag, D) = \SUGSMStep(z, z_c, \zground, \tau, \alpha, B, C, c_1, \cZ'; \judge), 
\end{align*}
(Algorithm~\ref{alg:mirror-prox-iteration-sm}) and $c_2$ as defined in Line~\ref{line:c2-ref}. Then, the algorithm runs in $O(1)$-matrix-vector queries. Moreover, if $\flag = \smooth$, then $\breg{z^\star}{z'} \leq \paren{1 + \frac{\alpha}{\tau}}^{-1} \breg{z^\star}{z}.$ Otherwise, $\flag = \guilty$ and $\normInline{\ground{B}{\zground} - D}_{\cS_p}^p \leq \normInline{\ground{B}{\zground}}_{\cS_p}^p - c_2 \tau^p.$
\end{corollary}
\begin{proof} The proof follows immediately from Lemma~\ref{lemma:mirror-prox-step-guarantee-sm} and the definition of a $p$-smooth-guilty judge (Definition~\ref{def:smooth-guilty-judge}). The query complexity follows from Definition~\ref{def:smooth-guilty-judge} and the fact that all remaining pseudocode lines of Algorithm~\ref{alg:mirror-prox-iteration-sm} require at most $O(1)$-matrix-vector queries. 
\end{proof} 

\subsection{Strongly monotone smooth-until-guilty composite mirror prox}\label{subsec:smooth-until-proven-guilty-strongly-monotone}

In this section, we show how to use the smooth-until-guilty composite mirror prox steps from Section~\ref{sec:smsug-composite mirror prox steps} to solve minimax problems of the form \eqref{eq:form-h-inner-loop}. The pseudocode is shown in Algorithm~\ref{alg:mirror-prox-sug-strongly-monotone}. 

Similar in spirit to Algorithm~\ref{alg:mirror-prox-sug-l2l2},  the algorithm initializes a \emph{model} $M_0 = 0$, matrix $A_0 = A$ (implicitly), and sets $k,j = 0$. The algorithm then runs multiple iterations of a while loop. In each iteration, the algorithm performs a strongly monotone $\SUGSMStep$ step (Algorithm~\ref{alg:mirror-prox-iteration-sm}) with the current $A_k$ and $M_k$. If the step returns $\guilty$, the algorithm \emph{updates the model} by  $M_{k+1} = M_k + \unground{D_k}{\zground}$ and $A_{k+1} \gets A - M_{k+1}$ and also increments the number of \emph{model-update iterations} (stored in the variable $k$). Otherwise, if the step returns $\smooth$, the number of progress iterations (stored in the variable $j$) is incremented. 

\RestyleAlgo{ruled}\label{alg:mirror-prox-sug-strongly-monotone}
\DontPrintSemicolon
\SetKwComment{Comment}{/* }{ */}
\begin{algorithm2e}[h!]
\caption{Smooth-Until-Proven-Guilty Composite Strongly Monotone Mirror Prox $\SUGStronglyMonotoneMirrorProx(z_c, \zground, \tau, \alpha, \epsilon, A, M, c_1, \cZ'; \judge)$}
\KwInput{Center point $z_c \in \cZint$, local norm point $\zground \in \cZ'$, smoothness threshold $\tau > 0$, regularization level $\alpha > 0$, constant $c_1 > 0$, and a convex and compact subset $\cZ' \subset \cB_{c_1, \Tilde{x}}$}
\KwInput{Matrix-vector oracle for $A$}
\KwInput{Model $M \in \R^{m \times n}$}
\KwParameter{Judge function $\judge$.} 
$j \gets 0, k \gets 0$\;
$z^0 \gets \argmin_{z \in \cZ'}r(z)$\; 
$A_0 \gets A - M$, $M_0 \gets M$ \tcp*{$\Range$ is as defined in the setup (see Table~\ref{table:setups})}
\While{$j \leq J$ where $J \defeq \ceil{(1 + {\tau/\alpha})\log(\Range/\epsilon)}$ } { 
    \BlankLine
    \tcp{  Strongly monotone mirror prox step }
    $(w^{j}, z^{j+1}, \modelUpdateStep, D) \gets \SUGSMStep(z^{j}, z_c, \zground, \tau, \alpha, A_k, M_k, c_1; \judge)$ \label{line:mirror-prox-step-sm-final}
    \BlankLine
    \If{$\flag = \guilty$}
    {
        \tcp{Move matrix $D$ into the model after undoing the local norm scaling}
        $D_{k+1} \gets D$ \tcp*{Keep track of the model updates, for analysis} \label{algline:SUG-SM-MP-update-D}
        $M_{k+1} \gets M_k + \unground{D_{k+1}}{\zground}$ \tcp*{Addition implemented explicitly, since $M_k, D_{k+1}$ are known explicitly}
        $A_{k+1} \gets A - M_{k+1}$ \tcp*{Subtraction implemented implicitly, since $A$ is not known explicitly}
        $k \gets k + 1$ \tcp*{Increment the number of model-update iterations (for analysis only)} \label{algline:SUG-SM-MP-update-k}
    }
    \lElse(\tcp*[f]{Otherwise, increment the number of progress iterations $j$ and store $z^{t+1}$}){
        $j \gets j + 1$ \label{algline:SUG-SM-MP-update-j}
    }
}
\Return{$(z^{J}, k, M_{k})$} \label{algline:SUG-SM-MP-return}

\end{algorithm2e}

We obtain the following correctness guarantee. 

\begin{theorem}[Correctness of Algorithm~\ref{alg:mirror-prox-sug-strongly-monotone}]\label{thm:main-inner-loop-convergence} Let $\judge$~be a $p$-smooth-guilty judge. Let $\epsilon > 0$. Let $z_c \in \cZint$ and $\zground \in \cZ' \subset \cB_{c_1, \Tilde{z}}$ for some absolute constant $c_1 > 1$ and $\Tilde{z} \in \cZint$. Let $A, M \in \R^{m \times n}$, and $\tau, \alpha > 0$. 
Suppose that Algorithm~\ref{alg:mirror-prox-sug-strongly-monotone} terminates and returns 
\begin{align*}
    (z, \cdot, \cdot) := \SUGStronglyMonotoneMirrorProx(z_c, \zground, \tau, \alpha, A, M, c_1, \epsilon, \cZ'; \judge), 
\end{align*}
Then $\breg{\proxStepSimpleZ{z_c}{\alpha}{\nabla_\pm f_A }{\cZ'}}{z} \leq \epsilon$. 
\end{theorem}
\begin{proof} For notational convenience, let $z^\star = \proxStepSimpleZ{z_c}{\alpha}{\nabla_\pm f_A }{\cZ'}$. First, it is easy to see that for each $k \in [K]$, we maintain the invariant that 
\begin{align*}
    A_k + M_k = A. 
\end{align*}
Consequently, by Corrollary~\ref{corr:lemma-sm-step}, we have that for each $j \in [J]$, 
\begin{align*}
    \breg{{z}^\star}{{z}^{j}} \leq \paren{1 + \frac{\alpha}{\tau}}^{-1} \breg{z^\star}{{z}^{j-1}}, 
\end{align*}
By recursion, we conclude that 
\begin{align*}
    \breg{z^\star}{z^{J}} \leq \paren{1 + \frac{\alpha}{\tau}}^{-J} \breg{z^\star}{\bar{z}^{0}} = \paren{1 + \frac{\alpha}{\tau}}^{-J} \breg{z^\star}{{z}^{0}} \leq \paren{1 + \frac{\alpha}{\tau}}^{-J} \Range \leq \epsilon,
\end{align*}
whenever $J \geq (1 + {\tau/\alpha})\log(\Range/\epsilon).$
\end{proof}

\begin{theorem}[Complexity of Algorithm~\ref{alg:mirror-prox-sug-strongly-monotone}]\label{thm:main-inner-loop-complexity} Suppose the assumptions of Theorem~\ref{thm:main-inner-loop-convergence} hold. Then,  
\begin{align*}
    (z, K, M') := \SUGStronglyMonotoneMirrorProx(z_c, \zground, \tau, \alpha, A, M, \epsilon, c_1, \cZ'; \judge)
\end{align*}
(Algorithm~\ref{alg:mirror-prox-sug-strongly-monotone})
terminates after making at most $O(K + (\tau/\alpha) \log(1/\epsilon))$-matrix-vector queries, and moreover, $\normInline{\ground{A - M'}{\zground}}_{\cS_p}^p \leq \normInline{\ground{A - M}{\zground}}_{\cS_p}^p - c_2 K\tau^p.$ 

\end{theorem}

\begin{proof} Recall that by Corollary~\ref{corr:lemma-sm-step}, for each iteration $t$ such that $\flag$ is guilty in Line~\ref{line:mirror-prox-step-sm-final}, we have that 
\begin{align}\label{eq:t}
    \normInline{\ground{A_k}{\zground} - D}_{\cS_p}^p \leq\normInline{\ground{A_k}{\zground}}_{\cS_p}^p - c_2 \tau^p
\end{align}
Now, we prove, by induction on $k$, that 
\begin{align}\label{eq:induct-K}
    \normInline{\ground{A_{K}}{\zground}}_{\cS_p}^p \leq \normInline{\ground{A_{0}}{\zground}}_{\cS_p}^p - c_2 K \tau^p. 
\end{align}
In the base case of $k = 1$, note that by \eqref{eq:t},
\begin{align*}
    \normInline{\ground{A_{1}}{\zground}}_{\cS_p}^p &= \normInline{\ground{A_{0}}{\zground} - \ground{M_1}{\zground}}_{\cS_p}^p = \normInline{\ground{A_0 - \unground{{D}_1}{\zground}}{\zground}}_{\cS_p}^p = \normInline{\ground{A_{0}}{\zground} - D}_{\cS_p}^p \leq \normInline{\ground{A_0}{\zground}}_{\cS_p}^p - c_2\tau^p. 
\end{align*}
In the display above, the third equality holds due to linearity of the operation $(\cdot)_{\zground}$ in Definition~\ref{def:setup-details} and the definition of $D_1$. Now, assume that \eqref{eq:induct-K} is true up to $K-1$. 
\begin{align*}
    \normInline{\ground{A_{K}}{\zground}}_{\cS_p}^p &= \normInline{\ground{A_{K-1} - \unground{{D}_{K}}{\zground}}{\zground}}_{\cS_p}^p = \normInline{\ground{A_{K-1}}{\zground} - D}_{\cS_p}^p \\
    &\leq \normInline{\ground{A_{K-1}}{\zground}}_{\cS_p}^p - c_2\tau^p \leq \normInline{\ground{A_{0}}{\zground}}_{\cS_p}^p - c_2 K\tau^p, 
\end{align*}
where the second equality holds due to linearity of the operation $(\cdot)_{\zground}$ in Definition~\ref{def:setup-details} and the definition of $D_K$, the first inequality holds by \eqref{eq:t}, and the second inequality holds by the inductive hypothesis. The proof now follows from the above line by noting that $A_K = A - M'$, and $A_0 = A - M$ and the complexity bound in \Cref{corr:lemma-sm-step}.  
\end{proof}

We now conclude with a corollary of Theorems~\ref{thm:main-inner-loop-convergence} and \ref{thm:main-inner-loop-complexity} which will be useful for later analysis.

\begin{corollary}
    \label{cor:SM-mirror-prox-complexity-when-passmodel-is-false}
    Let $\judge$~be a $p$-smooth-guilty judge $\epsilon > 0$, $z_c \in \cZint$ and $\zground \in \cZ' \subset \cB_{c_1, \Tilde{z}}$ for some absolute constant $c_1 > 1$ and some $\Tilde{z} \in \cZint$. Let $A \in \R^{m \times n}$, and $\alpha > 0$, $\normInline{\ground{A}{\zground}}_{\cS_p} \leq \zeta$, and $c_2$ be as defined in Line~\ref{line:c2-ref}. Then, 
    \begin{align*}
        (z, \cdot, \cdot) := \SUGStronglyMonotoneMirrorProx(z_c, \zground, \tau, \alpha, A, 0, \epsilon, \cZ'; \judge), 
    \end{align*}
    (Algorithm~\ref{alg:mirror-prox-sug-strongly-monotone}) makes $O(\tau\alpha^{-1}\log(\Range/\epsilon) + \zeta^p\tau^{-p})$-matrix-vector queries and $\breg{\proxStepSimpleZ{z_c}{\alpha}{\nabla_\pm f_A }{\cZ'}}{z} \leq \epsilon$.  
    \end{corollary}
    \begin{proof} The bound follows directly from Theorems~\ref{thm:main-inner-loop-convergence} and~\ref{thm:main-inner-loop-complexity}. We have 
    \begin{align*}
        0 \leq \normInline{\ground{A-M'}{\zground}}_{\cS_p}^p \leq \normInline{A}_{\cS_p}^p - c_2 K \tau^p \leq \zeta^p - c_2 K \tau^p
    \end{align*}
    which implies that $K \leq \zeta^p/(c_2\tau^p)$. By Corollary~\ref{corr:lemma-sm-step}, the total number of matrix-vector queries is at most $K + J  = O(\tau\alpha^{-1}\log(\Range/\epsilon) + \zeta^p/(c_2 \tau^p))$. 
    \end{proof} 
\section{Algorithms for $\ell_2$-$\ell_1$ and $\ell_1$-$\ell_1$ games}\label{sec:elltwoelloneandelloneellone}

In this section we provide algorithms for obtaining an $\epsilon$-solution of the minimax objective
\begin{align}
    \min_{x \in \xset} \max_{y \in \yset} \inbraces{ f_A(x, y) \defeq y^\top A x } \label{eq:target-minimax-prob}.
\end{align}
in the $\ellTwoEllOne$ and $\ellOneEllOne$ setups of Definition~\ref{def:matrix-vector-games-setups}. We operate in these setups throughout the section and if a statement does not explicitly distinguish between them, it applies to both setups.

We provide our algorithms in Section~\ref{subsec:ell2ell1-ell1ell1-putting-all-together} by combining the outer loop of Section~\ref{sec:prox-point-regret} with the subroutine of Section~\ref{sec:innerloops} to implement the $\DAPO$ oracle (Definition~\ref{def:DAPO}) in the context of these setups. But first, in Section~\ref{subsec:suffices-to-truncate} we show that to approximately solve \eqref{eq:target-minimax-prob}, it suffices to approximately solve the same problem except over the truncated constraint sets $\xtrunc$ and $\ytrunc$ given in Definition~\ref{def:matrix-vector-games-setups} and Table \ref{table:setups}, with $\nu^{-1}$ bounded above by a polynomial in $m$, $n$, and $1 / \epsilon$. This truncation facilitates a variety of approximations and technical lemmas used to implement the $\DAPO$ oracle. In Section~\ref{subsec:stability-best-response}, we prove a stability result for regularized versions of \eqref{eq:target-minimax-prob} with truncation, which is key to reducing the implementation of the $\DAPO$ oracle to solving subproblems over multiplicative balls (Definition~\ref{def:stable-ball}), as required by the subroutine of Section~\ref{sec:innerloops}. We give our binary search subroutine in Section~\ref{subsec:binary-search}. This binary search is used to obtain the appropriate level of regularization to implement a \emph{kinetic} $\DAPO$ oracle (recall Definition~\ref{def:DAPO}), which allows us to ultimately apply the iteration guarantee of Lemma~\ref{lem:prox-point-iteration-bound}. Finally, as stated previously, we put everything together in Section~\ref{subsec:ell2ell1-ell1ell1-putting-all-together} and give an algorithm for solving \eqref{eq:target-minimax-prob} in $\tilde{O}(\epsilon^{-8/9})$ queries. Lastly, in Section~\ref{sec:four-fifths}, we demonstrate how to use the techniques developed in Section~\ref{subsec:suffices-to-truncate} through Section~\ref{subsec:ell2ell1-ell1ell1-putting-all-together} to obtain an algorithm for the $\ell_2$-$\ell_1$ setup which runs in $\tilde{O}(\epsilon^{-7/9})$-matrix-vector queries.

We assume throughout Section~\ref{sec:elltwoelloneandelloneellone} unless explicitly stated otherwise
 that the matrix $A$ is normalized appropriately depending on the setup, namely $\norm{A}_{2 \to \infty} = \max_{i \in [m]} \norm{A_{i,:}}_2 \le 1$ in the $\ell_2$-$\ell_1$ setup and $\norm{A}_{\text{max}} \le 1$ in the $\ellOneEllOne$ setup. Regarding the truncation parameter $\nu > 0$, for greater generality we do not fix it up-front, and therefore a lemma or other result which does not make reference to a specific range of values allows for all $\nu$ for which the truncated constraint sets are nonempty.

\subsection{Constraint set truncation}
\label{subsec:suffices-to-truncate}

Recalling the definition of the truncated constraint sets $\xtrunc$ and $\ytrunc$ from Table \ref{table:setups} (and also $\ztrunc \defeq \xtrunc \times \ytrunc$ from Definition~\ref{def:matrix-vector-games-setups}), we first give a truncation argument similar to \cite[Sec.\ 8.1]{carmon2024whole}. Adapting the reasoning of \cite[Thm.\ 8.1]{carmon2024whole}, we show in Lemma~\ref{lem:truncation-for-ell2ell1-ell1ell1} below that it suffices to obtain an approximate solution over the truncated domains with $\nu$ set to an inverse polynomial in the problem parameters. But first, Lemma~\ref{lem:there-is-a-nearby-point-in-trunc-simplex}, which is used in the proof of Lemma~\ref{lem:truncation-for-ell2ell1-ell1ell1}, shows that every point in the simplex can be approximated in $\ell_1$-distance by a point in the truncated simplex.

\begin{lemma}
    \label{lem:there-is-a-nearby-point-in-trunc-simplex}
For any $u \in \simplex^d$ and $\nu \in (0, 1 / (4d))$, there exists $q \in \simplex^d_\nu$ such that $\norm{u - q}_1 \le 2 \nu d$.
\end{lemma}

\begin{proof}
We will give a procedure to construct $q$. Let $J \subseteq [d]$ denote the indices for which the corresponding entries of $u$ are less than $\nu$, i.e., $J \defeq \inbraces{i \in [d] : [u]_i < \nu}$. Then to start, set $[q]_i \gets \nu$ for all $i \in J$, and $[q]_k \gets [u]_k$ for all $k \in [d] \setminus J$. Note that at this stage, we have $\innorm{u - q}_1 \le \nu d$, although $q$ is not necessarily a probability distribution. Toward remedying the latter, note $\sum_{i \in [d] \setminus J} [u]_i \ge 3/4$, since otherwise $u$ would not be a probability distribution as $\sum_{i \in J} [u]_i \le \nu d \le 1/4$. Then $\sum_{i \in [d] \setminus J} ([u]_i - \nu) \ge 3/4 - \nu d \ge 1/2$, meaning at least half of the mass of $u$ lies above the truncation level $\nu$. Thus, we can remove some of the corresponding mass from $q$ to make it into a probability distribution. Formally, note $1 \le \innorm{q}_1 \le 1 + \nu d \le 5/4$. Then, we can iterate over $[d] \setminus J$ and decrease the corresponding coordinates of $q$ one by one down to $\nu$ until $\innorm{q}_1 = 1$ as required. Since $\innorm{q}_1 \le 1 + \nu d$, this final step causes $\innorm{u - q}_1$ to increase by at most $\nu d$, and therefore $\innorm{u - q}_1 \le 2 \nu d$ at the end of this process.
\end{proof}

We now give the main lemma for this section:

\begin{lemma}[Truncation for $\ell_2$-$\ell_1$ and $\ell_1$-$\ell_1$]
    \label{lem:truncation-for-ell2ell1-ell1ell1}
For $\epsilon > 0$ and $0 < \nu \le \frac{\min \inbraces{\epsilon, 1}}{8 \max \inbraces{m, n}}$, any $\epsilon / 2$-solution $z' \in \ztrunc$ of 
\begin{align}
    \min_{x \in \xset_\nu} \max_{y \in \yset_\nu} y^\top A x. \label{eq:truncated-minimax-prob}
\end{align}
is an $\epsilon$-solution of \eqref{eq:target-minimax-prob}.
\end{lemma}

\begin{proof}
Let $\gap(\cdot)$ and $\gaptrunc(\cdot)$ denote the gap functions of \eqref{eq:target-minimax-prob} and \eqref{eq:truncated-minimax-prob} respectively per Definition~\ref{def:epsilon-solution}. Defining for all $x \in \xset$ and $y \in \yset$ the functions
\begin{align*}
    g(x) \defeq \max_{y \in \yset} f_A(x, y), ~~ g\trunc(x) \defeq \max_{y \in \ytrunc} f_A(x, y), ~~ h(y) \defeq \min_{x \in \xset} f_A(x, y), ~~\text{and}~~ h\trunc(y) \defeq \min_{x \in \xtrunc} f_A(x, y), 
\end{align*}
note
\begin{align*}
    \gap(z') = g(z'\x) - h(z'\y) ~~\text{and}~~ \gaptrunc(z') = \gtrunc(z'\x) - \htrunc(z'\y) \le \epsilon / 2
\end{align*}
by assumption. Thus, it suffices to show 
\begin{align*}
    \inabs{g(x) - \gtrunc(x)} \overle{(i)} \epsilon / 4 ~~\text{and}~~ \inabs{h(y) - \htrunc(y)} \overle{(ii)} \epsilon / 4
\end{align*}
for all $x \in \xset$ and $y \in \yset$. Starting with the $\ell_2$-$\ell_1$ setup, the inequality $(ii)$ is trivial since $\xset = \xtrunc$. As for $(i)$, note $y \mapsto f_A(x, y)$ is 1-Lipschitz in the $\ell_1$-norm for any fixed $x \in \xset = \ball^n$ by the normalizing assumption $\norm{A}_{2 \to \infty} = \max_{i \in [m]} \norm{A_{i,:}}_2 \le 1$. Hence, for any $x \in \B^n$, we have
\begin{align}
    \label{eq:pin}
    \max_{y \in \ytrunc} f_A(x, y) \le  \max_{y \in \yset} f_A(x, y) \le \max_{y \in \ytrunc} f_A(x, y) + \epsilon / 4,
\end{align}
due to the fact that for any $y \in \yset$, there exists $y' \in \ytrunc$ such that $\norm{y - y'}_1 \le 2 \nu m$ by Lemma~\ref{lem:there-is-a-nearby-point-in-trunc-simplex}.

As for the $\ell_1$-$\ell_1$ setup, observe $y \mapsto y^\top A x$ and $x \mapsto y^\top A x$ are both 1-Lipschitz in the $\ell_1$-norm for any fixed $x \in \simplex^n$ and $y \in \simplex^m$ respectively by the assumption $\inmaxnorm{A} \le 1$. Thus, the inequalities $(i)$ and $(ii)$ both follow by analogous reasoning to \eqref{eq:pin}.

\end{proof}

\subsection{Stability with respect to best responses}
\label{subsec:stability-best-response}

In this section, we show a stability result (discussed in Section~\ref{sec:overview-of-approach} and further below) for the regularized minimax objective
\begin{align}
    \label{eq:reg-minimax-in-stability-sec}
    \min_{x \in \xset_\nu} \max_{y \in \yset_\nu} y^\top A x + \alpha \xbreg{x'}{x} - \alpha \ybreg{y'}{y}
\end{align}
for $\alpha > 0$ and $z' = (x', y') \in \xtrunc \times \ytrunc$. Recalling from Section~\ref{sec:prelim} that the unique exact solution of \eqref{eq:reg-minimax-in-stability-sec} is given by $\prox^\alpha_{z'}(\gm f_A; \ztrunc)$, we first introduce what we refer to as a ``best response'' with respect to the center point $(x', y')$ in Definition~\ref{def:alpha-best-response}. Note that $(\xtilde, \ytilde)$ in Definition~\ref{def:alpha-best-response} is equivalent to $\prox_{z'}^\alpha (\gm f_A (z'))$ with $z' \defeq (x', y')$.

\begin{definition}[$\alpha$-best response]
    \label{def:alpha-best-response}
Defining for $(x', y') \in \xtrunc \times \ytrunc$
\begin{align*}
    \xtilde \defeq \argmin_{x \in \xset_\nu} y'^\top A x + \alpha \xbreg{x'}{x} ~~~\text{and}~~~ \ytilde \defeq \argmax_{y \in \yset_\nu} y^\top A x' - \alpha \ybreg{y'}{y},
\end{align*}
we call $(\xtilde, \ytilde)$ the \emph{$\alpha$-best response to $(x', y')$}.
\end{definition}

Our main result is Lemma~\ref{lem:stability-wrt-best-response}, which says that if the divergence from the center point $z' = (x', y')$ in \eqref{eq:reg-minimax-in-stability-sec} to the (unique) exact solution is order $\alpha^2$, then the $\alpha$-best response to $z'$ is coordinate-wise multiplicatively close to the solution in the coordinates which are constrained to a probability simplex. As discussed in Section~\ref{sec:overview-of-approach}, this stability result is the key to connecting the subproblem solver given in Section~\ref{sec:innerloops}, which only operates over multiplicative balls, to the outer loop of Section~\ref{sec:prox-point-regret}. Lemma~\ref{lem:stability-wrt-best-response} implies that as long as $\breg{z'}{\zopt} \le O(\alpha^2)$, then we can solve \eqref{eq:reg-minimax-in-stability-sec} by restricting to a multiplicative ball around the $\alpha$-best response $\ztilde$.

\begin{lemma}
    \label{lem:stability-wrt-best-response}
    For $\alpha, c > 0$, suppose $z' = (x', y') \in \xtrunc \times \ytrunc$ and $\zopt \defeq \prox^\alpha_{z'}(\gm f_A; \ztrunc)$ are such that $\breg{z'}{\zopt} \le c \alpha^2$, and let $\ztilde$ denote the $\alpha$-best response to $z'$ per Definition~\ref{def:alpha-best-response}. Then $[\zopt\y]_i \approx_{\exp(2\sqrt{2c})} [\ztilde\y]_i$ for all $i \in [m]$ in both the $\ellTwoEllOne$ and $\ellOneEllOne$ setups, and also $[\zopt\x]_j \approx_{\exp(2\sqrt{2c})} [\ztilde\x]_j$ for all $j \in [n]$ in the $\ellOneEllOne$ setup.
\end{lemma}

The proof of Lemma~\ref{lem:stability-wrt-best-response} is given at the end of this section. First, we state and prove two technical supporting lemmas (stated independently of the context of the $\ellTwoEllOne$ and $\ellOneEllOne$ setups). In Lemma~\ref{lem:prox-step-on-trunc-simplex} we characterize a proximal step (with respect to a fixed vector) on the truncated simplex.

\begin{lemma}
    \label{lem:prox-step-on-trunc-simplex}
    For $\nu \in (0, 1 / d)$, $\alpha > 0$, $\theta \in \R^d$, and $q \in \simplex^d_\nu$, define
    \begin{align}
        \label{eq:trunc-step}
        u \defeq \argmin_{z \in \simplex^d_\nu} \theta^\top z + \alpha \cdot \inKL{q}{z}.
    \end{align}
    Then for some $\Lambda > 0$, we have
    \begin{align*}
        [u]_i = \max \inbraces*{ \Lambda [q]_i \exp \inparen*{- \frac{[\theta]_i}{\alpha}}, \nu} ~~~\text{for all $i \in [d]$}
    \end{align*}
    \end{lemma}

    \newcommand{\lambdaopt}{\lambda^\star}
    \newcommand{\muopt}{\mu^\star}

    \begin{proof}
        For $\lambda \in \R$ and $\mu \in \R^d$, let
        \begin{align*}
            L(z, \mu, \lambda) \defeq \theta^\top z + \alpha \cdot \inKL{q}{z}  - \mu^\top z + \lambda \inparen*{\sum_{i \in [d]} [z]_i - 1}
        \end{align*}
        denote the Lagrangian associated with the convex optimization problem \eqref{eq:trunc-step}. By Slater's condition, strong duality holds, and any primal-dual optimal variables $(\zopt, \muopt, \lambdaopt)$ are a saddle point of the Lagrangian. Letting $(\zopt, \muopt, \lambdaopt)$ denote some choice of primal-dual optimal variables, note that $(\zopt, \muopt)$ is a saddle point of $L(\cdot, \cdot, \lambdaopt)$. This implies (note that $\zopt$ is unique by strong convexity of the primal problem \eqref{eq:trunc-step})
        \begin{align*}
            \zopt = \argmin_{z \in \R^d} \max_{\mu \in \R^d_{\ge 0}} L(z, \mu, \lambdaopt) = \argmin_{z \in \R^d : [z]_i \ge \nu, \, \forall i \in [d]} 
            \theta^\top z + \alpha \cdot \inKL{q}{z} + \lambdaopt \inparen*{\sum_{i \in [d]} [z]_i - 1}.
        \end{align*}
        Note that the remaining constraints $[z]_i \ge \nu$ for $i \in [d]$ are coordinate-wise separable, and therefore
        \begin{align*}
            [\zopt]_i = \max \inbraces*{\argmin_{\zeta \in \R} \inbraces{ [\theta]_i \zeta + \alpha \zeta \log (\zeta / [q]_i) + \lambdaopt \zeta} , \nu}
            \text{ for all }
            i \in [d]\,.
        \end{align*}
        The result follows from a direct computation of $\argmin_{\zeta \in \R} \inbraces{ [\theta]_i \zeta + \alpha \zeta \log (\zeta / [q]_i) + \lambdaopt \zeta }$, which reveals $\Lambda = \exp(- \lambdaopt / \alpha) / e$.
    \end{proof}

In Lemma~\ref{lem:general-stability-helper}, we show that if we take proximal steps on the truncated simplex with respect to vectors which are close in $\ell_\infty$-distance, then the resulting coordinates are multiplicatively close.

\begin{lemma}
    \label{lem:general-stability-helper}
    For $\nu \in (0, 1 / d)$, $\alpha > 0$, vectors $\theta, \xi \in \R^d$, and $q \in \simplex^d_\nu$, define
    \begin{align*}
        u_\theta \defeq \argmin_{z \in \simplex^d_\nu} \theta^\top z + \alpha \cdot \inKL{q}{z} ~~~\text{and}~~~ u_\xi \defeq \argmin_{z \in \simplex^d_\nu} \xi^\top z + \alpha \cdot \inKL{q}{z}.
    \end{align*}
    Then $u_\theta \approx_{\delta} u_\xi$ with $\delta \defeq \exp \inparen*{\frac{2 \norm{\theta - \xi}_\infty}{\alpha}}$.
\end{lemma}

\begin{proof} 
Lemma~\ref{lem:prox-step-on-trunc-simplex} and the given assumption implies for all $i \in [d]$:
\begin{align*}
    [u_\theta]_i = \max \inbraces*{ \Lambda [\rho]_i, \nu}  ~~~\text{and}~~~ [u_\xi]_i = \max \inbraces*{ \Lambda' [\rho']_i, \nu},
\end{align*}
for some $\Lambda, \Lambda' > 0$ and $\rho, \rho' \in \R^d$ such that with $C \defeq \exp(\norm{\theta - \xi}_\infty / \alpha)$ we have $\rho \approx_C \rho'$.
Thus, it suffices to show $\Lambda \approx_C \Lambda'$. Supposing $\Lambda  > C \Lambda'$ for the sake of contradiction, note $\Lambda [\rho]_i >  \Lambda' [\rho'_i]$ for all $i \in [d]$, in which case $u_\theta$ and $u_\xi$ cannot both be probability distributions. Similarly, supposing $\Lambda  < \Lambda' / C$ for contradiction, we have $\Lambda [\rho]_i <  \Lambda' [\rho'_i]$ for all $i \in [d]$, which is a contradiction for the same reason.
\end{proof}

Finally, we give the proof of Lemma~\ref{lem:stability-wrt-best-response}.

\begin{proof}[Proof of Lemma~\ref{lem:stability-wrt-best-response}]
The fact that $\zopt$ is the exact solution of \eqref{eq:reg-minimax-in-stability-sec} implies
\begin{align*}
    \zopt\x = \argmin_{x \in \xtrunc} {\zopt\y}^\top A x + \alpha \xbreg{x'}{x} ~~~\text{and}~~~ \zopt\y = \argmax_{y \in \ytrunc} y^\top A \zopt\x - \alpha \ybreg{y'}{y}.
\end{align*}
Furthermore, note that
\begin{align*}
    \innorm{A^\top (\zopt\y - y')}_\infty \le \max_{i} \norm{A_{:,i}}_\infty  \innorm{\zopt\y - y'}_1 \le \sqrt{2 \ybreg{y'}{\zopt\y}} \le \alpha \sqrt{2c}
\end{align*}
since $\max_{i} \norm{A_{:,i}}_\infty \le 1$ in both setups, and $\norm{A(\zopt\x - x')}_\infty \le \alpha \sqrt{2c}$ holds in the $\ellOneEllOne$ setup by analogous reasoning. Then we conclude by applying Lemma~\ref{lem:general-stability-helper}.
\end{proof}

\subsection{Binary search subroutine}
\label{subsec:binary-search}

In this section, we give the binary search subroutine described in Section~\ref{sec:overview-of-approach}. Formally, the goal of this section is, for $\beta > 0$ (which will ultimately be set to $\epsilon^{1/3}$ in our applications, although we keep it general throughout this section) and a given center point $\zcenter \in \zset_\nu$, to either: (i) return $\beta$ only if $\breg{\zcenter}{\prox^{\beta}_{\zcenter}(\gm f_A; \zset_\nu)} \le 2.8 \beta^2$, or else (ii) return $\alphaopt > \beta$ such that $1.2 \alphaopt^2 \le \breg{z_c}{\prox^{\alphaopt}_{\zcenter}(\gm f_A; \zset_\nu)} \le 2.8 \alphaopt^2$. Note that the first statement is an only if, so we may return an $\alphaopt$ satisfying (ii) even if $\breg{\zcenter}{\prox^{\beta}_{\zcenter}(\gm f_A; \zset_\nu)} \le 2.8 \beta^2$. (Recall also from Section~\ref{sec:prelim} that for $\alpha > 0$, we have that $\prox_{z_c}^\alpha(\gm f_A; \ztrunc)$ is the unique exact solution of the regularized minimax objective  $\min_{x \in \xtrunc} \max_{y \in \ytrunc} y^\top A x + \alpha \xbreg{\zcenterx}{x} - \alpha \ybreg{\zcentery}{y}$.)

Our binary search procedure is given below in Algorithm~\ref{alg:binary-search}. Before describing the pseudocode, we note that Algorithm~\ref{alg:binary-search} calls Algorithm~\ref{alg:mirror-prox-sug-strongly-monotone} as a subroutine to approximate proximal mappings in several ways. However, to clarify the purpose of these calls to Algorithm~\ref{alg:mirror-prox-sug-strongly-monotone} in the pseudocode, we define a wrapper function\footnote{By \emph{wrapper function,} we mean a function whose goal is to call a second subroutine with very little additional computation. The intention is only to provide a purpose-built interface for the second subroutine.} in Definition~\ref{def:CAPW} below, and then prove in Corollary~\ref{cor:CAPW-properties} that it achieves certain multiplicative and additive approximations. To make them easier to distinguish, the subscripts in $\gammapw$ and $\gammab$ in Definition~\ref{def:CAPW} stand for ``pointwise'' and ``Bregman'' respectively. Finally, throughout the remainder of Section~\ref{sec:elltwoelloneandelloneellone}, we define $\sball_{c, z}$ for $c > 1$ and $z \in \cZint$ as in Definition~\ref{def:stable-ball}, and further define $\sball_{c, z', \nu} \defeq \sball_{c, z'} \cap \zset_\nu$ for $z' \in \ztrunc$.

\begin{definition}[$\CWF$]
    \label{def:CAPW}
For $\zcenter, \zground, \ztilde \in \ztrunc$, some $\tau, \alpha, \gammab, \gammapw > 0$, a matrix-vector oracle for $A$, some $M \in \R^{m \times n}$, an absolute constant $c > 1$, and a $p$-smooth-guilty judge $\judge$, a \\ \emph{$(\zcenter, \zground, \ztilde, \tau, \alpha, A, M, c, \gammapw, \gammab; \judge)$-coordinate-wise wrapper function} returns
\begin{align*}
    \CWF(\zcenter, \zground, \ztilde, \tau, \alpha, A, M, c, \gammapw, \gammab; \judge) \defeq
\SUGStronglyMonotoneMirrorProx(z_c, \zground, \tau, \alpha, \epsprim, A, M, c, \cZ'; \judge),
\end{align*}
where, for an absolute constant $C' > 0$ defined in Corollary~\ref{cor:CAPW-properties}, we let
\begin{align*}
    Z' \defeq \sball_{c, \ztilde, \nu} = \sball_{c, \ztilde} \cap \ztrunc ~~~\text{and}~~~ \epsprim \defeq C' \min \inbraces*{\gammapw^2 \nu^2, \frac{\gammab^2}{1 + \log^2 \nu^{-1}}}.
\end{align*}
\end{definition}

The purpose of the $\CWF$ wrapper is elucidated in Corollary~\ref{cor:CAPW-properties} below. In particular, $\epsilon'$ is chosen in Definition~\ref{def:CAPW} so that the output $w$ of the $\CWF$ wrapper is multiplicatively close to those coordinates of $\zopt \defeq \prox^\alpha_{\zcenter}(\gm f_A; \sball_{c, \ztilde, \nu})$ which lie in the probability simplex. Furthermore, the divergence from the center point $\zcenter$ to $w$ and $\zopt$ is approximately the same.

\begin{corollary}[$\CWF$ properties]
    \label{cor:CAPW-properties}
Letting $w \gets \CWF(\zcenter, \zground, \ztilde, \tau, \alpha, A, M, c, \gammapw, \gammab; \judge)$ for some inputs satisfying the specifications of Definition~\ref{def:CAPW} and an appropriate choice of the absolute constant $C' > 0$, and with $\zopt \defeq \prox^\alpha_{\zcenter}(\gm f_A; \sball_{c, \ztilde, \nu})$, we have all of the following:
\begin{align*}
  & [w]_i \approx_{1 + \gammapw} [\zopt]_i , && \text{for all $i \in [d]$ in the $\ellOneEllOne$ setup}, \\
  &  [w\y]_i \approx_{1 + \gammapw} [\zopt\y]_i , && \text{for all $i \in [m]$ in the $\ellTwoEllOne$ setup, and} \\
  & \inabs{\breg{\zcenter}{w} - \breg{\zcenter}{\zopt}} \le \gammab, &&\text{in both setups.}
\end{align*}
\end{corollary}

\begin{proof}
This is immediate from Theorem~\ref{thm:main-inner-loop-convergence} and Lemmas \ref{lem:breg-error-to-pointwise-error} and \ref{lem:breg-error-to-breg-error-from-base}.
\end{proof}

Let us now walk through the pseudocode of Algorithm~\ref{alg:binary-search}. First of all, the goal of the $\checkdiv$ subroutine starting in Line~\ref{algline:checkdiv-function} is, for a given $\alpha > 0$ and center point $\zcenter \in \ztrunc$ with $\zoptalpha \defeq \prox^\alpha_{\zcenter}(\gm f_A ; \ztrunc)$, to determine whether $\breg{\zcenter}{\zopt_\alpha}$ lies within the ``target interval'' $[1.2 \alpha^2, 2.8 \alpha^2]$ or else $\breg{\zcenter}{\zoptalpha} < 1.8 \alpha^2$ or $\breg{\zcenter}{\zoptalpha} > 2.2 \alpha^2$. These three cases are guaranteed respectively per Lemma~\ref{lem:checkdiv-correctness} below when $\checkdiv$ returning the flags $\justRight$, $\tooBig$, or $\tooSmall$. Note that the intervals are \emph{not} mutually exclusive and the return flags serve as ``only if'' guarantees (e.g., $\checkdiv$ returns $\justRight$ only if $\breg{\zcenter}{\zopt_\alpha} \in [1.2 \alpha^2, 2.8 \alpha^2]$); the reason for this is that $\breg{\zcenter}{\zoptalpha}$ can only be approximated by Algorithm~\ref{alg:mirror-prox-sug-strongly-monotone} and its $\CWF$ wrapper, so some overlap is necessary.

Then setting aside the internals of the $\checkdiv$ subroutine for a moment, Lines \ref{algline:b-search-check-div-beta-call} and \ref{algline:b-search-beta-return} serve as the initial check as to whether we can return $\beta$; in particular, $\beta$ is returned in Line~\ref{algline:b-search-beta-return} only if $\breg{\zcenter}{\prox^{\beta}_{\zcenter}(\gm f_A; \zset_\nu)} \le 2.8 \beta^2$ by the return guarantee for the $\checkdiv$ function. If Algorithm~\ref{alg:binary-search} does not return in Line~\ref{algline:b-search-beta-return}, then the binary search begins and we search for $\alphaopt \in [\beta, 3]$ such that $1.2 \alphaopt^2 \le \breg{z_c}{\prox^{\alphaopt}_{\zcenter}(\gm f_A; \zset_\nu)} \le 2.8 \alphaopt^2$. Lemma~\ref{lem:b-search-correctness} guarantees this happens after $\Otilde(1)$ iterations of the repeat loop starting in Line~\ref{algline:b-search-repeat-loop}, where we hide polylog factors $\nu^{-1}$ and $\beta^{-1}$ (hence the necessity of these parameters being bounded by a polynomial in $m$, $n$, and $\epsilon$ in our application to matrix games). The proof of Lemma~\ref{lem:b-search-correctness} relies on a Lipschitz guarantee given in Lemma~\ref{lem:Lipschitzness-of-h} for the function $h(\alpha) \defeq \breg{\zcenter}{\zopt_\alpha} - 2 \alpha^2$ on intervals bounded away from 0.

Let us now examine the internals of the $\checkdiv$ subroutine starting in Line~\ref{algline:checkdiv-function}. Note that it is within the $\checkdiv$ subroutine that the model-passing flag $\passModel$ given as input to Algorithm~\ref{alg:binary-search} changes the behavior of the algorithm. We set $\passModel \gets \false$ when obtaining our $\Otilde(\epsilon^{-8/9})$ guarantee for both $\ellOneEllOne$ and $\ellTwoEllOne$ games in Section~\ref{subsec:ell2ell1-ell1ell1-putting-all-together}; in this case, the presence of models $M$ in Algorithm~\ref{alg:binary-search} can be ignored since they are always cleared in Line~\ref{algline:check-div-tau-M-when-passmodel-false} before the $\CWF$ wrapper (and therefore Algorithm~\ref{alg:mirror-prox-sug-strongly-monotone}) is called. We set $\passModel \gets \true$ in our $\Otilde(\epsilon^{-7/9})$ guarantee for $\ellTwoEllOne$ games in Section~\ref{sec:four-fifths}; in this case, models are passed between separate calls to $\CWF$ (and therefore Algorithm~\ref{alg:mirror-prox-sug-strongly-monotone}) so that progress is saved between calls. 

For the sake of the current discussion, we assume $\passModel \gets \false$ and defer discussion of the $\passModel \gets \true$ case to Section~\ref{sec:four-fifths}. Then, in Line~\ref{algline:checkdiv-alpha-best-response} we set $\ztilde$ to the $\alpha$-best response with respect to $\zcenter$ per Definition~\ref{def:alpha-best-response}. We set our local norm point $\zground$ to $\ztilde$ in Line~\ref{algline:check-div-zground-tau-when-passmodel-true} (any point within $\sball_{C, \ztilde, \nu}$ would work), and then call Algorithm~\ref{alg:mirror-prox-sug-strongly-monotone} in Line~\ref{algline:checkdiv-CAPW-call} through the $\CWF$ wrapper. Given the inputs of that line, Algorithm~\ref{alg:mirror-prox-sug-strongly-monotone} will approximate $\prox^\alpha_{\zcenter}(\gm f_A; \sball_{C, \ztilde, \nu})$, and thus it is necessary to connect this to $\zopt_\alpha \defeq \prox^\alpha_{\zcenter}(\gm f_A; \zset_\nu)$ to prove the $\checkdiv$ guarantee Lemma~\ref{lem:checkdiv-correctness}. This is where the stability result Lemma~\ref{lem:stability-wrt-best-response} enters the picture, as it guarantees $\zopt_\alpha$ lies in a multiplicative ball around $\ztilde$ as long as $\breg{\zcenter}{\zopt_\alpha} = O(\alpha^2)$. Then, the $\checkcoords$ call in Line~\ref{algline:check-div-coords-return} serves to check whether the result $w$ of the $\CWF$ call in Line~\ref{algline:checkdiv-CAPW-call} and therefore $\prox^\alpha_{\zcenter}(\gm f_A; \sball_{C, \ztilde, \nu})$ (by the guarantee of Corollary~\ref{cor:CAPW-properties}) lies in the interior of $\sball_{C, \ztilde, \nu}$, in which case $\prox^\alpha_{\zcenter}(\gm f_A; \sball_{C, \ztilde, \nu})$ and $\zopt_\alpha$ coincide. If this is not the case, we can safely return $\tooBig$ in Line~\ref{algline:check-div-coords-return} since the condition $\breg{\zcenter}{\zopt_\alpha} = O(\alpha^2)$ was not satisfied with a small enough constant. Otherwise, the fact that $\zopt_\alpha$ and $\prox^\alpha_{\zcenter}(\gm f_A; \sball_{C, \ztilde, \nu})$ coincide means we can directly approximate $\breg{\zcenter}{\zopt_\alpha}$ via $\breg{\zcenter}{w}$ and return appropriately in Lines \ref{algline:check-div-breg-check-toobig} through \ref{algline:check-div-breg-check-justright}.

\RestyleAlgo{ruled}
\DontPrintSemicolon
\SetKwComment{Comment}{/* }{ */}
\begin{algorithm2e}[h!]
\caption{Binary search subroutine $\bsearch(\epsilon, \zcenter, \beta, A, M_0, \tau; \judge, \passModel, p, \Xi)$}
\label{alg:binary-search}
\KwInput{Precision $\epsilon > 0$, center point $\zcenter \in \ztrunc$, minimum subproblem regularization level $\beta \in (0, 3)$, matrix-vector oracle for $A$ s.t. $\innorm{A}_{2 \to \infty} \le 1$ in the $\ellTwoEllOne$ setup or $\inmaxnorm{A} \le 1$ in the $\ellOneEllOne$ setup, model $M_0 \in \R^{m \times n}$, smoothness threshold $\tau > 0$}
\KwParameter{Judge function $\judge$, model-passing flag $\passModel$ (either $\true$ or $\false$), parameter $p \in [1, 2]$, and an upper bound $\Upsilon_p(A) \leq \Xi$ (see \eqref{eq:chi})}
\tcp{Determine whether we should return $\beta$}
$(\betaDivFlag, M_1) \gets \checkdiv(\beta, \zcenter, A, M_0, \tau; \judge, \passModel)$ \label{algline:b-search-check-div-beta-call}

\lIf{$\betaDivFlag = \tooSmall$ or $\betaDivFlag = \justRight$}{\Return{$(\beta, M_1)$} \label{algline:b-search-beta-return}}

\tcp{If not, begin binary search}

$\alpha_\ell \gets \beta$, ~$\alpha_r \gets 3$, ~and~ $t \gets 1$ \tcp*{Initialize search interval $[\alpha_\ell, \alpha_r]$}

\Repeat(\tcp*[f]{Proof of termination in Lemma~\ref{lem:b-search-correctness}}){\label{algline:b-search-repeat-loop}}{

    $\alpha_m \gets (\alpha_\ell + \alpha_r) / 2$ ~and~ $t \gets t + 1$ \label{algline:b-search-new-midpoint}

    $(\alphaDivFlag, M_t) \gets \checkdiv(\alpha_m, \zcenter, A, M_{t - 1}, \tau; \judge, \passModel)$ \label{algline:b-search-check-div-call-repeat-loop}

    \lIf(\tcp*[f]{New interval is $[\alpha_\ell, \alpha_m]$}){$\alphaDivFlag = \tooSmall$}{$\alpha_r \gets \alpha_m$ \label{algline:b-search-got-toosmall}}

    \lElseIf(\tcp*[f]{New interval is $[\alpha_m, \alpha_r]$}){$\alphaDivFlag = \tooBig$}{$\alpha_\ell \gets \alpha_m$ \label{algline:b-search-got-toobig}}

    \lElse(\tcp*[f]{$\alphaDivFlag = \justRight$}){\Return{$(\alpha_m, M_t)$} \label{algline:b-search-return-finished-b-search}}

} 

\BlankLine

\tcp{Returns flag ($\tooSmall$, $\tooBig$, or $\justRight$) and new model $M'$}
\Function{$\checkdiv(\alpha, \zcenter, A, M, \tau; \judge, \passModel)$ \label{algline:checkdiv-function}}{
    
    $C \defeq 500$, ~$\gammapw \defeq 0.1$, ~and~ $\gammab \defeq \alpha^2 / 10$ \label{algline:checkdiv-set-C-and-gammas}

    \tcp{Compute $\alpha$-best response to $\zcenter$ (Definition~\ref{def:alpha-best-response})}

     $\ztilde \defeq (\xtilde, \ytilde)$ for $\xtilde \gets \argmin_{x \in \xtrunc} \zcentery^\top A x + \alpha \xbreg{\zcenterx}{x}$ and $\ytilde \gets \argmax_{y \in \ytrunc} y^\top A \zcenterx - \alpha \ybreg{\zcentery}{y}$ \label{algline:checkdiv-alpha-best-response}

    \uIf{$\passModel$}
    {
        $\zground \gets \argmin_{z \in \sball_{C, \ztilde, \nu}} \breg{\zcenter}{z}$ \label{algline:check-div-zground-tau-when-passmodel-true}
        
        \tcp{If $\zground$ is too far, we do not need to call the $\CWF$ wrapper function}

        \lIf{$\breg{\zcenter}{\zground} > 3 \alpha^2$}{\Return{$(\tooBig, M' \defeq M)$} \label{algline:check-div-movement-return}} 
    } 
    \lElse(\tcp*[f]{In particular, clear model if $\passModel = \false$}){
        $\zground \gets \ztilde$ ~and~ $M \gets 0$   \label{algline:check-div-tau-M-when-passmodel-false}
    }

    \tcp{Definition~\ref{def:CAPW} and Cor. \ref{cor:CAPW-properties}; the second return value of $\CWF$ is only used in the analysis (not algorithmically)}

    $(w, \cdot, M') \gets \CWF(\zcenter, \zground, \ztilde, \tau, \alpha, A, M, C, \gammapw, \gammab; \judge)$  \label{algline:checkdiv-CAPW-call}

    \lIfNot{$\coordsInRange \defeq \checkcoords(w, \ztilde)$}{\Return{$(\tooBig, M')$} \label{algline:check-div-coords-return}}

    \lElseIf{$\breg{\zcenter}{w} > 2.5 \alpha^2$}{\Return{$(\tooBig, M')$} \label{algline:check-div-breg-check-toobig}}
    \lElseIf{$\breg{\zcenter}{w} < 1.5 \alpha^2$}{\Return{$(\tooSmall, M')$} \label{algline:check-div-breg-check-toosmall}}
    \lElse(\tcp*[f]{$\breg{\zcenter}{w} \in [1.5 \alpha^2, 2.5 \alpha^2]$}){\Return{$(\justRight, M')$} \label{algline:check-div-breg-check-justright}}
}

\BlankLine

\Function{$\checkcoords(w, \ztilde)$}{

    \uIf{in the $\ellTwoEllOne$ setup and there exists $i \in [m]$ s.t. $[w\y]_i \notin [\frac{1}{400}[\ztilde\y]_i, 400[\ztilde\y]_i]$ \label{algline:checkcoords-ell2ell1-condition}}{
        \Return{$\coordsInRange \defeq \false$}
      }
      \uElseIf{in the $\ellOneEllOne$ setup and there exists $j \in [d]$ s.t. $[w]_j \notin [\frac{1}{400}[\ztilde]_j, 400[\ztilde]_j]$ \label{algline:checkcoords-ell1ell1-condition}}{
        \Return{$\coordsInRange \defeq \false$}
      }
      \lElse{\Return{$\coordsInRange \defeq \true$}}
}

\end{algorithm2e}

We now state and prove our formal guarantee given in Lemma~\ref{lem:checkdiv-correctness} for the $\checkdiv$ helper function starting in Line~\ref{algline:checkdiv-function} of Algorithm~\ref{alg:binary-search} for any inputs. As in the rest of this paper, we may use a dot $(\cdot)$ in place of return values which are not relevant to a given statement (e.g., the new model $M'$ returned by $\checkdiv$ is not relevant to the statement of Lemma~\ref{lem:checkdiv-correctness}). 

\begin{lemma}[$\checkdiv$ correctness]
    \label{lem:checkdiv-correctness}
For any inputs $\alpha > 0$, $\zcenter \in \ztrunc$, a matrix-vector oracle for $A$, some $M \in \R^{m \times n}$, and parameters $\judge$ and $\passModel$, letting $\zoptalpha \defeq \prox^\alpha_{\zcenter}(\gm f_A ; \ztrunc)$, if the function $\checkdiv$ (Line~\ref{algline:checkdiv-function} in Algorithm~\ref{alg:binary-search}) returns: 
\begin{enumerate}[label=(\alph*)]
    \item $(\tooSmall, \cdot)$, then $\breg{\zcenter}{\zoptalpha} < 1.8 \alpha^2$,

    \item $(\justRight, \cdot)$, then $\breg{\zcenter}{\zoptalpha} \in [1.2 \alpha^2, 2.8 \alpha^2]$,

    \item $(\tooBig, \cdot)$, then $\breg{\zcenter}{\zoptalpha} > 2.2 \alpha^2$.
\end{enumerate}
Furthermore, if the execution of $\checkdiv$ reaches Line~\ref{algline:check-div-breg-check-toobig} (i.e., it does not return in Lines \ref{algline:check-div-movement-return} or \ref{algline:check-div-coords-return}), then $\zoptalpha \in \sball_{C, \ztilde, \nu}$ for $C$ and $\ztilde$ as defined in Lines \ref{algline:checkdiv-set-C-and-gammas} and \ref{algline:checkdiv-alpha-best-response} respectively.
\end{lemma}

\begin{proof}
    Throughout this proof, we will ignore the second argument (i.e., the new model $M'$) returned by $\checkdiv$ since it has no bearing on the lemma statement, and we let $C$, $\gammapw$, $\gammab$, $\ztilde$, and all other variables be as they are defined in $\checkdiv$. Then, we proceed through the various ways $\checkdiv$ can return in line order, and prove the additional claim when we reach it:

    \textit{Line~\ref{algline:check-div-movement-return}:} $\checkdiv$ will return $\tooBig$ in Line~\ref{algline:check-div-movement-return} if $\passModel = \true$ and $\breg{\zcenter}{\zground} > 3 \alpha^2$, where we set $\zground \gets \argmin_{z \in \sball_{C, \ztilde, \nu}} \breg{\zcenter}{z}$. Note first that if it is the case that $\zoptalpha \in \sball_{C, \ztilde, \nu}$, then $\breg{\zcenter}{\zoptalpha} \ge \breg{\zcenter}{\zground} > 3 \alpha^2$. Then supposing $\zoptalpha \notin \sball_{C, \ztilde, \nu}$, we have $\breg{\zcenter}{\zoptalpha} > 4 \alpha^2$ by Lemma~\ref{lem:stability-wrt-best-response}, since $C = 500 > \exp(2 \sqrt{2 \cdot 4})$. (Note that $\ztilde$ defined in Line~\ref{algline:checkdiv-alpha-best-response} is precisely the $\alpha$-best response to $\zcenter$ per Definition~\ref{def:alpha-best-response}.)

    \textit{Line~\ref{algline:check-div-coords-return}:} Let $\zoptlocal \defeq \prox^\alpha_{\zcenter}(\gm f_A ; \sball_{C, \ztilde, \nu})$. (The subscript $\ell$ stands for ``local.'') By the choice of $\gammapw$ in Line~\ref{algline:checkdiv-set-C-and-gammas}, the definition of $w$ in Line~\ref{algline:checkdiv-CAPW-call}, and Corollary~\ref{cor:CAPW-properties}, we have in particular that $0.9 \le [w\y]_i / [\zoptlocaly]_i \le 1.1$ for all $i \in [m]$ in the $\ellTwoEllOne$ setup, and $0.9 \le [w]_j / [\zoptlocal]_j \le 1.1$ for all $j \in [d]$ in the $\ellOneEllOne$ setup. Let us first consider the $\ellOneEllOne$ setup, in which case the conditions in Line~\ref{algline:checkcoords-ell1ell1-condition} imply we return $\tooBig$ in Line~\ref{algline:check-div-coords-return} if and only if there exists $j \in [d]$ such that $[w]_j \notin [\frac{1}{400}[\ztilde]_j, 400[\ztilde]_j]$. Note that $[w]_j \notin [\frac{1}{400}[\ztilde]_j, 400[\ztilde]_j]$ implies $[\zoptlocal]_j \notin [\frac{1}{300} [\ztilde]_j, 300 [\ztilde]_j]$ by the above. Furthermore, we claim the latter implies $\breg{\zcenter}{\zoptalpha} > 3 \alpha^2$. Indeed, if it were the case that $\breg{\zcenter}{\zoptalpha} \le 3 \alpha^2$, then Lemma~\ref{lem:stability-wrt-best-response} would imply $[\zoptalpha] \in [\frac{1}{300} [\ztilde]_j, 300 [\ztilde]_j]$ since $300 > \exp(2 \sqrt{2 \cdot 3})$, which would in turn imply $\zoptalpha \in \sball_{C, \ztilde, \nu}$ and $\zoptalpha = \zoptlocal$ in particular, a contradiction. Analogous reasoning implies that if we return $\tooBig$ in Line~\ref{algline:check-div-coords-return} in the $\ellTwoEllOne$ setup, then $\breg{\zcenter}{\zoptalpha} > 3 \alpha^2$.

    \textit{Additional claim:} We now prove the additional claim at the end of Lemma~\ref{lem:checkdiv-correctness}. Note that if $\checkdiv$ does not return in Line~\ref{algline:check-div-coords-return}, then it must be the case that $[w\y]_i \in [\frac{1}{400}[\ztilde\y]_i, 400[\ztilde\y]_i]$ for all $i \in [m]$ in the $\ellTwoEllOne$ setup, and $[w]_j \in [\frac{1}{400}[\ztilde]_j, 400[\ztilde]_j]$ for all $j \in [d]$ in the $\ellOneEllOne$ setup. We will prove the claim in the $\ellOneEllOne$ setup, as the reasoning is analogous in the $\ellTwoEllOne$ setup. Then defining $\zoptlocal \defeq \prox^\alpha_{\zcenter}(\gm f_A ; \sball_{C, \ztilde, \nu})$ as in the justification for Line~\ref{algline:check-div-coords-return} above and recalling $0.9 \le [w]_j / [\zoptlocal]_j \le 1.1$ for all $j \in [d]$, we have $[\zoptlocal]_j \in [\frac{1}{450}[\ztilde]_j, 450[\ztilde]_j]$ for all $j \in [d]$ in particular. Then the constraint to $\sball_{C, \ztilde, \nu}$ in the definition of $\zoptlocal$ is not binding, and thus $\zoptalpha = \zoptlocal$, implying the claim.

    \textit{Lines \ref{algline:check-div-breg-check-toobig}, \ref{algline:check-div-breg-check-toosmall}, and \ref{algline:check-div-breg-check-justright}:} By the choice of $\gammab$ in Line~\ref{algline:checkdiv-set-C-and-gammas}, Corollary~\ref{cor:CAPW-properties}, and the additional claim in Lemma~\ref{lem:checkdiv-correctness}, we have $\inabs{ \breg{\zcenter}{w} - \breg{\zcenter}{\zoptalpha} } \le \alpha^2 / 10$. The correctness in regard to these lines follows.
\end{proof}

Next, we give our main correctness guarantee for Algorithm~\ref{alg:binary-search} in Lemma~\ref{lem:b-search-correctness} below.

\begin{lemma}[Algorithm~\ref{alg:binary-search} correctness]
    \label{lem:b-search-correctness}
    For a given $\alpha > 0$ and with $\zcenter$ as input to Algorithm~\ref{alg:binary-search}, define $\zoptalpha \defeq \prox^\alpha_{\zcenter}(\gm f_A; \ztrunc)$ (namely, $\zoptalpha$ is parameterized by $\alpha$).
    If Algorithm~\ref{alg:binary-search} returns $(\beta, \cdot)$ in Line~\ref{algline:b-search-beta-return}, then $\breg{\zcenter}{\zoptbeta} \le 2.8 \beta^2$. Otherwise, the repeat loop starting in Line~\ref{algline:b-search-repeat-loop} repeats at most
    \begin{align*}
         \log \inparen*{\frac{M( 1 + \log \nu^{-1} )}{\nu \beta^3}} \text{ times for some absolute constant $M > 0$,}
    \end{align*}
    at which point Algorithm~\ref{alg:binary-search} returns $(\alphaopt, \cdot)$ in Line~\ref{algline:b-search-return-finished-b-search} where $\alphaopt \in [\beta, 3]$ satisfies $\breg{\zcenter}{\zoptalphaopt} \in [1.2 \alphaopt^2, 2.8 \alphaopt^2]$.
\end{lemma}

\begin{proof}
Note that Algorithm~\ref{alg:binary-search} returns in Line~\ref{algline:b-search-beta-return} if and only if the call to $\checkdiv$ in Line~\ref{algline:b-search-check-div-beta-call} returns $(\tooSmall, \cdot)$ or $(\justRight, \cdot)$, in which case the claim for Line~\ref{algline:b-search-beta-return} follows immediately from Lemma~\ref{lem:checkdiv-correctness}. 

\textit{An invariant of the repeat loop:} Now consider the repeat loop starting in Line~\ref{algline:b-search-repeat-loop}, and define $h : \R_{>0} \to \R$ via $h(\alpha) \defeq \breg{\zcenter}{\zoptalpha} - 2 \alpha^2$. We claim the repeat loop maintains the invariant that there exists some $\alpha' \in [\alpha_\ell, \alpha_r]$ such that $h(\alpha') = 0$. (Formally, every time right before Line~\ref{algline:b-search-new-midpoint} executes, such an $\alpha'$ exists.) We proceed by induction on the stronger claim that $h(\alpha_\ell) > 0$ and $h(\alpha_r) < 0$ every time right before Line~\ref{algline:b-search-new-midpoint} executes, which suffices by the continuity of $h$ over compact intervals due to Lemma~\ref{lem:Lipschitzness-of-h}. In the base case, note $\alpha_\ell = \beta$ and $\alpha_r = 3$, in which case our goal is to show $h(\beta) > 0$ and $h(3) < 0$.
For the former, note that Algorithm~\ref{alg:binary-search} only reaches the repeat loop starting in Line~\ref{algline:b-search-repeat-loop} if the call to $\checkdiv$ in Line~\ref{algline:b-search-check-div-beta-call} returned $(\tooBig, \cdot)$, in which case $\breg{\zcenter}{\zoptbeta} > 2 \beta^2$ by Lemma~\ref{lem:checkdiv-correctness}. The fact that $h(3) < 0$ follows immediately from Lemma~\ref{lem:starting-value-b-search}. 

For the induction step, suppose at the start of Line~\ref{algline:b-search-new-midpoint} (at some point during the execution of the repeat loop) it is the case that $h(\alpha_\ell) > 0$ and $h(\alpha_r) < 0$. The algorithm only reaches Line~\ref{algline:b-search-new-midpoint} again if it does not return in Line~\ref{algline:b-search-return-finished-b-search}, and thus we can focus on the two cases of Lines \ref{algline:b-search-got-toosmall} and \ref{algline:b-search-got-toobig}. If the update in Line~\ref{algline:b-search-got-toosmall} executes, the new interval becomes $[\alpha_\ell, \alpha_m]$, in which case it suffices to show $h(\alpha_m) < 0$. Note that this happens if and only if the call to $\checkdiv$ in Line~\ref{algline:b-search-check-div-call-repeat-loop} returns $(\tooSmall, \cdot)$, in which case the claim follows from Lemma~\ref{lem:checkdiv-correctness}. If the update in Line~\ref{algline:b-search-got-toobig} executes, the new interval becomes $[\alpha_m, \alpha_r]$, in which case $h(\alpha_m) > 0$ by similar reasoning. 

\textit{Termination of the repeat loop:} Note that $h$ is $G_h \defeq \frac{M'(1 + \log \nu^{-1})}{\nu \beta}$-Lipschitz over $[\beta, 3]$ for some absolute constant $M' > 0$ per Lemma~\ref{lem:Lipschitzness-of-h}. We claim then that if right before Line~\ref{algline:b-search-new-midpoint} executes (at some point during the execution of the repeat loop) we have $\alpha_r - \alpha_\ell \le 0.1 \cdot \beta^2 / G_h$, then it will return in Line~\ref{algline:b-search-return-finished-b-search} (and not repeat again). Indeed, note that $\alpha_r - \alpha_\ell \le 0.1 \cdot \beta^2 / G_h$ implies $\inabs{ h(\alpha) } \le 0.1 \cdot \beta^2 \le 0.1 \cdot \alpha^2$ for all $\alpha \in [\alpha_\ell, \alpha_r]$ for the current values of $\alpha_\ell$ and $\alpha_r$, due to the invariant proven above. In particular, this implies that $\alpha_m$ which gets passed to $\checkdiv$ in Line~\ref{algline:b-search-check-div-beta-call} during this iteration satisfies $\breg{\zcenter}{\zopt_{\alpha_m}} \in [1.9 \alpha_m^2, 2.1 \alpha_m^2]$. As a result, Lemma~\ref{lem:checkdiv-correctness} implies it must be the case that this $\checkdiv$ call returns $(\justRight, \cdot)$, in which case we return in Line~\ref{algline:b-search-return-finished-b-search} as claimed and the output $\alpha_m$ satisfies the stated property. Finally, the fact that we terminate when $\alpha_r - \alpha_\ell \le 0.1 \cdot \beta^2 / G_h$ straightforwardly yields the stated iteration bound.
\end{proof}

Finally, we conclude with a bound on the total number of matrix-vector queries made by Algorithm~\ref{alg:binary-search} for an optimal choice of inputs when $\passModel = \false$. However, we first need the following technical lemma, whose proof is deferred to Appendix~\ref{apx:linear-algebra}. In the next result, we use the following notation. For any $s, t \in [1, \infty]$, let $ \Uq{A}{s, t} \defeq \normInline{a}_{t}, \text{ where } [a]_i = \normInline{A_{i, :}}_s.$ That is, $\Uq{A}{s, t}$ is the $\ell_t$-norm of the vector $a$, whose the $i$-th entry is given by the $\ell_s$-th norm of the $i$-th row of $A$. For completeness, we show in Proposition~\ref{prop:instance-dependent-is-norm} in Appendix~\ref{apx:ommitted-proofs-ell1-ell1-ell2-ell1-final-algo} that $\innorm{\cdot}_{(s, t)}$ is indeed a norm.

\begin{restatable}{lemma}{generalconversion}\label{lemma:gen-conversion} Let $\zground \in \cZint$. Then, for any $p \in [1, 2]$, 
\begin{align}\label{eq:chi}
    \normInline{\ground{A}{\zground}}_{\cS_p} \leq \Upsilon_p(A) \defeq \begin{cases}
        \Uq{A}{\infty, 2p/(2-p)}, & \cX = \Delta^n, \\
        \Uq{A}{2, 2p/(2-p)}, & \cX = \B^n .
    \end{cases}
\end{align}
\end{restatable}

\Cref{lemma:gen-conversion} allows us to obtain instance-dependent rates in Section~\ref{subsec:ell2ell1-ell1ell1-putting-all-together} (Theorem~\ref{thm:general-norm-result}). However, to obtain Theorems~\ref{intro:l1l1} and Theorems~\ref{intro:l2l1} (our worst-case guarantees), we only require a special case of Lemma~\ref{lemma:gen-conversion} corresponding to $p = 2$. To make the results easier to parse, we prove this special case separately in Lemma~\ref{lemma:conversion} in Appendix~\ref{apx:linear-algebra} and also make the following remark. 

\begin{remark}\label{remark:general-conversion} Let $\zground \in \cZint$, then $\Upsilon_2(A) \leq 1$.
\end{remark}
\begin{proof} When $\cX = \B^n$, $\normInline{A}_{2 \to \infty} \leq 1$ by assumption. Indeed, observe that when $p = 2$,  Lemma~\ref{lemma:gen-conversion} implies $\Upsilon_2(A) = \max_{i \in [m]} \normInline{A_{i,:}}_2 = \normInline{A}_{2\to\infty} \leq 1$. Likewise, when $\cX = \Delta^n$, $\normInline{A}_{\max} \leq 1$ by assumption. Indeed, observe that when $p = 2$,  Lemma~\ref{lemma:gen-conversion} implies $\Upsilon_2(A) = \max_{i \in [m]} \normInline{A_{i,:}}_\infty = \normInline{A}_{\max} \leq 1$.
\end{proof}

Finally, we prove the following \Cref{lem:b-search-complexity-when-passmodel-false}, which will be used in Section~\ref{subsec:ell2ell1-ell1ell1-putting-all-together}.

\begin{lemma}[Algorithm~\ref{alg:binary-search} complexity when $\passModel = \false$]
    \label{lem:b-search-complexity-when-passmodel-false}
Consider inputs $\epsilon, \tau > 0$, $\beta \in (0, 3)$, $\zcenter \in \ztrunc$, $M_0 \in \R^{m \times n}$, $\passModel = \false$, a $p$-smooth guilty $\judge$ (e.g., Algorithm~\ref{alg:judge-l2l2}), $p \in [1, 2]$, and $\Upsilon_p(A) \leq \Xi$. Moreover, suppose that $\nu^{-1}$ and $\beta^{-1}$ are each bounded above by a polynomial in $m$, $n$, and $1 / \epsilon$. Then Algorithm~\ref{alg:binary-search} makes at most $\Otilde(\tau / \beta + \Xi^p / \tau^p)$ matrix-vector queries.
\end{lemma}

\begin{proof}
We first bound the number of matrix-vector queries made each time $\checkdiv$ is called in Algorithm~\ref{alg:binary-search}. Note that $\checkdiv$ is called once in Line~\ref{algline:b-search-check-div-beta-call} and potentially many times in Line~\ref{algline:b-search-check-div-call-repeat-loop}; all that will matter in the following is that each time it is called, the first argument (which we will now call $\alpha$ per Line~\ref{algline:checkdiv-function}) is at least $\beta$. Then looking at the $\checkdiv$ function itself starting in Line~\ref{algline:checkdiv-function}, it is clear that at most $O(1)$ matrix-vector queries are made in all lines other than the $\CWF$ wrapper function call in Line~\ref{algline:checkdiv-CAPW-call}. For the latter, note $M \gets 0$ due to Line~\ref{algline:check-div-tau-M-when-passmodel-false} since $\passModel = \false$. Then by combining the assumption on $\nu^{-1}$ and $\beta^{-1}$, the fact that $\alpha \ge \beta$, and the choice of $\gammapw$ and $\gammab$ in Line~\ref{algline:checkdiv-set-C-and-gammas} with Definition~\ref{def:CAPW} and Corollary~\ref{cor:SM-mirror-prox-complexity-when-passmodel-is-false} with $\zeta \gets \Xi$ due to Lemma~\ref{lemma:gen-conversion} gives that the $\CWF$ wrapper function call in Line~\ref{algline:checkdiv-CAPW-call} requires at most $\Otilde(\tau / \beta + \Xi^p / \tau^p)$ matrix-vector queries. We conclude by noting that due to the assumption on $\nu^{-1}$ and $\beta^{-1}$ as well as Lemma~\ref{lem:b-search-correctness}, $\checkdiv$ is called at most $\Otilde(1)$ times in Algorithm~\ref{alg:binary-search}.
\end{proof}

\subsection{Solving $\ell_1$-$\ell_1$ and $\ell_2$-$\ell_1$ games in $\tilde{O}(\epsilon^{-8/9})$-matrix-vector queries}
\label{subsec:ell2ell1-ell1ell1-putting-all-together}

In this section, we give our algorithms for the $\ellTwoEllOne$ and $\ellOneEllOne$ setups in Algorithm~\ref{alg:l2l1-l1l1-final-algo}, which is an instantiation of Algorithm~\ref{alg:proximal-point-regret} from Section~\ref{sec:prox-point-regret}, and prove the $\Otilde(\epsilon^{-8/9})$ complexity for $\ellOneEllOne$ and $\ellTwoEllOne$ games in Theorems \ref{intro:l1l1} and \ref{thm:l2l1-1/9-result} respectively. These theorems are a corollary of a more general instance-dependent result which we give in Theorem~\ref{thm:general-norm-result}. 

Before covering the pseudocode of Algorithm~\ref{alg:l2l1-l1l1-final-algo}, analogously to Section~\ref{subsec:binary-search}, we define another wrapper function for Algorithm~\ref{alg:mirror-prox-sug-strongly-monotone} in Definition~\ref{def:GAPW} to clarify, per Lemma~\ref{lem:GAPW-properties} given below, the approximations achieved by calls to Algorithm~\ref{alg:mirror-prox-sug-strongly-monotone} in the pseudocode. To make them easier to remember, the subscripts in $\gammav$ and $\gammagb$ in Definition~\ref{def:GAPW} stand for ``variational'' and ``global Bregman'' respectively.

\begin{definition}[$\GWF$]
    \label{def:GAPW}
    With $\zcenter, \zground, \ztilde, \tau, \alpha, A, M, c, \judge$ as in Definition~\ref{def:CAPW}, and additionally given $\gammav, \gammagb > 0$, 
a \emph{$(\zcenter, \zground, \ztilde, \tau, \alpha, A, M, c, \gammav, \gammagb; \judge)$-global wrapper function} returns
\begin{align*}
    \GWF(\zcenter, \zground, \ztilde, \tau, \alpha, A, M, c, \gammav, \gammagb; \judge) \defeq
\SUGStronglyMonotoneMirrorProx(z_c, \zground, \tau, \alpha, \epsprim, A, M, c, \cZ'; \judge),
\end{align*}
where, for an absolute constant $M' > 0$ defined in Lemma~\ref{lem:GAPW-properties}, we let
\begin{align*}
    Z' \defeq \sball_{c, \ztilde, \nu} = \sball_{c, \ztilde} \cap \ztrunc ~~~\text{and}~~~ 
    \epsprim \defeq M' \min \inbraces*{\frac{\gammagb^2}{1 + \log^2 \nu^{-1}}, \frac{\gammav^2}{1 + \alpha^2 \nu^{-2}}  }.
\end{align*}
\end{definition}

We now give the properties that the $\GWF$ wrapper function of Definition~\ref{def:GAPW} satisfies in Lemma~\ref{lem:GAPW-properties}. In particular, as long as $\zopt \defeq \prox_{\zcenter}^\alpha(\gm f_A ; \ztrunc)$ is in a multiplicative ball $\sball_{c, \ztilde, \nu}$ about $\ztilde$ (we will ensure this is always the case when $\GWF$ is called), the output of the $\GWF$ call $w$ satisfies an approximate variational inequality over $\ztrunc$, which will ultimately be used to implement the variational condition in the $\DAPO$ oracle per Definition~\ref{def:DAPO}. Furthermore, the $\GWF$ ensures $\breg{\zcenter}{w}$ approximates $\breg{\zcenter}{\zopt}$, which, combined with the guarantees for the binary search procedure in the previous section, will serve to ensure the $\DAPO$ oracle is kinetic.

\begin{lemma}[$\GWF$ properties]
    \label{lem:GAPW-properties}
Letting $w \gets \GWF(\zcenter, \zground, \ztilde, \tau, \alpha, A, M, c, \gammav, \gammagb; \judge) $ for some inputs satisfying the specifications of Definition~\ref{def:GAPW} and an appropriate choice of the absolute constant $M' > 0$, and supposing $\zopt \defeq \prox_{\zcenter}^\alpha(\gm f_A ; \ztrunc)$ satisfies $\zopt \in \sball_{c, \ztilde, \nu}$, we have
\begin{align*}
    \gm f_A(w)^\top (w - u) &\le \alpha \insquare{\breg{\zcenter}{u} - \breg{w}{u} - \breg{\zcenter}{w}} + \gammav,      ~~~\text{for all $u \in \ztrunc$, and} \\
     \inabs{\breg{\zcenter}{w} - \breg{\zcenter}{\zopt}} &\le  \gammagb.           &&
\end{align*}
\end{lemma}

\begin{proof}
The fact that $\zopt \in \sball_{c, \ztilde, \nu}$ implies $\breg{\zopt}{w} \le \epsilon'$ by Theorem~\ref{thm:main-inner-loop-convergence} for $\epsilon'$ defined in Definition~\ref{def:GAPW}. Then $\inabs{\breg{\zcenter}{w} - \breg{\zcenter}{\zopt}} \le  \gammagb$ is immediate from Lemma~\ref{lem:breg-error-to-breg-error-from-base}. As for the first property, by definition we have
\begin{align}
    \label{eq:def-zopt}
    \gm f_A(\zopt)^\top (\zopt - u) \le \alpha \insquare{\breg{\zcenter}{u} - \breg{\zopt}{u} - \breg{\zcenter}{\zopt}},  ~~~\text{for all $u \in \ztrunc$}.
\end{align}
Then, fixing an arbitrary $u \in \ztrunc$, rearranging \eqref{eq:def-zopt}, and adding terms to both sides yields
\begin{align*}
    & \gm f_A(w)^\top (w - u) - \alpha \insquare{\breg{\zcenter}{u} - \breg{w}{u} - \breg{\zcenter}{w}} \\
    \le & \underbrace{\gm f_A(w)^\top (w - u) - \gm f_A(\zopt)^\top (\zopt - u)}_{\oneC} + \underbrace{\alpha \insquare*{\breg{w}{u} - \breg{\zopt}{u}}}_{\twoC} + \underbrace{\alpha \insquare*{\breg{\zcenter}{w} - \breg{\zcenter}{\zopt}}}_{\threeC}.
\end{align*}
It suffices to show the last equation is bounded above by $\gammav$ for all $u \in \ztrunc$; we will proceed by bounding the terms $\oneC$, $\twoC$, and $\threeC$ one by one. First of all, we have
\begin{align*}
    \oneC &= \insquare{ \gm f_A(w) - \gm f_A(\zopt) + \gm f_A(\zopt) }^\top (w - u) - \gm f_A(\zopt)^\top (\zopt - u) \\
    &= \insquare{ \gm f_A(w) -  \gm f_A(\zopt) }^\top (w - u) + \gm f_A(\zopt)^\top (w - \zopt),
\end{align*}
in which case we have for all $u \in \ztrunc$:
\begin{align*}
    \inabs{\circled{1}} &\le \dualnorm{\gm f_A(w) -  \gm f_A(\zopt)}  \norm{w - u} + \dualnorm{\gm f_A(\zopt)} \norm{w - \zopt} \\
    &\overle{(i)} \norm{w - \zopt}  \norm{w - u} + 2            \norm{w - \zopt} \\
     &\le 4 \norm{w - \zopt} 
    \le 4 \sqrt{2 \breg{\zopt}{w}},
\end{align*}
where $(i)$ used the assumptions made at the beginning of Section~\ref{sec:elltwoelloneandelloneellone}, namely 
$\norm{A}_{2 \to \infty}  \le 1$ %
in the $\ell_2$-$\ell_1$ setup and $\norm{A}_{\text{max}} \le 1$ in the $\ellOneEllOne$ setup. Indeed, it is straightforward to show that $\gm f_A$ is 1-Lipschitz with respect to $\norm{\cdot}$ in both settings (recall $\norm{\cdot}$ is given by Definition~\ref{def:matrix-vector-games-setups} and Table \ref{table:setups}), and also to show that $\dualnorm{\gm f_A(z)}^2 \le 2$ for all $z \in \zset$ in both settings. (See Sections 4.1 and 4.2 in \cite{carmon2019variance} for explicit expressions for the dual norm, and also a source for the Lipschitz bound.)

As for $\twoC$, Lemmas \ref{lem:squared-norm-by-norm} and \ref{lem:bounding-KL-TO-THE-SAME-POINT} give that for some absolute constant $M_1 > 0$ and all $u \in \ztrunc$:
\begin{align*}
    \inabs{\twoC} = \alpha \inabs{\breg{w}{u} - \breg{\zopt}{u}} &\le  \alpha \inabs{\xbreg{w\x}{u\x} - \xbreg{\zopt\x}{u\x}} + \alpha \inabs{\ybreg{w\y}{u\y} - \ybreg{\zopt\y}{u\y}} \\
    & \le \frac{ M_1 \alpha}{\nu} \sqrt{\breg{\zopt}{w}}.
\end{align*}
Finally, Lemmas \ref{lem:KL-by-norm} and \ref{lem:squared-norm-by-norm} give that for some absolute constant $M_2 > 0$:
\begin{align*}
    \inabs{\threeC} = \alpha \inabs{\breg{\zcenter}{w} - \breg{\zcenter}{\zopt}} &\le \alpha \inabs{\xbreg{\zcenterx}{w\x} - \xbreg{\zcenterx}{\zopt\x}} + \alpha \inabs{\ybreg{\zcentery}{w\y} - \ybreg{\zcentery}{\zopt\y}} \\
    & \le M_2 \alpha (1 + \log \nu^{-1}) \sqrt{\breg{\zopt}{w}}.
\end{align*}
Then putting everything together, we've shown that for some absolute constant $M_3 > 0$ and all $u \in \ztrunc$:
\begin{align*}
    \oneC + \twoC + \threeC \le M_3 \inparen*{1 + \frac{\alpha}{\nu}} \cdot \sqrt{\breg{\zopt}{w}},
\end{align*}
and we conclude by the choice of $\epsilon'$ in Definition~\ref{def:GAPW}.
\end{proof}

Let us now step through the pseudocode of Algorithm~\ref{alg:l2l1-l1l1-final-algo}, where we analogously to Algorithm~\ref{alg:proximal-point-regret}, we define the condition in Line~\ref{algline:ell12ell1-final-while-loop} of Algorithm~\ref{alg:l2l1-l1l1-final-algo} to be $\true$ when $k = 0$. Furthermore, we note that the second return value of the call to $\GWF$ in Line~\ref{algline:ell12ell1-final-GAPW-call}, which contains the number of model-update steps made by the corresponding call to Algorithm~\ref{alg:mirror-prox-sug-strongly-monotone}, is only used in the analysis and not algorithmically (hence the dot).

Recall from the discussion in Section~\ref{subsec:binary-search} that the $\passModel$ flag determines whether models $M$ are passed between separate calls to Algorithm~\ref{alg:mirror-prox-sug-strongly-monotone}. The ultimate guarantees in this section assume $\passModel \gets \false$, and we defer discussion of the case where $\passModel \gets \true$ to Section~\ref{sec:four-fifths}. When $\passModel \gets \false$, the models in Algorithm~\ref{alg:l2l1-l1l1-final-algo} can be ignored since they are always cleared in Line~\ref{algline:ell12ell1-final-tau-M-when-clearmodel} before calling Algorithm~\ref{alg:mirror-prox-sug-strongly-monotone} in Line~\ref{algline:ell12ell1-final-GAPW-call}. 

Then, Algorithm~\ref{alg:l2l1-l1l1-final-algo} is an instantiation of Algorithm~\ref{alg:proximal-point-regret} for the operator $\gm f_A$ and dgf setup $(\ztrunc, \norm{\cdot}, r)$, with $\norm{\cdot}$ and $r$ given by Definition~\ref{def:matrix-vector-games-setups} and Table \ref{table:setups}. Indeed, letting $\zopt_{\alpha_k} \defeq \prox_{z^{k - 1}}^{\alpha_k}(\gm f_A; \ztrunc)$, the $\bsearch$ (Algorithm~\ref{alg:binary-search}) call in Line~\ref{algline:ell12ell1-final-bsearch-call} serves to ensure, due to Lemma~\ref{lem:b-search-correctness}, that either $\alpha_k = \beta$ and $\breg{z^{k - 1}}{\zopt_{\alpha_k}} \le 2.8 \alpha_k^2$, or else $\breg{z^{k - 1}}{\zopt_{\alpha_k}} \in [1.2 \alpha_k^2, 2.8 \alpha_k^2]$. In either case, having defined $\ztilde$ to be the $\alpha$-best response (Definition~\ref{def:alpha-best-response}) to $z^{k - 1}$ in Line~\ref{algline:ell12ell1-final-best-response}, the stability result Lemma~\ref{lem:stability-wrt-best-response} implies $\zopt_{\alpha_k}$ is in a multiplicative ball around $\ztilde$. Therefore, having set the local norm point $\zground$ to $\ztilde$ in Line~\ref{algline:ell12ell1-final-tau-M-when-clearmodel} (any point in $\sball_{C, \ztilde, \nu}$ would work), the assumption in Lemma~\ref{lem:GAPW-properties} is satisfied, and thus the approximations in Lemma~\ref{lem:GAPW-properties} combined with the guarantees due to Lemma~\ref{lem:b-search-correctness} discussed above ensure the $\GWF$ call in Line~\ref{algline:ell12ell1-final-GAPW-call} implements $(\beta, 2)$-kinetic $\DAPO$ (Definition~\ref{def:DAPO}). We thereby obtain an iteration bound for Algorithm~\ref{alg:l2l1-l1l1-final-algo} due to Lemma~\ref{lem:prox-point-iteration-bound}, and combine this with bounds on the total complexity of each call to Algorithm~\ref{alg:mirror-prox-sug-strongly-monotone} (Lemma~\ref{lem:b-search-complexity-when-passmodel-false} and Corollary~\ref{cor:SM-mirror-prox-complexity-when-passmodel-is-false}). We conclude by optimizing over $\beta$ and the smoothness threshold $\tau$.

\RestyleAlgo{ruled}
\DontPrintSemicolon
\SetKwComment{Comment}{/* }{ */}
\begin{algorithm2e}[h!]
\caption{Algorithm for the $\ellTwoEllOne$ and $\ellOneEllOne$ setups}
\label{alg:l2l1-l1l1-final-algo}
\KwInput{Matrix-vector oracle for $A$ s.t. $\innorm{A}_{2 \to \infty} \le 1$ in the $\ellTwoEllOne$ setup or $\inmaxnorm{A} \le 1$ in the $\ellOneEllOne$ setup}
\KwInput{Precision $\epsilon > 0$, minimum subproblem regularization level $\beta \in (0, 3)$}
\KwInput{Smoothness threshold $\tau > 0$}
\KwParameter{Judge function $\judge$, model-passing flag $\passModel$ (takes value $\true$ or $\false$), $p \in [1, 2]$, and an upper bound $\Upsilon_p(A) \leq \Xi$ (see \eqref{eq:chi})}

$z^0 \gets \argmin_{z \in \ztrunc} r(z)$, ~$M_0 \gets 0$, ~$C \defeq 500$, ~and~ $k \gets 0$ \tcp*{$M_0 \in \R^{m \times n}$} \label{algline:ell12ell1-final-z^0-M_0-C-k}

\While(\tcp*[f]{Recall $\Range$ for both setups is given in Table \ref{table:setups}}){$\sum_{j \in [k]} \alpha_j^{-1} < \Range / \epsilon$}{ \label{algline:ell12ell1-final-while-loop}

    $k \gets k + 1$ \;

    $(\alpha_k, M_{k - 1/2}) \gets \bsearch(\epsilon, z^{k - 1}, \beta, A, M_{k - 1}, \tau; \judge, \passModel, p, \Xi)$ \label{algline:ell12ell1-final-bsearch-call}

    \tcp{Compute $\alpha_k$-best response to $z^{k - 1}$ (Definition~\ref{def:alpha-best-response})}
    $\ztilde \defeq (\xtilde, \ytilde)$ ~for~ $\xtilde \gets \argmin_{x \in \xtrunc} {z^{k - 1}\y}^\top A x + \alpha_k \xbreg{z^{k - 1}\x}{x}$ ~and $\ytilde \gets \argmax_{y \in \ytrunc} y^\top A z^{k - 1}\x - \alpha_k \ybreg{z^{k - 1}\y}{y}$ \label{algline:ell12ell1-final-best-response}

    \lIf{\passModel}{
        $\zground \gets \argmin_{z \in \sball_{C, \ztilde, \nu}} \breg{z^{k - 1}}{z}$ \label{algline:ell12ell1-final-pass-model-true-zground-tau-choice}
    }

    \lElse(\tcp*[f]{Always clear the model if $\passModel = \false$}){
    $\zground \gets \ztilde$ ~and~ $M_{k - 1/2} = 0$ \label{algline:ell12ell1-final-tau-M-when-clearmodel}
    }

    $(z^k, \cdot, M_k) \gets \GWF(z^{k - 1}, \zground, \ztilde, \tau, \alpha_k, A, M_{k - 1/2}, C, \epsilon, \alpha_k^2 / 10; \judge)$ \label{algline:ell12ell1-final-GAPW-call}  \tcp{$\gm f_A (z^k)^\top (z^k - u) \le \alpha_k \insquare{ \breg{z^{k - 1}}{u} - \breg{z^{k}}{u}     -  \breg{z^{k - 1}}{z^k} } + \epsilon, \, \forall u \in \ztrunc$; Def. \ref{def:GAPW}, Lem. \ref{lem:GAPW-properties}, \ref{lem:connecting-GAPW-to-DAPO}}  
}

\Return{$\zbar \defeq \frac{1}{S} \sum_{j \in [K]} \alpha_j^{-1} z^j$, where $K \defeq k$ and $S \defeq \sum_{j \in [K]} \alpha_j^{-1}$}\label{algline:ell12ell1-final-return}

\end{algorithm2e}

Toward formalizing this discussion, we first show it is indeed the case that $\zopt_{\alpha_k}$ is contained in a multiplicative ball around $\ztilde$ at every iteration.

\begin{lemma}
    \label{lem:final-algo-prox-point-in-mult-ball}
    For every iteration $k \ge 1$ of Algorithm~\ref{alg:l2l1-l1l1-final-algo}, we have $\prox^{\alpha_k}_{z^{k - 1}}(\gm f_A; \ztrunc) \in \sball_{C, \ztilde, \nu}$.
\end{lemma}

\begin{proof}
Let $\zopt_{\alpha_k} \defeq \prox^{\alpha_k}_{z^{k - 1}}(\gm f_A; \ztrunc)$ for brevity, and note Line~\ref{algline:ell12ell1-final-bsearch-call} and Lemma~\ref{lem:b-search-correctness} imply in particular that $\breg{z^{k - 1}}{\zopt_{\alpha_k}} \le 3 \alpha_k^2$, in which case Lemma~\ref{lem:stability-wrt-best-response} gives $\exp(-2 \sqrt{6}) \le [\zopt_{\alpha_k \mathsf{y}}]_i / [\ztilde\y]_i \le \exp(2 \sqrt{6})$ for all $i \in [m]$ in the $\ellTwoEllOne$ setup, and $\exp(-2 \sqrt{6}) \le [\zopt_{\alpha_k}]_j / [\ztilde]_j \le \exp(2 \sqrt{6})$ for all $j \in [d]$ in the $\ellOneEllOne$ setup. Then $\zopt_{\alpha_k} \in \sball_{C, \ztilde, \nu}$ follows since $C = 500 > \exp(2 \sqrt{6})$. 
\end{proof}

Then, Lemma~\ref{lem:connecting-GAPW-to-DAPO} shows that Algorithm~\ref{alg:l2l1-l1l1-final-algo} is indeed an instantiation of Algorithm~\ref{alg:proximal-point-regret} for the operator $\gm f_A$ and dgf setup $(\ztrunc, \norm{\cdot}, r)$, with $\norm{\cdot}$ and $r$ given by Definition~\ref{def:matrix-vector-games-setups} and Table \ref{table:setups}.

\begin{lemma}
    \label{lem:connecting-GAPW-to-DAPO}
For all $k \ge 1$, the output $(z^k, \cdot, \cdot)$ of the $\GWF$ wrapper function call in Line~\ref{algline:ell12ell1-final-GAPW-call} of Algorithm~\ref{alg:l2l1-l1l1-final-algo} satisfies
\begin{align*}
    \gm f_A (z^k)^\top (z^k - u) \le \alpha_k \insquare{ \breg{z^{k - 1}}{u} - \breg{z^{k}}{u}     -  \breg{z^{k - 1}}{z^k} } + \epsilon, ~~~\text{for all $u \in \ztrunc$},
\end{align*}
and at least one of the following: $\alpha_k = \beta$ or $\breg{z^{k - 1}}{z^k} \ge \alpha_k^2$. In other words, Line~\ref{algline:ell12ell1-final-GAPW-call} of Algorithm~\ref{alg:l2l1-l1l1-final-algo} implements a $(\beta, 2)$-kinetic $\DAPO$ (Definition~\ref{def:DAPO}) for the operator $\gm f_A$ and the dgf setup $(\ztrunc, \norm{\cdot}, r)$, with $\norm{\cdot}$ and $r$ given by Definition~\ref{def:matrix-vector-games-setups} and Table \ref{table:setups}.
\end{lemma}

\begin{proof}
The approximate variational inequality is immediate from Lemmas \ref{lem:GAPW-properties} and \ref{lem:final-algo-prox-point-in-mult-ball}, and the second claim follows from Lemmas \ref{lem:GAPW-properties}, \ref{lem:final-algo-prox-point-in-mult-ball}, and \ref{lem:b-search-correctness} along with a triangle inequality.
\end{proof}

Next, we prove the correctness of Algorithm~\ref{alg:l2l1-l1l1-final-algo} and give an iteration bound in Corollary~\ref{cor:l2l1-l1l1-final-algo-correctness-iter-bound}. The correctness result is with respect to the truncated constraint sets, but recall that we have already argued in Lemma~\ref{lem:truncation-for-ell2ell1-ell1ell1} that it suffices to solve the truncated problem for an appropriate choice of $\nu$.

\begin{corollary}[Algorithm~\ref{alg:l2l1-l1l1-final-algo} correctness and iteration bound]
    \label{cor:l2l1-l1l1-final-algo-correctness-iter-bound}
    For any $\nu > 0$ such that the constraint sets are nonempty and inputs $\epsilon, \tau > 0$, $\beta \in (0, 3)$, judge function $\judge$, and $\passModel$ either $\true$ or $\false$, Algorithm~\ref{alg:l2l1-l1l1-final-algo} terminates after at most $\inparen{\beta / \epsilon + 2^{1/3} \epsilon^{- 2/3}} \Range + 2$ iterations (namely, $K$ defined in Line~\ref{algline:ell12ell1-final-return} is bounded by this quantity), and the output $\zbar$ is a $2 \epsilon$-solution of $\min_{x \in \xtrunc} \max_{y \in \ytrunc} y^\top A x$.
\end{corollary}

\begin{proof}
It is immediate from Lemma~\ref{lem:connecting-GAPW-to-DAPO} that Algorithm~\ref{alg:l2l1-l1l1-final-algo} is an instantiation of Algorithm~\ref{alg:proximal-point-regret} with $g \gets \gm f_A$ and dgf setup given by $(\ztrunc, \norm{\cdot}, r)$, with $\norm{\cdot}$ and $r$ given by Definition~\ref{def:matrix-vector-games-setups} and Table \ref{table:setups}. Thus, the result follows from Lemmas \ref{lemma:regret-bounds-error}, \ref{lem:prox-point-correctness}, \ref{lem:prox-point-iteration-bound}, and \ref{lem:connecting-GAPW-to-DAPO}.
\end{proof}

We now bound the complexity of Algorithm~\ref{alg:l2l1-l1l1-final-algo} when $\passModel = \false$.

\begin{lemma}[Algorithm~\ref{alg:l2l1-l1l1-final-algo}\label{lemma:general-overall-guarantee} complexity when $\passModel = \false$]
    \label{lem:l2l1-l1l1-final-algo-complexity-when-passmodel-false}
    Consider inputs $\epsilon, \tau > 0$, $\beta \in (0, 3)$, $\passModel = \false$, a 2-smooth guilty $\judge$ (e.g., Algorithm~\ref{alg:judge-l2l2}), $p \in [1, 2]$ and an upper bound $\Upsilon_p(A) \leq \Xi$. Suppose $\nu^{-1}$ and $\beta^{-1}$ are each bounded above by a polynomial in $m$, $n$, and $1 / \epsilon$. Then, Algorithm~\ref{alg:l2l1-l1l1-final-algo} makes at most
    \begin{align*}
    \Otilde(1) \cdot \inparen{\beta / \epsilon + \epsilon^{- 2/3} + 1} \inparen{\tau / \beta + \Xi^p / \tau^p + 1}\text{ matrix-vector queries.}
\end{align*}
\end{lemma}

\begin{proof}
By Corollary~\ref{cor:l2l1-l1l1-final-algo-correctness-iter-bound}, Algorithm~\ref{alg:l2l1-l1l1-final-algo} terminates after at most $\Otilde(\beta / \epsilon + \epsilon^{- 2/3})$ iterations, and each call to $\bsearch$ in Line~\ref{algline:ell12ell1-final-bsearch-call} requires at most $\Otilde(\tau / \beta + \Xi^p / \tau^p)$ matrix-vector queries due to Lemma~\ref{lem:b-search-complexity-when-passmodel-false}. Then, it suffices to show that the total number of queries made in each call to the $\GWF$ wrapper function in Line~\ref{algline:ell12ell1-final-GAPW-call} is also at most $\Otilde(\tau / \beta + \Xi^p / \tau^p)$. Indeed, note that $\alpha_k \ge \beta$ by Lemma~\ref{lem:b-search-correctness}, and also $M_{k - 1/2} \gets 0$ due to Line~\ref{algline:ell12ell1-final-tau-M-when-clearmodel} since $\passModel = \false$. Then the subclaim follows by combining the assumption on $\nu^{-1}$ and $\beta^{-1}$, the fact that $\alpha_k \ge \beta$, and the fact that $\gammav \gets \epsilon$ and $\gammagb \gets \alpha_k^2 / 10$ in the $\GWF$ call in Line~\ref{algline:ell12ell1-final-GAPW-call} with Definition~\ref{def:GAPW} and Corollary~\ref{cor:SM-mirror-prox-complexity-when-passmodel-is-false} with $\zeta \gets \Xi$ due to Lemma~\ref{lemma:gen-conversion}.
\end{proof}

Using Lemma~\ref{lemma:general-overall-guarantee}, we obtain the following general result for $\ell_2$-$\ell_1$ and $\ell_1$-$\ell_1$ games.

\begin{theorem}
    \label{thm:general-norm-result}
For any $p \in [1, 2]$, there is a deterministic algorithm that computes an $\epsilon$-solution for $\ell_2$-$\ell_1$ games where $\Upsilon_p(A) \leq \Xi$ 
using $\Otilde(\Xi^{p/(p+1)}\epsilon^{-(3p + 2)/(3(p+1))})$-matrix-vector queries. Likewise, for any $p \in [1, 2]$, there is a deterministic algorithm that computes an $\epsilon$-solution for $\ell_1$-$\ell_1$ games where $\Upsilon_p(A) \leq \Xi$ 
using $\Otilde(\Xi^{p/(p+1)}\epsilon^{-(3p + 2)/(3(p+1))} + \epsilon^{-2/3})$-matrix-vector queries.
\end{theorem}

\begin{proof}
A straightforward calculation shows that $\inparen{\beta / \epsilon + \epsilon^{- 2/3} + 1} \inparen{\tau / \beta + \Xi^p / \tau^p + 1}$ is minimized up to a constant multiplicative factor (for $p \in [1, 2]$) when $\beta \gets \epsilon^{1/3}$ and $\tau \gets \Xi^{p / (p + 1)} \epsilon^{1 / (3(p + 1))}$. Thus, the result follows from choosing $\nu$ as in Lemma~\ref{lem:truncation-for-ell2ell1-ell1ell1} and running Algorithm~\ref{alg:l2l1-l1l1-final-algo} with inputs as in Lemma~\ref{lem:l2l1-l1l1-final-algo-complexity-when-passmodel-false} along with $\beta \gets \epsilon^{1/3}$ and $\tau \gets \Xi^{p / (p + 1)} \epsilon^{1 / (3(p + 1))}$. The correctness follows from Corollary~\ref{cor:l2l1-l1l1-final-algo-correctness-iter-bound} and Lemma~\ref{lem:truncation-for-ell2ell1-ell1ell1} and the complexity bound follows from Lemma~\ref{lem:l2l1-l1l1-final-algo-complexity-when-passmodel-false}.
\end{proof}

Theorem~\ref{thm:general-norm-result} yields fine-grained rates when the payoff matrix $A$ satisfies strong bounds on $\Upsilon_p(A)$, and is in this sense, an instance-dependent rate. As an example, suppose that we know that the payoff matrix satisfies $\Upsilon_1(A) \leq 1$; then Theorem~\ref{thm:general-norm-result} implies an $\tilde{O}(\epsilon^{-5/6})$-matrix-vector query complexity. 

However, we can also invoke Theorem~\ref{thm:general-norm-result} to obtain a \emph{worst-case guarantee} as well. Indeed, using Remark~\ref{remark:general-conversion}, we have that $\Upsilon_2(A) \leq 1$, thus, we can always invoke Theorem~\ref{thm:general-norm-result} with $p = 2, \Xi = 1$. Hence, we achieve a \emph{worst-case} complexity of $\Otilde(\Range \epsilon^{-8/9})$-matrix-vector queries when $\passModel = \false$ below. This yields our ultimate result for zero-sum games.

\zerosummain*

We also obtain a matching query complexity for $\ell_2$-$\ell_1$ games, which is restated below. However, we will improve upon Theorem~\ref{thm:l2l1-1/9-result} in the next section.

\begin{theorem}
    \label{thm:l2l1-1/9-result}
There is a \emph{deterministic} algorithm that computes an $\epsilon$-solution for $\ell_2$-$\ell_1$ games using $\Otilde(\epsilon^{-8/9})$-matrix-vector queries.
\end{theorem}

\begin{proof}[Proof of Theorem~\ref{thm:l2l1-1/9-result} and Theorem~\ref{intro:l1l1}] These results follow immediately from Theorem~\ref{thm:general-norm-result} with $p = 2$ and $\Xi = 1$, because we always have that $\Upsilon_2(A) \leq 1$ due to Remark~\ref{remark:general-conversion}.
\end{proof}

In the case of $\ell_2$-$\ell_1$ games, we can obtain an improved query complexity of $\tilde{O}(\epsilon^{-7/9})$-matrix-vector queries by setting $\passModel = \true$ and applying a more careful analysis, which we discuss in the following subsection. 

\subsection{Solving $\ell_2$-$\ell_1$ games in $\tilde{O}(\epsilon^{-7/9})$-matrix-vector queries}\label{sec:four-fifths}

Finally, in the remainder of Section~\ref{sec:elltwoelloneandelloneellone} we operate specifically in the $\ellTwoEllOne$ setup of Definition~\ref{def:matrix-vector-games-setups} and prove Theorem~\ref{intro:l2l1}. In order to prove our query complexity of $\tilde{O}(\epsilon^{-7/9})$-matrix-vector queries for $\ell_2$-$\ell_1$ games, we require some additional technical lemmas, which we derive in Sections~\ref{sec:7/9-lin-alg-prelims} through~\ref{sec:movement-bound-ground-points}. Then, we prove our final result for $\ell_2$-$\ell_1$ games in Section~\ref{subsubsec:final-l2-l1}. 

The key idea in our analysis is to set the $\passModel$ flag to $\true$ in Algorithm~\ref{alg:l2l1-l1l1-final-algo} and demonstrate that the information learned about $A$ in one call to $\SUGStronglyMonotoneMirrorProx$ (Algorithm~\ref{alg:mirror-prox-sug-strongly-monotone}) can be leveraged in subsequent calls to reduce the overall matrix-vector query complexity of Algorithm~\ref{alg:l2l1-l1l1-final-algo}. In particular, recall from the discussion at the beginning of Section~\ref{subsec:smooth-until-proven-guilty-strongly-monotone} that each iteration of Algorithm~\ref{alg:mirror-prox-sug-strongly-monotone} (an iteration of Algorithm~\ref{alg:mirror-prox-sug-strongly-monotone} starts with the call to $\SUGSMStep$ in Line~\ref{line:mirror-prox-step-sm-final}) falls into one of two buckets: In a model-update step, Lines \ref{algline:SUG-SM-MP-update-D} through \ref{algline:SUG-SM-MP-update-k} execute, we make progress in learning the matrix $A$, and the counter storing the total number of model-update steps $k$ is increased in Line~\ref{algline:SUG-SM-MP-update-k}. Otherwise, in a progress iteration, Line~\ref{algline:SUG-SM-MP-update-j} executes and the counter storing the number of progress iterations $j$ is increased.

Note also that the second return value $k$ of Algorithm~\ref{alg:mirror-prox-sug-strongly-monotone} is precisely the total number of model-update steps performed in that call. Then supposing Algorithm~\ref{alg:mirror-prox-sug-strongly-monotone} is called $J$ times in total throughout the run of Algorithm~\ref{alg:l2l1-l1l1-final-algo} (the two locations where it can be called are Line~\ref{algline:ell12ell1-final-GAPW-call} of Algorithm~\ref{alg:l2l1-l1l1-final-algo} and Line~\ref{algline:checkdiv-CAPW-call} in Algorithm~\ref{alg:binary-search}), and letting $(\cdot, K_j, M_j)$ denote the output of the $j$-th call to Algorithm~\ref{alg:mirror-prox-sug-strongly-monotone}, we can express the total number of model-update steps performed over all calls to Algorithm~\ref{alg:mirror-prox-sug-strongly-monotone} as $\sum_{j \in [J]} K_j$. Ultimately, we obtain a tighter bound on this quantity in Lemma~\ref{lem:model-update-steps-bound} via a potential analysis reminiscent of that of Theorem~\ref{thm:main-inner-loop-complexity} and Corollary~\ref{cor:SM-mirror-prox-complexity-when-passmodel-is-false}. In particular, we argue in the proof of Lemma~\ref{lem:model-update-steps-bound} that $\innorm{\ground{A - M_j}{\zground^j}}_F^2$ is being driven toward 0 as the total number of model-update steps so far $\sum_{\ell \in [j]} K_\ell$ increases, where $\zground^j$ is the local norm point input (i.e., the second argument) passed to $\SUGStronglyMonotoneMirrorProx$ when it is called for the $j$-th time. However, the key difference compared to the analysis of Corollary~\ref{cor:SM-mirror-prox-complexity-when-passmodel-is-false}, which is happening within a single call to $\SUGStronglyMonotoneMirrorProx$ and for a fixed local norm point $\zground$, is that the local norm point $\zground^j$ is changing between calls to $\SUGStronglyMonotoneMirrorProx$, and therefore we must bound the total progress lost when $j$ changes in the potential $\innorm{\ground{A - M_j}{\zground^j}}_F^2$. Key ingredients for the latter are covered in Sections \ref{sec:7/9-lin-alg-prelims}, \ref{subsubsec:7/9-one-sided-proj}, and \ref{sec:movement-bound-ground-points}, before we put everything together in Section~\ref{subsubsec:final-l2-l1}.

\subsubsection{Linear-algebraic preliminaries for main $\ell_2$-$\ell_1$ games result}\label{sec:7/9-lin-alg-prelims}

In this section, we give some linear-algebraic preliminaries. First, we state the following lemma, which describes how the Frobenius norm of $\ground{A}{z'}$ relates to that of $\ground{A}{z}$ for any $z, z' \in \cZint$. As discussed above, this lemma will ultimately be used to reason about how the potential $\innorm{\ground{A - M_j}{\zground^j}}_F^2$ changes as $j$ increments.

\begin{lemma}\label{lemma:model-bound} Let $z, z' \in \cZint$ and $A \in \R^{m \times n}$. Then, 
\begin{align*}
    \normInline{\ground{A}{z'} }_F^2 \leq \normInline{\ground{A}{z}}_F^2 + \normInline{z\y - z'\y}_1 \cdot \normInline{A}_{2 \to \infty}^2.
\end{align*}
\end{lemma}
\begin{proof} For notational convenience, let $y = z\y, y' = z'\y$. By triangle inequality,
\begin{align*}
     \normInline{\diag({y'})^{1/2} A' }_F^2 &= \sum_{i \in [m], j \in [n]} [y']_i A_{ij}^2 = \sum_{i \in [m], j \in [n]} [y]_i A_{ij}^2 + \sum_{i \in [m], j \in [n]} ([y']_i - [y]_i) A_{ij}^2 \\
     &= \normInline{\diag(y)^{1/2} A}_F^2 + \sum_{i \in [m], j \in [n]} ([y']_i - [y]_i) A_{ij}^2. 
\end{align*}
So, it remains to bound the second term. Note that 
\begin{align*}
    \sum_{i \in [m], j \in [n]} ([y']_i - [y]_i) A_{ij}^2 &= \sum_{i \in [m]} ([y']_i - [y]_i) \sum_{j \in [n]} A_{ij}^2 \leq \sum_{i \in [m]} |[y']_i - [y]_i| \cdot \normInline{A_{i,:}}_2^2 \\
    &\leq \normInline{y - y'}_1 \cdot \normInline{A}_{2 \to \infty}^2.
\end{align*}
\end{proof}

We also require the following lemma, which guarantees that (repeatedly) applying a projection matrix to the right of $A \in \R^{m\times n}$ cannot grow its $(2 \to \infty)$-norm. 

\begin{definition} A matrix $P \in \R^{n \times n}$ is \emph{an orthogonal projection matrix} if $P^2 = P$ and $P^\top = P$. 
\end{definition}

\begin{lemma}\label{lemma:projection} Let $A \in \R^{m \times n}$ and $P = P_1 \cdots P_K \in \R^{n \times n}$ where each $P_k$ is an orthogonal projection matrix. Then, $\norm{A P}_{2 \to \infty} \leq \normInline{A}_{2 \to \infty}.$
\end{lemma}
\begin{proof} Without loss of generality, assume that each $P_k \neq 0$. By definition $ \norm{A P}_{2 \to \infty} = \max_{x \in \B^n}  \norm{A P x}_{\infty}.$ By submultiplicativity of the operator norm, for any $x \in \B^n$, we have
\begin{align}\label{eq:demonstrated}
    \normInline{Px}_2  = \norm{\prod_{k \in [k]} P_k x}_{2} \leq \normInline{x}_2 \prod_{k \in [k]} \normInline{P_k}_2 = \normInline{x}_2,  
\end{align}
where the third inequality holds because $P_k^2 = P_k \neq 0$ implies that $\normInline{P_k}_2^2 = \normInline{P_k} = 1$. Consequently, 
\begin{align*}
     \norm{A P}_{2 \to \infty} = \max_{x \in \B^n}  \norm{A P x}_{\infty} = \max_{x': x' = Px \text{ for some } x\in \B^n}  \norm{A x'}_{\infty} \leq \max_{x \in \B^n}  \norm{A x}_{\infty} = \normInline{A}_{2 \to \infty}.
\end{align*}
where inequality holds due to \eqref{eq:demonstrated}. 
\end{proof}

\subsubsection{One-sided projection judge $\judge_{\project}$}
\label{subsubsec:7/9-one-sided-proj}

In this section, we provide an alternative implementation of a $2$-smooth guilty judge, as shown in Algorithm~\ref{alg:2-proj-judge}, which enables our improvement over Theorem~\ref{thm:l2l1-1/9-result}. In comparison to $\judge_p$ (Algorithm~\ref{alg:judge-l2l2}) which is simultaneously a $p$-smooth-guilty judge for all $p \in [1, \infty)$, Algorithm~\ref{alg:2-proj-judge} performs only a \emph{one-sided} projection to $A$. The following lemma shows this is still sufficient to ensure that Algorithm~\ref{alg:2-proj-judge} is a $2$-smooth-guilty judge. We prove the following guarantees. 

\RestyleAlgo{ruled}
\SetKwComment{Comment}{/* }{ */}
\begin{algorithm2e}[ht]
\caption{$\judge_{\project}(z, z', B, \tau)$}
\label{alg:2-proj-judge}
\KwInput{Vectors $z, z' \in \R^d$, smoothness threshold $\tau > 0$.}
\KwInput{Matrix-vector oracle for a matrix $B \in \R^{m \times n}$.} 
\tcp{Check if pair of unit vectors has large bilinear form in $B$}
\lIf{$z\y^\top B z\x > \tau \normInline{z\y}_2\normInline{z\x}_2 \label{line:if1-2}$}{
    $u \gets \normalize(z\x)$
}
\lElseIf{${z'\y}^\top B z'\x > \tau \normInline{z'\y}_2\normInline{z'\x}_2$ \label{line:if2-2}}{
   $u \gets \normalize(z'\x)$ 
}
\lElse{
    \Return{$(\smooth, 0)$}
}
$D \gets B uu^\top$ \label{line:rank-one-update}\; 
\Return{$({\guilty},  D)$}
\end{algorithm2e}

\begin{restatable}{lemma}{smoothguiltyjudgeproj}\label{lemma:proj-judge} $\judge_{\project}$ (Algorithm~\ref{alg:2-proj-judge}) is a $2$-smooth-guilty-judge. Moreover, $B - D = BP$ where $P$ is an orthogonal projection matrix. 
\end{restatable}
\begin{proof} If neither of the if statements (Lines~\ref{line:if1-2} or \ref{line:if2-2}) execute, then the algorithm outputs $(\smooth, 0)$ as required. Otherwise, let $v = \normalize(z'\y)$ and observe that $v^\top B u > \tau$. We need to prove that $\normInline{Bu}_2 \geq \tau$ and $\norm{ B - D }_{F}^2 \leq \norm{B}_{F}^2 - \normInline{Bu}_2^2$. Because the Euclidean norm is self-dual, we have that $\normInline{Bu}_2 \defeq \max_{y \in \B^m} y^\top (Bu) \geq v^\top Bu \geq \tau. $ Now, note that $B - D = Buu^\top$. Hence,
\begin{align*}
    \normInline{B - D}_F^2 &= \tr(B^\top B - 2B^\top D + D^\top D) \\
    &= \normInline{B}_F^2 - 2 \tr(B^\top Buu^\top) + \tr(uu^\top B^\top B uu^\top) \\
    &= \normInline{B}_F^2 - 2 u^\top B^\top Bu + u^\top B^\top B u \\
    &= \normInline{B}_F^2 - \normInline{Bu}_2^2 \leq \normInline{B}_F^2 - \tau^2. 
\end{align*}
The algorithm requires only $O(1)$-matrix-vector queries, hence $\judge_{\project}$ is a $2$-smooth-guilty judge. For the final guarantee, note that $B-D = B(I - uu^\top)$ where $(I - uu^\top)$ is an orthogonal projection matrix.  
\end{proof}

Using $\judge_{\project}$ (Algorithm~\ref{alg:2-proj-judge}) in Theorems~\ref{thm:main-inner-loop-convergence} and \ref{thm:main-inner-loop-complexity}, we obtain the following.

\begin{lemma}[Complexity of Algorithm~\ref{alg:mirror-prox-sug-strongly-monotone} with $\judge_{\project}$]\label{lemma:main-inner-loop-convergence-four-fifths} Let $\epsilon > 0$. Let $z_c \in \cZint$ and $\zground \in \cZ' \subset \cB_{c_1, \Tilde{z}}$ for some absolute constant $c_1 > 1$ and $\Tilde{z} \in \cZint$. Let $A, M \in \R^{m \times n}$, and $\tau, \alpha > 0$.  Then,  
\begin{align*}
    (z, K, M') := \SUGStronglyMonotoneMirrorProx(z_c, \zground, \tau, \alpha, A, M, \epsilon, \cZ'; \judge_{\project})
\end{align*}
(Algorithm~\ref{alg:mirror-prox-sug-strongly-monotone}) satisfies the following guarantees:
\begin{enumerate}
    \item The algorithm terminates after making at most $O(K + (\tau/\alpha) \log(1/\epsilon))$-matrix-vector queries, 
    \item $\breg{\proxStepSimpleZ{z_c}{\alpha}{\nabla_\pm f_A }{\cZ'}}{z} \leq \epsilon$, 
    \item $\normInline{\ground{A - M'}{\zground}}_{F}^2 \leq \normInline{\ground{A - M}{\zground}}_{F}^2 - c_2 K \tau^2$, and
    \item $A-M' = (A-M)P$, where $P = \prod_{i \in [K]} P_i$ for orthogonal projection matrices $P_1, ..., P_k \in \R^{n \times n}$.
\end{enumerate}
\end{lemma}
\begin{proof} The first three guarantees follow directly from Lemma~\ref{lemma:proj-judge} and Theorem~\ref{thm:main-inner-loop-convergence} and Theorem~\ref{thm:main-inner-loop-complexity} applied with $p = 2$. Thus, it remains to prove the final property, which we will prove by induction. We will show that for each $k \in [K]$, $A - M_k = (A-M)P$ where $P$ is a product of $k$ orthogonal projection matrices.

In the base case, $M_0 = M$ and hence $A - M_0 = (A-M)I$ clearly satisfies the claim. Now, note that $M_{k+1} = M_{k} + \ground{D_{k+1}}{\zground, *}$ where, by Lemma~\ref{lemma:proj-judge}, $\ground{A_k}{\zground} - D_{k+1} = \ground{A_k}{\zground} P'$ for some orthogonal projection matrix $P'$. Hence, by Definition~\ref{def:setup-details}, we have 
\begin{align*}
    A - M_{k+1} &= A - M_{k} - \unground{D_{k+1}}{\zground} = A_k - \ground{D_{k+1}}{\zground, *} \\
    &= \unground{\ground{A_k}{\zground} - D_{k+1}}{\zground} = \unground{\ground{A_k}{\zground}P'}{\zground} = N_{\cY, \zground} N_{\cY, \zground}^{-1} A_k P' = A_k P' \\
    &= (A - M_k) P' = (A-M) P P' = (A-M) \bar{P}. 
\end{align*}
where the last line equality holds by the inductive hypothesis, for some $\bar{P}$ which is the product of $k$ orthogonal projection matrices. Hence, the last property holds by induction. 
\end{proof}

\subsubsection{Movement bound for the ground points}\label{sec:movement-bound-ground-points}

In Lemma~\ref{lem:movement-bound-local-norm-points} below, we obtain a bound on the sum of the distances between consecutive local norm points $\zground$ (namely, the second argument) passed to $\SUGStronglyMonotoneMirrorProx$ (Algorithm~\ref{alg:mirror-prox-sug-strongly-monotone}) during the run of Algorithm~\ref{alg:l2l1-l1l1-final-algo} when $\passModel = \true$. This is another key ingredient in bounding increases to the potential $\innorm{\ground{A - M_j}{\zground^j}}_F^2$ as $j$ increments; see the discussion at the beginning of Section~\ref{sec:four-fifths}. The proof of Lemma~\ref{lem:movement-bound-local-norm-points} relies on a sublemma, Lemma~\ref{lem:movement-bound-single-divergence-bound}, which, for a given iteration $k$ of Algorithm~\ref{alg:l2l1-l1l1-final-algo}, bounds $\breg{z^{k - 1}}{\zground} \le O(\breg{z^{k - 1}}{z^k})$ for any local norm point $\zground$ passed to Algorithm~\ref{alg:mirror-prox-sug-strongly-monotone} during iteration $k$. (Recall that calls to Algorithm~\ref{alg:mirror-prox-sug-strongly-monotone} happen in two places during the run of Algorithm~\ref{alg:l2l1-l1l1-final-algo}: in Line~\ref{algline:ell12ell1-final-GAPW-call} of Algorithm~\ref{alg:l2l1-l1l1-final-algo} itself [recall Definition~\ref{def:GAPW}], and in Line~\ref{algline:checkdiv-CAPW-call} of Algorithm~\ref{alg:binary-search} [recall Definition~\ref{def:CAPW}], which is in turn called in Line~\ref{algline:ell12ell1-final-bsearch-call} of Algorithm~\ref{alg:l2l1-l1l1-final-algo}.) Having proven Lemma~\ref{lem:movement-bound-single-divergence-bound}, Lemma~\ref{lem:movement-bound-local-norm-points} follows straightforwardly by combining the movement bound $\sum_{k \in [K - 1]} \breg{z^{k - 1}}{z^k} \le \Otilde(1)$ due to Lemma~\ref{lem:prox-point-correctness} with the fact that Algorithm~\ref{alg:mirror-prox-sug-strongly-monotone} is called at most $\Otilde(1)$ times during an iteration of Algorithm~\ref{alg:l2l1-l1l1-final-algo} due to Lemma~\ref{lem:b-search-correctness}.

Regarding the proof of Lemma~\ref{lem:movement-bound-single-divergence-bound}, this is a key place where we use the alternate behavior of Algorithms \ref{alg:binary-search} and \ref{alg:l2l1-l1l1-final-algo} due to setting $\passModel \gets \true$. Let us go over these changes now, starting with Algorithm~\ref{alg:binary-search}. Note that if $\passModel \gets \true$, the local norm point $\zground$ is set to $\argmin_{z \in \sball_{C, \ztilde, \nu}} \breg{\zcenter}{z}$ in Line~\ref{algline:check-div-zground-tau-when-passmodel-true} of Algorithm~\ref{alg:binary-search}. Assuming we are within iteration $k$ of Algorithm~\ref{alg:l2l1-l1l1-final-algo}, $\zcenter = z^{k - 1}$ (see Line~\ref{algline:ell12ell1-final-bsearch-call} of Algorithm~\ref{alg:l2l1-l1l1-final-algo}), so this is equivalent to $\argmin_{z \in \sball_{C, \ztilde, \nu}} \breg{z^{k - 1}}{z}$. Note that this is a valid choice of $\zground$ in terms of applying the guarantees of Algorithm~\ref{alg:mirror-prox-sug-strongly-monotone}; Algorithm~\ref{alg:mirror-prox-sug-strongly-monotone} does not care which point you use as the local norm as long as it is contained within $\sball_{C, \ztilde, \nu}$. Thus, the choice $\zground \gets \argmin_{z \in \sball_{C, \ztilde, \nu}} \breg{z^{k - 1}}{z}$ is optimal since our goal is to minimize $\breg{z^{k - 1}}{\zground}$. 

Then, the proof of $\breg{z^{k - 1}}{\zground} \le O \inparen{ \breg{z^{k - 1}}{z^k} }$ in Lemma~\ref{lem:movement-bound-single-divergence-bound} proceeds by a case-by-case argument depending on the different ways this call to $\checkdiv$ might return. In particular, the choice of $\zground \gets \argmin_{z \in \sball_{C, \ztilde, \nu}} \breg{z^{k - 1}}{z}$ is important because, assuming $\zopt_\alpha \defeq  \prox_{z^{k - 1}}^\alpha (\gm f_A; \zset_\nu) \in \sball_{C, \ztilde, \nu}$ (e.g., this happens if the execution of $\checkdiv$ reaches Line~\ref{algline:check-div-breg-check-toobig} per Lemma~\ref{lem:checkdiv-correctness}), we have $\breg{z^{k - 1}}{\zground} \le \breg{z^{k - 1}}{\zopt_\alpha}$ and can then conclude certain subcases by bounding $\breg{z^{k - 1}}{\zopt_\alpha} \le O(\breg{z^{k - 1}}{z^k})$ (see the proof for details). 

Another important ingredient is the ``early return'' without calling Algorithm~\ref{alg:mirror-prox-sug-strongly-monotone} in Line~\ref{algline:check-div-movement-return} if $\breg{\zcenter}{\zground} = \breg{z^{k - 1}}{\zground}$ is too big, specifically $\breg{z^{k - 1}}{\zground} > 3 \alpha^2$. Here, the choice of $\zground \gets \argmin_{z \in \sball_{C, \ztilde, \nu}} \breg{z^{k - 1}}{z}$ is again important because it allows us to justify returning $\tooBig$ in this line if $\breg{\zcenter}{\zground} > 3 \alpha^2$ (see the proof of Lemma~\ref{lem:checkdiv-correctness}) \emph{without} needing to call Algorithm~\ref{alg:mirror-prox-sug-strongly-monotone} with this ``bad'' local norm point which potentially makes $\breg{\zcenter}{\zground} = \breg{z^{k - 1}}{\zground}$ too big.

Finally, regarding Algorithm~\ref{alg:l2l1-l1l1-final-algo} itself, $\zground$ is set to $\argmin_{z \in \sball_{C, \ztilde, \nu}} \breg{z^{k - 1}}{z}$ in Line~\ref{algline:ell12ell1-final-pass-model-true-zground-tau-choice} when $\passModel \gets \true$ for similar reasons to those discussed above. We now state and prove Lemma~\ref{lem:movement-bound-single-divergence-bound}:

\begin{lemma}
    \label{lem:movement-bound-single-divergence-bound}
Let $\zground$ denote some local norm point passed to $\SUGStronglyMonotoneMirrorProx$ (Algorithm~\ref{alg:mirror-prox-sug-strongly-monotone}) during iteration $k$ of Algorithm~\ref{alg:l2l1-l1l1-final-algo}, where $\passModel = \true$. Then $\breg{z^{k - 1}}{\zground} \le O \inparen{ \breg{z^{k - 1}}{z^k} }$.
\end{lemma}

\begin{proof}
Note that Algorithm~\ref{alg:mirror-prox-sug-strongly-monotone} may be called in two places during an iteration of Algorithm~\ref{alg:l2l1-l1l1-final-algo}: in Line~\ref{algline:ell12ell1-final-GAPW-call} of Algorithm~\ref{alg:l2l1-l1l1-final-algo} (recall the $\GWF$ wrapper per Definition~\ref{def:GAPW}) and in Line~\ref{algline:checkdiv-CAPW-call} of Algorithm~\ref{alg:binary-search} (due to the $\bsearch$ call in Line~\ref{algline:ell12ell1-final-bsearch-call} of Algorithm~\ref{alg:l2l1-l1l1-final-algo}; recall also the $\CWF$ wrapper per Definition~\ref{def:CAPW}). As for the former, clearly $z^k \in \sball_{C, \ztilde, \nu}$ since Algorithm~\ref{alg:mirror-prox-sug-strongly-monotone} returns a point in this set given the inputs of Line~\ref{algline:ell12ell1-final-GAPW-call} of Algorithm~\ref{alg:l2l1-l1l1-final-algo}, in which case the claim is immediate from the choice of $\zground$ in Line~\ref{algline:ell12ell1-final-pass-model-true-zground-tau-choice} in Algorithm~\ref{alg:l2l1-l1l1-final-algo}. 

Now consider the latter, namely Line~\ref{algline:checkdiv-CAPW-call} of Algorithm~\ref{alg:binary-search}, which is part of the $\checkdiv$ subroutine starting in Line~\ref{algline:checkdiv-function} of Algorithm~\ref{alg:binary-search}. This $\checkdiv$ subroutine call could originate from either Line~\ref{algline:b-search-check-div-call-repeat-loop} or Line~\ref{algline:b-search-check-div-beta-call} of Algorithm~\ref{alg:binary-search}; we will refer to these as $\checkdiv$ Case 1 and $\checkdiv$ Case 2 respectively below.

\textit{$\checkdiv$ Case 1:} Note that because we did not return early in Line~\ref{algline:check-div-movement-return} of Algorithm~\ref{alg:binary-search} from this $\checkdiv$ call, we have $\breg{z^{k - 1}}{\zground} \le 3 \alpha^2$. (Here, $\alpha$ has been passed as input to this $\checkdiv$ call per Line~\ref{algline:checkdiv-function}, and note also that $z_c$ in Algorithm~\ref{alg:binary-search} is equal to $z^{k - 1}$ by Line~\ref{algline:ell12ell1-final-bsearch-call} of Algorithm~\ref{alg:l2l1-l1l1-final-algo}.) Note also that in this case, Algorithm~\ref{alg:binary-search} must eventually return in Line~\ref{algline:b-search-return-finished-b-search} with the final value of $\alpha_m$ being set to $\alpha_k$ in Line~\ref{algline:ell12ell1-final-bsearch-call} of Algorithm~\ref{alg:l2l1-l1l1-final-algo}. This implies $\breg{z^{k - 1}}{z^k} \ge \alpha_k^2$ in particular by Lemma~\ref{lem:connecting-GAPW-to-DAPO}, since $\alpha_k > \beta$ in this case.

We now consider the different ways this $\checkdiv$ call could return in Lines~\ref{algline:check-div-coords-return}~to~\ref{algline:check-div-breg-check-justright} of Algorithm~\ref{alg:binary-search}:
\begin{itemize}
    \item If the $\checkdiv$ call returns $\tooBig$ in Line~\ref{algline:check-div-coords-return} or Line~\ref{algline:check-div-breg-check-toobig}, then the binary search in Lines \ref{algline:b-search-new-midpoint} through \ref{algline:b-search-return-finished-b-search} only ever searches through intervals for which $\alpha$ is a lower bound for the remainder of Algorithm~\ref{alg:binary-search}. The final value of $\alpha_m$ which gets returned eventually in Line~\ref{algline:b-search-return-finished-b-search} must satisfy $\alpha_m \ge \alpha$, implying $\alpha_k \ge \alpha$, and therefore combining this with $\breg{z^{k - 1}}{\zground} \le 3 \alpha^2$ and $\breg{z^{k - 1}}{z^k} \ge \alpha_k^2$ from above yields the claim.

    \item Now suppose $\checkdiv$ returns $\tooSmall$ in Line~\ref{algline:check-div-breg-check-toosmall}. By the final claim in Lemma~\ref{lem:checkdiv-correctness}, we have $\zopt_\alpha \defeq \prox_{z_c}^\alpha (\gm f_A; \zset_\nu) =  \prox_{z^{k - 1}}^\alpha (\gm f_A; \zset_\nu) \in \sball_{C, \ztilde, \nu}$ (recall $z_c = z^{k - 1}$), in which case the choice of $\zground$ in Line~\ref{algline:check-div-zground-tau-when-passmodel-true} of Algorithm~\ref{alg:binary-search} implies $\breg{z^{k - 1}}{\zground} \le \breg{z^{k - 1}}{\zopt_\alpha}$. Next, note that because this $\checkdiv$ call returns $\tooSmall$, we have that $\alpha$ is an upper bound on the intervals which the binary search in Lines \ref{algline:b-search-new-midpoint} through \ref{algline:b-search-return-finished-b-search} considers for the remainder of Algorithm~\ref{alg:binary-search}, and therefore $\alpha_k \le \alpha$. Letting $\zopt_{\alpha_k} \defeq \prox_{z^{k - 1}}^{\alpha_k} (\gm f_A; \zset_\nu)$, the fact that $\alpha_k \le \alpha$ and Lemma~\ref{lem:monotonicity} imply $\breg{z^{k - 1}}{\zopt_\alpha} \le \breg{z^{k - 1}}{\zopt_{\alpha_k}}$. Then Lemma~\ref{lem:final-algo-prox-point-in-mult-ball} and the second property in Lemma~\ref{lem:GAPW-properties} (recall also that $\gammagb \gets \alpha_k^2 / 10$ in Line~\ref{algline:ell12ell1-final-GAPW-call} of Algorithm~\ref{alg:l2l1-l1l1-final-algo}) imply $\inabs{\breg{z^{k - 1}}{\zopt_{\alpha_k}} - \breg{z^{k - 1}}{z^k}} \le \alpha_k^2 / 10$, which together with $\breg{z^{k - 1}}{z^k} \ge \alpha_k^2$ from above implies $\breg{z^{k - 1}}{z^k} \ge \frac{1}{2} \breg{z^{k - 1}}{\zopt_{\alpha_k}}$. Combining the latter with $\breg{z^{k - 1}}{\zground} \le \breg{z^{k - 1}}{\zopt_\alpha}$ and $\breg{z^{k - 1}}{\zopt_\alpha} \le \breg{z^{k - 1}}{\zopt_{\alpha_k}}$ from above yields the claim.

    \item Now suppose $\checkdiv$ returns $\justRight$ in Line~\ref{algline:check-div-breg-check-justright}. Then, Algorithm~\ref{alg:binary-search} returns in Line~\ref{algline:b-search-return-finished-b-search}, and thus $\alpha = \alpha_k$. Then the claim holds due to $\breg{z^{k - 1}}{\zground} \le 3 \alpha^2$ and $\breg{z^{k - 1}}{z^k} \ge \alpha_k^2$ from above.
\end{itemize}

\textit{$\checkdiv$ Case 2:} We now consider the case where the $\checkdiv$ subroutine call originates from Line~\ref{algline:b-search-check-div-beta-call} of Algorithm~\ref{alg:binary-search}, meaning the value of $\alpha$ passed to this $\checkdiv$ call is $\beta$. First, consider the case where this $\checkdiv$ call returns either $\tooSmall$ in Line~\ref{algline:check-div-breg-check-toosmall} or $\justRight$ in Line~\ref{algline:check-div-breg-check-justright}; in either case, Algorithm~\ref{alg:binary-search} immediately returns in Line~\ref{algline:b-search-beta-return} and $\alpha_k = \beta$. Therefore, the claim follows from the choice of $\zground$ in Line~\ref{algline:check-div-zground-tau-when-passmodel-true} since $z^k \in \sball_{C, \ztilde, \nu}$. (Indeed, note that $C = 500$ in both Algorithm~\ref{alg:binary-search} and \ref{alg:l2l1-l1l1-final-algo}, and also $\ztilde$ in Line~\ref{algline:ell12ell1-final-best-response} of Algorithm~\ref{alg:l2l1-l1l1-final-algo} during iteration $k$ gets set to the same value as in Line~\ref{algline:checkdiv-alpha-best-response} of this $\checkdiv$ call since $\alpha_k = \beta$ and $z_c = z^{k - 1}$.) Then, consider the case where this $\checkdiv$ call returns $\tooBig$ in either Line~\ref{algline:check-div-coords-return} or Line~\ref{algline:check-div-breg-check-toobig} of Algorithm~\ref{alg:binary-search}. In this case, the binary search starting in Line~\ref{algline:b-search-new-midpoint} begins, and the claim follows from the same reasoning as in the ``$\checkdiv$ returns $\tooBig$'' case under Subclaim 1 above.
\end{proof}

 Toward proving the movement bound between consecutive local norm points given in Lemma~\ref{lem:movement-bound-local-norm-points}, we first give a technical lemma in Corollary~\ref{cor:breg-generalized-triangle}. Corollary~\ref{cor:breg-generalized-triangle} follows from a generalized triangle inequality for the KL divergence over the truncated simplex due to \cite{carmon2024whole}, which we restate in Lemma~\ref{lem:KL-generalized-triangle} for convenience.

\begin{lemma}[Example 5.2 in \cite{carmon2024whole}]
    \label{lem:KL-generalized-triangle}
For any $\nu \in (0, 1/4]$ and $q, u, w \in \simplex_{\nu}^d$, we have
\begin{align*}
    \inKL{q}{u} + \inKL{u}{q} \le 6 \log (\nu^{-1}) \cdot \inparen*{\min \inbraces*{\inKL{u}{w}, \inKL{w}{u}} + \min \inbraces*{\inKL{q}{w}, \inKL{w}{q}}}\,.
\end{align*}
\end{lemma}

\begin{corollary}
    \label{cor:breg-generalized-triangle}
For $q, u, w \in \ztrunc$, we have for some absolute constant $G > 0$:
\begin{align*}
    \breg{q}{u} + \breg{u}{q} \le G(1 + \log \nu^{-1}) \cdot \inparen*{
        \min \inbraces*{\breg{u}{w} + \breg{w}{u}} + \min \inbraces*{\breg{q}{w} + \breg{w}{q}}
    }.
\end{align*}
\end{corollary}

\begin{proof}
This is a straightforward consequence of Lemma~\ref{lem:KL-generalized-triangle} as well as Example 5.1 in \cite{carmon2024whole}.
\end{proof}

Finally, we state and prove Lemma~\ref{lem:movement-bound-local-norm-points}:

\begin{lemma}
    \label{lem:movement-bound-local-norm-points}
Suppose $\nu^{-1}$ and $\beta^{-1}$ are bounded above by a polynomial in $m$, $n$, and $1 / \epsilon$, and for inputs $\epsilon > 0$, $\beta$, judge function $\judge$, and $\passModel = \true$ to Algorithm~\ref{alg:l2l1-l1l1-final-algo}, let $\zground^1, \zground^2, \dots$ denote the sequence of local norm points passed to $\SUGStronglyMonotoneMirrorProx$ (Algorithm~\ref{alg:mirror-prox-sug-strongly-monotone}) during Algorithm~\ref{alg:l2l1-l1l1-final-algo}. That is, $\zground^j$ is the second argument passed to $\SUGStronglyMonotoneMirrorProx$ when it is called for the $j$-th time in Algorithm~\ref{alg:l2l1-l1l1-final-algo}. Then
\begin{align*}
 \sum_{j \ge 1} \innorm{\zground^{j} - \zground^{j + 1}} \le \Otilde \Big( \sqrt{ \beta / \epsilon + \epsilon^{-2/3} }\Big).
\end{align*}
\end{lemma}

\begin{proof}
We first note that it suffices to show
\begin{align}
    \label{eq:div-sum-goal}
    \sum_{j \ge 1} \breg{\zground^{j}}{\zground^{j + 1}} \le \Otilde(1)
\end{align}
by Cauchy-Schwarz, since the number of times $\SUGStronglyMonotoneMirrorProx$ is called during the run of Algorithm~\ref{alg:l2l1-l1l1-final-algo} is at most $\tilde{O}(\beta /\epsilon + \epsilon^{-2/3})$ by Corollary~\ref{cor:l2l1-l1l1-final-algo-correctness-iter-bound} (bounding the number of times $\bsearch$ and the $\GWF$ wrapper function are called in Algorithm~\ref{alg:l2l1-l1l1-final-algo}) and Lemma~\ref{lem:b-search-correctness} (bounding the number of times the $\CWF$ wrapper function is called in Algorithm~\ref{alg:binary-search}).

Toward proving \eqref{eq:div-sum-goal}, note that if $\zground^j$ and $\zground^{j + 1}$ are both local norm points passed to Algorithm~\ref{alg:mirror-prox-sug-strongly-monotone} during iteration $k$ of Algorithm~\ref{alg:l2l1-l1l1-final-algo}, we have $\smash{\breg{\zground^j}{\zground^{j + 1}} \le \Otilde(\breg{z^{k - 1}}{z^k})}$ by Corollary~\ref{cor:breg-generalized-triangle} and Lemma~\ref{lem:movement-bound-single-divergence-bound}. Similarly, if $\zground^j$ and $\zground^{j + 1}$ are local norm points passed to Algorithm~\ref{alg:mirror-prox-sug-strongly-monotone} in iterations $k$ and $k + 1$ of Algorithm~\ref{alg:l2l1-l1l1-final-algo} respectively (i.e., $\zground^j$ is the last local norm point of iteration $k$ and $\zground^{j + 1}$ is the first local norm point of iteration $k + 1$), we have $\breg{\zground^j}{\zground^{j + 1}} \le \Otilde \inparen{\breg{z^{k - 1}}{z^k} + \breg{z^{k}}{z^{k + 1}}}$ by Corollary~\ref{cor:breg-generalized-triangle} and Lemma~\ref{lem:movement-bound-single-divergence-bound}. Note that Algorithm~\ref{alg:mirror-prox-sug-strongly-monotone} is called at most $\Otilde(1)$ times per iteration of Algorithm~\ref{alg:l2l1-l1l1-final-algo} due to Lemma~\ref{lem:b-search-correctness}, and we conclude since $\sum_{k \in [K]} \breg{z^{k - 1}}{z^k} = \Otilde(1)$. Indeed, note $\sum_{k \in [K - 1]} \breg{z^{k - 1}}{z^k} = \Otilde(1)$ by Lemmas~\ref{lem:prox-point-correctness}~and~\ref{lem:connecting-GAPW-to-DAPO}. Additionally, we claim $\breg{z^{K - 1}}{z^K} = O(1)$. To prove the latter, let $\zopt_{\alpha_K} \defeq \prox^{\alpha_K}_{z^{K - 1}}(\gm f_A; \zset_\nu)$, and note $\breg{z^{K - 1}}{\zopt_{\alpha_K}} = O(1)$ by \Cref{lem:b-search-correctness}. Then $\breg{z^{K - 1}}{z^K} = O(1)$ by \Cref{algline:ell12ell1-final-GAPW-call} of \Cref{alg:l2l1-l1l1-final-algo}, \Cref{lem:GAPW-properties}, and \Cref{lem:final-algo-prox-point-in-mult-ball}.
\end{proof}

\subsubsection{Putting it all together}\label{subsubsec:final-l2-l1}

In this section, we ultimately prove the $\Otilde(\epsilon^{-7/9})$ guarantee for $\ellTwoEllOne$ games in Theorem~\ref{intro:l2l1}. But first, in Lemma~\ref{lem:model-update-steps-bound} we use the potential argument described at the beginning of Section~\ref{sec:four-fifths} along with the ingredients of the previous sections to bound the total number of model-update steps made over all calls to Algorithm~\ref{alg:mirror-prox-sug-strongly-monotone} from within Algorithm~\ref{alg:l2l1-l1l1-final-algo} when $\passModel = \true$.

\begin{lemma}
    \label{lem:model-update-steps-bound}
Consider Algorithm~\ref{alg:l2l1-l1l1-final-algo} with inputs $\epsilon, \tau > 0$, $\beta \in (0, 3)$, $\passModel = \true$, and $\judge = \judge_{\project}$, and suppose $\nu^{-1}$ and $\beta^{-1}$ are each bounded above by a polynomial in $m$, $n$, and $1 / \epsilon$. Suppose $\SUGStronglyMonotoneMirrorProx$ (Algorithm~\ref{alg:mirror-prox-sug-strongly-monotone}) is called $J$ times during the run of Algorithm~\ref{alg:l2l1-l1l1-final-algo}, and for $j \in [J]$, let $(\cdot, K_j, M_j)$ denote the output of $\SUGStronglyMonotoneMirrorProx$ when it is called for the $j$-th time. Then
\begin{align}
    \label{eq:model-update-steps-bound}
    \sum_{j \in [J]} K_j \le \Otilde \inparen*{\frac{\sqrt{\beta / \epsilon + \epsilon^{-2/3}}}{\tau^2}}.
\end{align}
\end{lemma}

\begin{proof}
For $j \in [J]$, let $\zground^j$ denote the local norm point input (i.e., the second argument) passed to $\SUGStronglyMonotoneMirrorProx$ when it is called for the $j$-th time, and let $M_0 \defeq 0 \in \R^{m \times n}$ as in Line~\ref{algline:ell12ell1-final-z^0-M_0-C-k} of Algorithm~\ref{alg:l2l1-l1l1-final-algo}. Recalling the definition of $M_j$ for $j \in [J]$ in the lemma statement, note that for all $j \in [J] \cup \inbraces{0}$, it is also the case that $M_j$ is the model passed to $\SUGStronglyMonotoneMirrorProx$ when the latter is called for the $(j + 1)$-th time. Then by Lemma~\ref{lemma:main-inner-loop-convergence-four-fifths},
\begin{align}
    \label{eq:all-together-tau-decrease}
    \normInline{\ground{A - M_{j + 1}}{\zground^{j + 1}}}_{F}^2 \leq \normInline{\ground{A - M_{j}}{\zground^{j + 1}}}_{F}^2 - c_2 K_{j + 1} \tau^2, ~~~\text{for all $j \in [J - 1] \cup \inbraces{0}$.}
\end{align}
for some absolute constant $c_2 > 0$. Furthermore, note that for all $j \in [J - 1]$, we have
\begin{align*}
    \innorm{\ground{A - M_j}{\zground^{j + 1}}}_F^2 &\overle{(i)} \innorm{\ground{A - M_j}{\zground^{j}}}_F^2 + \innorm{\zground^{j + 1} - \zground^j} \norm{A - M_j}^2_{2 \to \infty} \\
    &\overle{(ii)} \innorm{\ground{A - M_j}{\zground^{j}}}_F^2 + \innorm{\zground^{j + 1} - \zground^j},
\end{align*}
where $(i)$ follows from Lemma~\ref{lemma:model-bound}, and $(ii)$ follows because $\norm{A - M_j}^2_{2 \to \infty} \le \norm{A}_{2 \to \infty} \le 1$. Indeed, the latter is due to Lemma~\ref{lemma:projection} and the fact that the fourth item in Lemma~\ref{lemma:main-inner-loop-convergence-four-fifths} (recall also $M_0 = 0$) implies we can express $A - M_j = A P^{(j)}$, for a matrix $P^{(j)} \in \R^{n \times n}$ which is a product of $n \times n$ orthogonal projection matrices. Combining the preceding displays, we obtain 
\begin{align}
    \label{eq:all-together-one-step}
    \normInline{\ground{A - M_{j + 1}}{\zground^{j + 1}}}_{F}^2 \le \innorm{\ground{A - M_j}{\zground^{j}}}_F^2 + \innorm{\zground^{j + 1} - \zground^j} - c_2 K_{j + 1} \tau^2, ~~~\text{for all $j \in [J - 1]$}.
\end{align}
Noting $\innorm{(A - M_0)_{\zground^1}}_F^2 = \innorm{(A)_{\zground^1}}_F^2 \le 1$ due to Remark~\ref{remark:general-conversion}, unraveling \eqref{eq:all-together-one-step} and then applying \eqref{eq:all-together-tau-decrease} one more time with $j \gets 0$ to lower bound $\innorm{\ground{A - M_1}{\zground^{1}}}_F^2$ gives
\begin{align*}
    0 \le \innorm{\ground{A - M_{J}}{\zground^{J}}}_F^2 \le 1 + \sum_{j \in [J - 1]} \innorm{\zground^{j + 1} - \zground^j} - c_2 \tau^2 \sum_{j \in [J]} K_j.
\end{align*}
Then combining this with Lemma~\ref{lem:movement-bound-local-norm-points} gives
\begin{align*}
    \sum_{j \in [J]} K_j \le O \inparen*{\frac{\sum_{j \in [J - 1]} \innorm{\zground^{j + 1} - \zground^j}}{\tau^2}} = \Otilde \inparen*{\frac{\sqrt{\beta / \epsilon + \epsilon^{-2/3}}}{\tau^2}}.
\end{align*}
\end{proof}

Next, we combine Lemma~\ref{lem:model-update-steps-bound} with the iteration bound on Algorithm~\ref{alg:l2l1-l1l1-final-algo} due to Corollary~\ref{cor:l2l1-l1l1-final-algo-correctness-iter-bound} as well as the bound on the number of progress iterations in each call to Algorithm~\ref{alg:mirror-prox-sug-strongly-monotone} due to Theorem~\ref{thm:main-inner-loop-complexity} to bound the overall complexity when $\passModel = \true$.

\begin{lemma}
    \label{lem:general-tau-beta-model-passing-matvec-query-bound}
With the same inputs and assumptions of Lemma~\ref{lem:model-update-steps-bound}, Algorithm~\ref{alg:l2l1-l1l1-final-algo} makes at most 
\begin{align}
    \label{eq:final-general-iteration-bound-before-7/9}
    \Otilde \inparen*{
        (\beta / \epsilon + \epsilon^{-2/3} + 1) (\tau / \beta + 1) +  \frac{\sqrt{\beta / \epsilon + \epsilon^{-2/3}}}{\tau^2}
    } \text{ matrix-vector queries}.
\end{align}
\end{lemma}

\begin{proof}
We first note that the total number of matrix-vector queries other than those which are made internally by $\SUGStronglyMonotoneMirrorProx$ (namely Algorithm~\ref{alg:mirror-prox-sug-strongly-monotone}; recall $\SUGStronglyMonotoneMirrorProx$ is called in Line~\ref{algline:ell12ell1-final-GAPW-call} of Algorithm~\ref{alg:l2l1-l1l1-final-algo} and Line~\ref{algline:checkdiv-CAPW-call} of Algorithm~\ref{alg:binary-search}, the latter due to the call to Algorithm~\ref{alg:binary-search} in Line~\ref{algline:ell12ell1-final-bsearch-call} of Algorithm~\ref{algline:ell12ell1-final-bsearch-call}) is at most $\Otilde(\beta / \epsilon + \epsilon^{-2/3})$. Indeed, this follows because the iteration count of Algorithm~\ref{alg:l2l1-l1l1-final-algo} is bounded by $\Otilde(\beta / \epsilon + \epsilon^{-2/3})$ due to Corollary~\ref{cor:l2l1-l1l1-final-algo-correctness-iter-bound}, and every iteration makes $\Otilde(1)$ matrix-vector queries other than those due to $\SUGStronglyMonotoneMirrorProx$. (Recall that the repeat loop in Algorithm~\ref{alg:binary-search} executes at most $\Otilde(1)$ times due to Lemma~\ref{lem:b-search-correctness}.)

Thus, it remains to bound matrix-vector queries made internally by $\SUGStronglyMonotoneMirrorProx$ during the run of Algorithm~\ref{alg:l2l1-l1l1-final-algo}. Toward this goal, suppose $\SUGStronglyMonotoneMirrorProx$ is called $J$ times during the run of Algorithm~\ref{alg:l2l1-l1l1-final-algo}, and for $j \in [J]$, let $(\cdot, K_j, \cdot)$ denote the output of $\SUGStronglyMonotoneMirrorProx$ when it is called for the $j$-th time. Note then that due to the logic of Algorithm~\ref{alg:binary-search} (namely, $\beta$ is always a lower bound on the intervals it searches through), we have that whenever $\SUGStronglyMonotoneMirrorProx$ is called during Algorithm~\ref{alg:l2l1-l1l1-final-algo}, the fourth argument $\alpha$ (i.e., the regularization level) satisfies $\alpha \ge \beta$. As a result, Theorem~\ref{thm:main-inner-loop-complexity} gives that when $\SUGStronglyMonotoneMirrorProx$ is called for the $j$-th time during Algorithm~\ref{alg:l2l1-l1l1-final-algo}, at most $\Otilde(K_j + \tau / \beta + 1)$ matrix-vector queries are made. Corollary~\ref{cor:l2l1-l1l1-final-algo-correctness-iter-bound} and Lemma~\ref{lem:b-search-correctness} imply $J \le \Otilde(\beta / \epsilon + \epsilon^{-2/3} + 1)$, and therefore the total number of matrix-vector queries made internally by $\SUGStronglyMonotoneMirrorProx$ during the run of Algorithm~\ref{alg:l2l1-l1l1-final-algo} is at most
\begin{align*}
    \Otilde \inparen*{
        (\beta / \epsilon + \epsilon^{-2/3} + 1) (\tau / \beta + 1) + \sum_{j \in [J]} K_j.
    }
\end{align*}
We conclude by applying Lemma~\ref{lem:model-update-steps-bound}.
\end{proof}

To analyze \eqref{eq:final-general-iteration-bound-before-7/9}, there are two natural extremal parameter choices to consider. The first follows from noting that due to the $+1$ in the $(\tau / \beta + 1)$ term in \eqref{eq:final-general-iteration-bound-before-7/9}, it wouldn't help to set $\tau \ll \beta$. However, setting $\tau \gets \beta$ and optimizing the resulting bound \eqref{eq:final-general-iteration-bound-before-7/9} (up to multiplicative polylogarithmic factors) yields $\tau = \beta = \epsilon^{1/5}$, in which case the resulting matrix-vector query complexity is $\Otilde(\epsilon^{-4/5})$. Note that this choice of parameters corresponds to doing only $\Otilde(1)$ progress iterations each time $\SUGStronglyMonotoneMirrorProx$ (Algorithm~\ref{alg:mirror-prox-sug-strongly-monotone}) is called.

The second extremal choice follows by noting that  there is no benefit from choosing $\beta \ll \epsilon^{1/3}$ due to the $+ \epsilon^{-2/3}$ terms. However, setting $\beta \gets \epsilon^{1/3}$ and optimizing \eqref{eq:final-general-iteration-bound-before-7/9} (up to multiplicative polylogarithmic factors) yields $\tau \gets \epsilon^{2/9}$ for an improved overall query complexity of $\Otilde(\epsilon^{-7/9})$. Note that this choice of parameters corresponds to potentially running more than $\Otilde(1)$ progress iterations each time $\SUGStronglyMonotoneMirrorProx$ (Algorithm~\ref{alg:mirror-prox-sug-strongly-monotone}) is called, with the benefit of less total model-update steps. Furthermore, the iteration count of Algorithm~\ref{alg:l2l1-l1l1-final-algo} is minimized (up to multiplicative polylog factors) since it is equal to $\Otilde(\epsilon^{-2/3})$; see Corollary~\ref{cor:l2l1-l1l1-final-algo-correctness-iter-bound}. Indeed, $\Otilde(\epsilon^{-7/9})$ is the best query complexity achievable from \eqref{eq:final-general-iteration-bound-before-7/9}.

\svm*
\begin{proof} 
    The result follows from choosing $\nu$ as in Lemma~\ref{lem:truncation-for-ell2ell1-ell1ell1} and running Algorithm~\ref{alg:l2l1-l1l1-final-algo} with inputs as in Lemma~\ref{lem:model-update-steps-bound} along with $\beta \gets \epsilon^{1/3}$ and $\tau \gets \epsilon^{2/9}$. The correctness follows from Corollary~\ref{cor:l2l1-l1l1-final-algo-correctness-iter-bound} and Lemma~\ref{lem:truncation-for-ell2ell1-ell1ell1} and the complexity bound follows from Lemma~\ref{lem:general-tau-beta-model-passing-matvec-query-bound}.
\end{proof}

\section*{Acknowledgements}

We thank anonymous reviewers for helpful feedback. We thank Guy Kornowski and Ohad Shamir for helpful discussions and their enlightening and motivating work on lower bounds for matrix games. Ishani Karmarkar was funded in part by NSF CAREER Award CCF-1844855, NSF Grant CCF-1955039, and a PayPal research award.
Liam O'Carroll was funded in part by NSF Grant CCF-1955039.
Aaron Sidford was funded in part by a Microsoft Research Faculty Fellowship, NSF CAREER Award CCF-1844855, NSF Grant CCF1955039, and a PayPal research award.

\newpage

\bibliographystyle{plainnat}

\newpage

\addtocontents{toc}{\protect\setcounter{tocdepth}{0}} %
\appendix

\section{Technical lemmas on regret}\label{apx:regret-lemmas}

In this section, we provide two technical lemmas regarding regret for completeness. First, we prove Lemma~\ref{lemma:regret-bounds-error} which bounds the $\gap(\cdot)$ function in terms of regret.

\regretbounderror*
\begin{proof} Note that by convexity and concavity, for all $t \in [T]$ and $u \in \cZ$ we have
\begin{align*}
    f(u\x, w\y^t) &\geq f(w^t) + \nabla\x f(w^t)^\top (u\x - w\x^t), \text{ and } \\
    f(w^t\x, u\y) &\leq f(w^t) + \nabla\y f(w^t)^\top (u\y - w\y^t).
\end{align*}
Combining yields
\begin{align*}
    \nabla_\pm f(w^t) (w^t - u) \geq f(w\x^t, u\y) - f(u\x, w^t\y). 
\end{align*}
Applying Jensen's inequality twice (using convexity of $f$ in $x$ and concavity of $f$ in $y$), yields that for all $u \in \zset$:
\begin{align*}
    \frac{1}{\Lambda} \sum_{t \in [T]} \rho_t \nabla_\pm f(w^t)^\top (w^t - u) \geq \frac{1}{\Lambda} \sum_{t \in [T]} \rho_t [f(w\x^t, u\y) - f(u\x, w^t\y)] \geq f(\bar{w}\x, u\y) - f(u\x, \bar{w}\y). 
\end{align*}
The lemma follows as $\gap(\bar{w}) = \sup_{u \in \cZ} f(\bar{w}\x, u\y) - f(u\x, \bar{w}\y)$.
\end{proof}

Next, we show in Proposition~\ref{prop:regret-wrt-monotone-operator-nonnegative} that regret with respect to a monotone operator is always nonnegative.

\begin{proposition}
    \label{prop:regret-wrt-monotone-operator-nonnegative}
Let $(\zset, \innorm{\cdot}, r)$ denote a dgf setup per Definition~\ref{def:dgf-setup} and $g : \zset \to \R^d$ a monotone operator. Then for any $z^1, \dots, z^K \in \zset$ and $\rho_1, \dots, \rho_K > 0$ with $\Lambda \defeq \sum_{k \in [K]} \rho_k$, we have
\begin{align}
    \label{eq:apx-regret-nonnegative}
    \sup_{u \in \zset} \inbraces*{\frac{1}{\Lambda} \sum_{k \in [K]} \rho_k g(z^k)^\top (z^k - u) } \ge 0.
\end{align}
\end{proposition}

\begin{proof}
    We will instantiate the left-hand side of \eqref{eq:apx-regret-nonnegative} with $u \gets \zbar \defeq \frac{1}{\Lambda} \sum_{k \in [K]} \rho_k z^k$ and show that the result is nonnegative. Indeed, note
\begin{align*}
    \frac{1}{\Lambda} \sum_{k \in [K]} \rho_k g(z^k)^\top (z^k - \zbar) \overge{(i)} \frac{1}{\Lambda} \sum_{k \in [K]} \rho_k g(\zbar)^\top (z^k - \zbar) 
    = \frac{1}{\Lambda} g(\zbar)^\top \inparen*{ \sum_{k \in [K]} \rho_k z^k - \inparen*{\sum_{k \in [K]} \rho_k} \zbar } = 0,
\end{align*}
where $(i)$ used the fact that $g(z^k)^\top (z^k - \zbar) \ge g(\zbar)^\top (z^k - \zbar)$ for all $k \in [K]$ by the monotonicity of $g$.
\end{proof} %

\section{Omitted proofs from Section~\ref{sec:elltwoelloneandelloneellone}}\label{apx:ommitted-proofs-ell1-ell1-ell2-ell1-final-algo}

In Appendices~\ref{subapp:truncated-simplex-KL-norms}, \ref{subapp:technical-lemmas-prox-mappings}, and \ref{apx:linear-algebra} below, we collect and prove various technical lemmas used in Section~\ref{sec:elltwoelloneandelloneellone}. As in Section~\ref{sec:elltwoelloneandelloneellone}, whenever a lemma is stated in the context of the $\ellOneEllOne$ and $\ellTwoEllOne$ setups (Definition \ref{def:matrix-vector-games-setups}), if it does not explicitly distinguish between them, then it applies to both setups simultaneously.

\subsection{Technical lemmas regarding the truncated simplex, KL divergence, and norms}
\label{subapp:truncated-simplex-KL-norms}

In this section, we collect various technical lemmas regarding the truncated simplex, KL divergence, and norms used in Section~\ref{sec:elltwoelloneandelloneellone}. All of the lemmas in this section with the sole exception of Lemma~\ref{lem:breg-error-to-breg-error-from-base} are stated independently of the context of the $\ellTwoEllOne$ and $\ellOneEllOne$ setups (Definition \ref{def:matrix-vector-games-setups}) for greater generality. Recall that $\simplex^d_\nu \defeq \inbraces{z \in \simplex^d : [z]_i \ge \nu, \, \forall i \in [d]}$ denotes the truncated simplex. Then to start, Lemma~\ref{lem:entropy-truncated-simplex} shows that negative entropy is smooth and Lipschitz over the truncated simplex.

\begin{lemma}
    \label{lem:entropy-truncated-simplex}
    With $\nu \in (0, 1 / d)$, the negative entropy function $\entropy(x) \defeq \sum_{i \in [d]} [x]_i \log [x]_i$ is $(1 + \log \nu^{-1})$-Lipschitz and $\nu^{-1}$-smooth with respect to the $\ell_1$-norm over the truncated simplex $\simplex_\nu^d$.
\end{lemma}

\begin{proof}
We have $\grad e(x) = (\log [x]_1 + 1, \dots, \log [x]_d + 1)$ and $\hess e(x) = \showdiag(1 / [x]_1, \dots, 1/ [x]_d)$, in which case the Lipschitz bound follows from bounding $\norm{\grad e(x)}_\infty$ over $x \in \simplex_\nu^d$, and the smoothness bound follows since for any $x \in \simplex_\nu^d$ and $u \in \R^d$ with $\norm{u}_1 = 1$:
\begin{align*}
    u^\top \grad^2 e(x) u = \sum_{i \in [d]} \frac{[u]_i^2}{[x]_i} \le \frac{1}{\nu}.
\end{align*}
\end{proof}

The next five lemmas, namely Lemmas \ref{lem:KL-by-norm}, \ref{lem:squared-norm-by-norm}, \ref{lem:breg-error-to-pointwise-error}, \ref{lem:breg-error-to-breg-error-from-base}, and \ref{lem:bounding-KL-TO-THE-SAME-POINT}, facilitate various approximations made in Section~\ref{sec:elltwoelloneandelloneellone}.

\begin{lemma}
    \label{lem:KL-by-norm}
    Let $q, u, w \in \simplex^d_\nu$ for some $\nu \in (0, 1 / d)$. Then
    \begin{align*}
        \inabs{\inKL{q}{u} - \inKL{q}{w}} \le (1 + 2 \log \nu^{-1}) \norm{u - w}_1.
    \end{align*}
    \end{lemma}
    
    \begin{proof}
        Let $e(x) \defeq \sum_{i \in [d]} [x]_i \log [x]_i$ denote the negative entropy function.
    Then
    \begin{align*}
        \inabs{\inKL{q}{u} - \inKL{q}{w}} &= \inabs*{
            \sum_{i \in [d]} \inparen*{[u]_i \log \frac{[u]_i}{[q]_i} - [w]_i \log \frac{[w]_i}{[q]_i}}
            } \\
              &\le \inabs*{e(u) - e(w)} + \inabs*{\sum_{i \in [d]} ([w]_i - [u]_i) \log [q]_i} \\
              &\le (1 + \log \nu^{-1}) \norm{u - w}_1 + (\log \nu^{-1}) \norm{u - w}_1,
    \end{align*}
    where the last inequality used Lemma~\ref{lem:entropy-truncated-simplex}.
    \end{proof}

\begin{lemma}
    \label{lem:squared-norm-by-norm}
Let $q, u, w \in \ball^d$. Then 
\begin{align*}
    \inabs{\norm{q - u}_2^2 - \norm{q - w}_2^2} \le 4 \norm{u - w}_2
\end{align*}
\end{lemma}

\begin{proof}
We have
\begin{align*}
    \inabs{\norm{q - u}_2^2 - \norm{q - w}_2^2} &= \inabs*{\norm{q - u}_2 - \norm{q - w}_2} \cdot (\norm{q - u}_2 + \norm{q - w}_2) \\
    &\le \norm{u - w}_2 \cdot 4,
\end{align*}
where the final inequality follows from the reverse triangle inequality and that $q, u, w \in \ball^d$ by assumption.
\end{proof}

\begin{lemma}
    \label{lem:breg-error-to-pointwise-error}
    For $\nu \in (0, 1 / d)$, $\gamma > 0$, and $u, v \in \simplex_\nu^d$ such that $\inKL{u}{v} \le \frac{1}{2} \gamma^2 \nu^2$, we have $v \approx_{1 / (1 + \gamma)} u$.
\end{lemma}

\begin{proof}
    Letting $\delta \defeq \inKL{u}{v}$, we have for all $i \in [d]$:
    \begin{align*}
        \inabs*{ [v]_i - [u]_i } \le \innorm{v - u}_1  \le \sqrt{2 \delta},
    \end{align*}
    in which case it suffices to show $[u]_i + \sqrt{2 \delta} \le (1 + \gamma) [u]_i$ and $[u]_i - \sqrt{2 \delta} \ge \frac{1}{1 + \gamma} [u]_i$ for all $i \in [d]$. These are equivalent to $\delta \le \frac{1}{2} \gamma^2 [u]_i^2$ and $\delta \le \frac{1}{2} (\frac{\gamma}{1 + \gamma})^2 [u]_i^2$ respectively, in which case $\delta \le \frac{1}{2} \gamma^2 \nu^2$ suffices, yielding the result.
\end{proof}

Recall that Lemma~\ref{lem:breg-error-to-breg-error-from-base} below is the only lemma in Appendix~\ref{subapp:truncated-simplex-KL-norms} which is stated within the context of the $\ellOneEllOne$ and $\ellTwoEllOne$ setups (Definition~\ref{def:matrix-vector-games-setups}).

\begin{lemma}
    \label{lem:breg-error-to-breg-error-from-base}
For $\gamma > 0$ and $u, v, q \in \ztrunc$, there is an absolute constant $M > 0$ such that if $\breg{u}{v} \le  \frac{M \gamma^2}{1 + \log^2 \nu^{-1}}$, then $\inabs{\breg{q}{u} - \breg{q}{v}} \le \gamma$. 
\end{lemma}

\begin{proof}
For some absolute constants $M_1, M_2 > 0$, we have by Lemmas \ref{lem:KL-by-norm} and \ref{lem:squared-norm-by-norm}:
\begin{align*}
    \inabs{\breg{q}{u} - \breg{q}{v}} \le \inabs{\xbreg{q\x}{u\x} - \xbreg{q\x}{v\x}} + \inabs{\ybreg{q\y}{u\y} - \ybreg{q\y}{v\y}} &\le M_1(1 + \log \nu^{-1}) \cdot \norm{u - v}\\
    &\le M_2(1 + \log \nu^{-1}) \cdot \sqrt{\breg{u}{v}}.
\end{align*}
\end{proof}

\begin{lemma}
    \label{lem:bounding-KL-TO-THE-SAME-POINT}
Let $q, u, w \in \simplex_\nu^d$ for some $\nu \in (0, 1/d)$. Then
\begin{align*}
    \inabs{\inKL{u}{q} - \inKL{w}{q}} \le \frac{1}{\nu} \sqrt{2 \cdot \inKL{u}{w}}.
\end{align*}
\end{lemma}

\begin{proof}
    Suppose $\inKL{u}{w} = \frac{1}{2} \gamma^2 \nu^2$ for $\gamma > 0$, in which case Lemma~\ref{lem:breg-error-to-pointwise-error} gives that $1 / (1 + \gamma) \le [w]_i / [u]_i \le 1 + \gamma$ for all $i \in [d]$. Then
\begin{align*}
    \inabs{\inKL{u}{q} - \inKL{w}{q}} = \inabs*{
        \sum_{i \in [d]} \inparen*{[q]_i \log \frac{[q]_i}{[u]_i} - [q]_i \log \frac{[q]_i}{[w]_i}} 
        }  
    &= \inabs*{\sum_{i \in [d]} [q]_i \log \frac{[w]_i}{[u]_i}} \\
    &\le \sup_{i \in [d]} \inabs*{ \log \frac{[w]_i}{[u]_i} } \\
    &\le  \gamma          ,
\end{align*}
where the last inequality follows because $\inabs{ \log x } \le \gamma$ for all $x \in [1 / (1 + \gamma), 1 + \gamma]$. (E.g., this can be derived from the standard inequality $1 - 1 / x \le \log x \le x - 1$ for $x > 0$.) Then substituting back in $\inKL{u}{w}$ yields the result.
\end{proof}

Finally, we show for completeness that $\innorm{\cdot}_{(s, t)}$ is a norm.

\begin{proposition}
    \label{prop:instance-dependent-is-norm}
$\innorm{\cdot}_{(s, t)}$ is a norm on $\R^{m \times n}$ for any $s, t \in [1, \infty]$.
\end{proposition}

\begin{proof}
    Letting $A \in \R^{m \times n}$, the fact that $\innorm{A}_{(s, t)} = 0$ implies $A = 0$ is trivial. As for absolute homogeneity, note that for $\rho \in \R$:
\begin{align*}
    \innorm{\rho A}_{(s, t)} = \innorm{(\innorm{\rho A_{1, :}}_s, \dots, \innorm{ \rho A_{m, :}}_s)}_t &= \innorm{|\rho| \cdot  (\innorm{ A_{1, :}}_s, \dots, \innorm{ A_{m, :}}_s)}_t \\
    &= |\rho| \cdot \innorm{ (\innorm{ A_{1, :}}_s, \dots, \innorm{ A_{m, :}}_s)}_t \\
    &= |\rho| \cdot \innorm{ A}_{(s, t)}
\end{align*}
by applying absolute homogeneity of the $\ell_s$ and $\ell_t$ norms in succession. Finally, letting $B \in \R^{m \times n}$, we have
\begin{align*}
    \innorm{A + B}_{(s, t)} &= \innorm{(\innorm{A_{1, :} + B_{1, :}}_s, \dots, \innorm{A_{m, :} + B_{m, :}}_s)}_t  \\
    &\overle{(i)} \innorm{(\innorm{A_{1, :}}_s + \innorm{B_{1, :}}_s, \dots, \innorm{A_{m, :}}_s + \innorm{B_{m, :}}_s )}_t \\
    &\overle{(ii)} \innorm{(\innorm{A_{1, :}}_s, \dots, \innorm{A_{m, :}}_s )}_t + \innorm{(\innorm{B_{1, :}}_s, \dots, \innorm{B_{m, :}}_s )}_t \\
    &= \innorm{A}_{(s, t)} + \innorm{B}_{(s, t)},
\end{align*}
where $(i)$ follows because $\innorm{A_{i, :} + B_{i, :}}_s \le \innorm{A_{i, :}}_s + \innorm{B_{i, :}}_s$ for all $i \in [m]$ by the triangle inequality for the $\ell_s$-norm, and also because if $u, v \in \R^n$ and $v_i \ge u_i \ge 0$, then $\innorm{v}_t \ge \innorm{u}_t$. $(ii)$ uses the triangle inequality for the $\ell_t$-norm.
\end{proof}

\subsection{Technical lemmas regarding proximal mappings}
\label{subapp:technical-lemmas-prox-mappings}

In this section, we collect various technical lemmas regarding proximal mappings (Definition~\ref{def:proximal-mappings}) used in Section~\ref{sec:elltwoelloneandelloneellone}. First, in Lemma~\ref{lem:Lipschitz-type-saddle}, which is stated independently of the context of the $\ellTwoEllOne$ and $\ellOneEllOne$ setups for greater generality, we show a Lipschitz-type bound on the distance between proximal mappings.

\begin{lemma}
    \label{lem:Lipschitz-type-saddle}
Let $(\zset, \norm{\cdot}, r)$ denote a dgf setup (Definition~\ref{def:dgf-setup}) where $r$ is  $L_r$-smooth with respect to $\norm{\cdot}$ over $\zset$, and $\norm{z - z'} \le R$ for all $z, z' \in \zset$. For any continuous monotone operator $g : \zset \to \R$, $\alpha, \beta > 0$ and $q \in \zset$,
\begin{align*}
    \norm{\walpha - \wbeta} \le \frac{2 L_r R}{\max \inbraces{\alpha, \beta}} |\alpha - \beta|
    \text{ where }
    \walpha \defeq \prox^\alpha_{q}(g)
    \text{ and }
    \wbeta \defeq \prox^{\beta}_{q}(g)\,.
\end{align*}
\end{lemma}

\begin{proof}
    Without loss of generality suppose that $\alpha \ge \beta$ .
    Recalling the notation $\grad \breg{q}{z} = \grad r(z) - \grad r(q)$ for the gradient of $u \mapsto \breg{q}{u}$ evaluated at $z$, and defining the operators
    \begin{align*}
        \halpha(z) \defeq g(z) + \alpha \grad \breg{q}{z} ~~~\text{and}~~~ \hbeta(z) \defeq g(z) + \beta \grad \breg{q}{z},
    \end{align*}
    we have by the equivalent formulation of the proximal mapping \eqref{eq:equivalent-prox-condition},
    \begin{align*}
        \halpha(\walpha)^\top (\walpha - u) \le 0 ~~~\text{and}~~~  \hbeta(\wbeta)^\top (\wbeta - u') \le 0 ~~~\text{for all $u, u' \in \zset$}.
    \end{align*}
    Selecting $u \gets \wbeta$ and $u' \gets \walpha$ and summing yields $(\halpha(\walpha) - \hbeta(\wbeta))^\top (\walpha - \wbeta) \le 0$, implying
    \begin{align*}
    (\halpha(\walpha) - \halpha(\wbeta))^\top (\walpha - \wbeta) &\le [(\hbeta - \halpha)(\wbeta)]^\top (\walpha - \wbeta) \\
    &\le \dualnorm{(\hbeta - \halpha)(\wbeta)} \norm{\walpha - \wbeta} \\
        &= \dualnorm{ \beta \grad \breg{q}{\wbeta} - \alpha \grad \breg{q}{\wbeta}} \norm{\walpha - \wbeta} \\
        &= |\alpha - \beta| \cdot \dualnorm{\grad r(\wbeta) - \grad r(q)}  \norm{\walpha - \wbeta} \\
        &\le L_r R \cdot |\alpha - \beta|  \cdot \norm{\walpha - \wbeta}
    \end{align*}
    On the other hand, the fact that $\halpha$ is $\alpha$-strongly monotone relative to $r$ 
    along with the fact that $r$ is 1-strongly convex with respect to $\norm{\cdot}$ implies
    \begin{align*}
        (\halpha(\walpha) - \halpha(\wbeta))^\top (\walpha - \wbeta) \ge \alpha \breg{\walpha}{\wbeta} \ge \frac{\alpha}{2} \norm{\walpha - \wbeta}^2.
    \end{align*}
    Combining yields $\norm{\walpha - \wbeta} \le L_r R \cdot |\alpha - \beta| \cdot \frac{2}{\alpha}$, and we conclude using the assumption from the start of the proof that $\alpha \ge \beta$.
\end{proof}

Next, in Lemma~\ref{lem:monotonicity}, which is stated independently of the context of the $\ellTwoEllOne$ and $\ellOneEllOne$ setups for greater generality, we give a monotonicity guarantee with respect to the divergence from the center point.

\begin{lemma}
    \label{lem:monotonicity}
Let $(\zset, \norm{\cdot}, r)$ denote a dgf setup (Definition~\ref{def:dgf-setup}) with $g : \zset \to \R$ a continuous monotone operator, some $\alpha \ge \beta > 0$, and $q \in \zset$. Then $\walpha \defeq \prox^\alpha_{q}(g)$ and $\wbeta \defeq \prox^{\beta}_{q}(g)$ satisfy $\breg{q}{\walpha} \le \breg{q}{\wbeta}$.
\end{lemma}

\begin{proof}
Applying Definition~\ref{def:proximal-mappings}, we have for all $u, u' \in \zset$:
\begin{align*}
    g(\walpha)^\top (\walpha - u) &\le \alpha \insquare{\breg{q}{u} - \breg{\walpha}{u} - \breg{q}{\walpha}}, \\
    g(\wbeta)^\top (\wbeta - u') &\le \beta \insquare{\breg{q}{u'} - \breg{\wbeta}{u'} - \breg{q}{\wbeta}}.
\end{align*}
Setting $u \gets \wbeta$, $u' \gets \walpha$, summing, and using the monotonicity of $g$ yields
\begin{align*}
    0 &\le \alpha \insquare{\breg{q}{\wbeta} - \breg{\walpha}{\wbeta} - \breg{q}{\walpha}} + \beta \insquare{\breg{q}{\walpha} - \breg{\wbeta}{\walpha} - \breg{q}{\wbeta}} \\
    &= (\alpha - \beta) \breg{q}{\wbeta} + (\beta - \alpha) \breg{q}{\walpha} - \alpha \breg{\walpha}{\wbeta} - \beta \breg{\wbeta}{\walpha}.
\end{align*}
To conclude, we have
\begin{align*}
    (\alpha - \beta) \breg{q}{\walpha} \le (\alpha - \beta) \breg{q}{\wbeta} - \alpha \breg{\walpha}{\wbeta} - \beta \breg{\wbeta}{\walpha} \le (\alpha - \beta) \breg{q}{\wbeta}.
\end{align*}
\end{proof}

Below in Lemma~\ref{lem:Bregman-are-Lipschitz}, which is stated within the context of the $\ellOneEllOne$ and $\ellTwoEllOne$ setups (Definition \ref{def:matrix-vector-games-setups}), we obtain a Lipschitz-type condition on the divergence from the same center point to two proximal mappings with different levels of regularization. 

\begin{lemma}
    \label{lem:Bregman-are-Lipschitz}
For $\alpha, \beta > 0$ and $q \in \ztrunc$, let $w \defeq \prox_q^{\alpha}(\gm f_A; \ztrunc)$ and $w' \defeq \prox_q^{\beta}(\gm f_A; \zset_\nu)$. Then there exists an absolute constant $M > 0$ such that 
\begin{align*}
    \inabs{\breg{q}{w} - \breg{q}{w'}} \le \frac{M(1 + \log \nu^{-1})}{\nu \cdot \max \inbraces{\alpha, \beta}}  \cdot |\alpha - \beta|  .
\end{align*}
\end{lemma}

\begin{proof}
    For some absolute constants $M_1, M_2 > 0$, we have:
\begin{align*}
    \inabs{\breg{q}{w} - \breg{q}{w'}} &\le \inabs{\xbreg{q\x}{w\x} - \xbreg{q\x}{w'\x}} + \inabs{\ybreg{q\y}{w\y} - \ybreg{q\y}{w'\y}} \\
    &\overle{(i)} M_1(1 + \log \nu^{-1}) \cdot \norm{w - w'} \\
    &\overle{(ii)}    \frac{M_2(1 + \log \nu^{-1})}{\nu \cdot \max \inbraces{\alpha, \beta}}  \cdot |\alpha - \beta|  .                           
\end{align*}
Where $(i)$ uses Lemmas \ref{lem:KL-by-norm} and \ref{lem:squared-norm-by-norm} and $(ii)$ uses Lemmas \ref{lem:Lipschitz-type-saddle} and \ref{lem:entropy-truncated-simplex}:
\end{proof}

Next, we use Lemma~\ref{lem:Bregman-are-Lipschitz} to prove a particular function, which we will search for an approximate root of in our binary search, is Lipschitz over compact intervals. Lemma~\ref{lem:Lipschitzness-of-h} is stated within the context of the $\ellOneEllOne$ and $\ellTwoEllOne$ setups.

\begin{lemma}
    \label{lem:Lipschitzness-of-h}
For a fixed $q \in \ztrunc$ and parameter $\alpha > 0$, let $w_\alpha \defeq \prox_q^{\alpha}(\gm f_A; \ztrunc)$ (namely, $w_\alpha$ is parameterized by $\alpha$) and define $h : \R_{>0} \to \R$ via $h(\alpha) \defeq \breg{q}{w_\alpha} - 2 \alpha^2$. Then for any $\alpha, \alpha' \in [b, c]$ for some $c > b > 0$, there exists an absolute constant $M > 0$ such that
\begin{align*}
    |h(\alpha) - h(\alpha')| \le M \inparen*{\frac{1 + \log \nu^{-1}}{\nu b} + c} \cdot |\alpha - \alpha'|.
\end{align*}
\end{lemma}

\begin{proof}
Note that we can bound
\begin{align*}
    \inabs{\alpha - \alpha'^2} = \inabs{\alpha - \alpha'} \cdot (\alpha + \alpha') \le \inabs{\alpha - \alpha'} \cdot 2c,
\end{align*}
in which case combining this with Lemma~\ref{lem:Bregman-are-Lipschitz} implies for some absolute constant $M > 0$:
\begin{align*}
    |h(\alpha) - h(\alpha')| \le \inabs{\breg{q}{w_\alpha} - \breg{q}{w_{\alpha'}}} + 2 \inabs{\alpha^2 - \alpha'^2} 
    \le M \inparen*{\frac{1 + \log \nu^{-1}}{\nu b} + c} \cdot |\alpha - \alpha'|.
\end{align*}
\end{proof}

The following lemma, stated within the context of the $\ellOneEllOne$ and $\ellTwoEllOne$ setups, will be used to justify a starting interval for our binary search.

\begin{lemma}
    \label{lem:starting-value-b-search}
For any $q \in \zset_\nu$ and $\alpha \ge 1$, we have that $w \defeq \prox_q^{\alpha}(\gm f_A; \zset_\nu)$ satisfies $\breg{q}{w} \le 4$.
\end{lemma}

\begin{proof}
Note that $w$ is the unique exact solution of the minimax problem
\begin{align*}
    \min_{x \in \xset_\nu} \max_{y \in \yset_\nu} f_A(x, y) + \alpha \xbreg{q\x}{x} - \alpha \ybreg{q\y}{y}.
\end{align*}
Furthermore, by the normalizing assumptions on the matrix $A$ (see the start of Section~\ref{sec:elltwoelloneandelloneellone}), note that the range of $f_A$ is at most 1, i.e., $|f_A(z)| \le 1$ for all $z \in \zset_\nu$. Then suppose for the sake of contradiction that $\breg{q}{w} > 4$, in which case either $\breg{q\x}{w\x} > 2$ or $\breg{q\y}{w\y} > 2$. Assuming the former, we have
\begin{align*}
    w\x = \argmin_{x \in \xset_\nu} f_A(x, w\y) + \alpha \xbreg{q\x}{x} - \alpha \ybreg{q\y}{w\y} = 
    \argmin_{x \in \xset_\nu} f_A(x, w\y) + \alpha \xbreg{q\x}{x} 
\end{align*}
since $w\x$ is the best response to $w\y$.
But then 
\begin{align*}
    f_A(q\x, w\y) + \alpha \xbreg{q\x}{q\x} = f_A(q\x, w\y) \le 1 < f_A(w\x, w\y) + \alpha \xbreg{q\x}{w\x}.
\end{align*}
The case where $\breg{q\y}{w\y} > 2$ is analogous.
\end{proof}

\subsection{Local Schatten norm bound}\label{apx:linear-algebra}

Throughout this appendix, we operate in the $\ellTwoEllOne$ and $\ellOneEllOne$ setups of Definition~\ref{def:matrix-vector-games-setups}; if a statement does not explicitly distinguish between them, it applies to both setups. In this appendix, we state and prove a number of useful linear-algebraic facts used throughout the main body of the paper, culminating in a proof of Lemma~\ref{lemma:gen-conversion}, restated below. 

\generalconversion*

As a warmup towards proving Lemma~\ref{lemma:gen-conversion}, we first prove a special case of it, when $p = 2$. 

\begin{lemma}\label{lemma:conversion} Let $\zground \in \cZint$ and $A \in \R^{m \times n}$. Then, 
\begin{align*}
    \normInline{\ground{A}{\zground}}_F \leq \begin{cases}
        \normInline{A}_{2 \to \infty}, & \cX = \B^n, \\
        \normInline{A}_{\max}, & \cY = \Delta^m.
    \end{cases}
\end{align*}
\end{lemma}
\begin{proof} First, suppose $\cX =\B^n$. Then, by Definition~\ref{def:setup-details},
\begin{align*}
    \normInline{\ground{A}{\zground}}_F^2 &= \sum_{i \in [m] j \in [n]} [{\zground}\y]_i A_{ij}^2 = \sum_{i \in [m]} [{\zground}\y]_i \sum_{j \in [n]} A_{ij}^2 \\
    &\leq \normInline{A}_{2\to\infty}^2 \sum_{i \in [m]} [{\zground}\y]_i = \normInline{A}_{2\to\infty}^2, 
\end{align*}
where the last line used ${\zground}\y \in \Delta^m$. Now, suppose $\cX =\Delta^n$. Then, by Definition~\ref{def:setup-details}
\begin{align*}
    \normInline{\ground{A}{\zground}}_F^2 &= \sum_{i \in [m] j \in [n]} [{\zground}\y]_i A_{ij}^2 = \sum_{i \in [m]} [{\zground}\y]_i \sum_{j \in [n]} [{\zground}\x]_j A_{ij}^2 \\
    &\leq (\max_{ij} |A_{ij}|)^2 \sum_{i \in [m]} [{\zground}\y]_i \sum_{j \in [n]} [{\zground}\x]_j = (\max_{ij} |A_{ij}|)^2 = \normInline{A}_{\max},
\end{align*}
where the last line used $\zground \in \Delta^d$. 
\end{proof}

Next, we will generalize Lemma~\ref{lemma:conversion} to other Schatten-$p$ norms. To do so, we need several intermediate technical lemmas. First, we use the following result from \cite{rohde2011estimation} to prove Lemma~\ref{lemma:subadditivity-intro} below. 

\begin{lemma}[Equation (2.1) of \cite{rohde2011estimation}]\label{lemma:subadditivity-sp} For $A, B \in \R^{m \times n}$ and $p \in (0, 1]$, $\normInline{A+B}_{\cS_p}^p \leq \normInline{A}_{\cS_p}^p + \normInline{B}_{\cS_p}^p$. 
\end{lemma}

\begin{lemma}\label{lemma:subadditivity-intro} Let $A, B \in \R^{n \times n}$ be positive definite matrices and $p \in (0, 1]$. Then, 
\begin{align*}
    \tr( (A+B)^{p} ) \leq \tr(A^p) + \tr(B^p).
\end{align*}
\end{lemma}
\begin{proof} Note that by the definition of singular values (singular values of $A$ are the square roots of the eigenvalues of $A^\top A$) and the Schatten-$p$ norm, we have 
\begin{align}\label{eq:rewrite}
    \normInline{A+B}_{\cS_p}^p = \tr( ((A+B)^\top(A+B))^{p/2} ) = \tr( (A+B)^{p} ), 
\end{align}
where the last equality uses the fact that $(A+B)$ is PSD in that for any PSD matrix $C$,
\begin{align}\label{eq:psd}
    \tr((C^\top C)^{p/2}) = \tr((C^2)^{p/2}) = \tr(C^p). 
\end{align}
Now, by Lemma~\ref{lemma:subadditivity-sp} and \eqref{eq:rewrite}, we have that 
\begin{align*}
    \tr( (A+B)^{p} ) = \normInline{A+B}_{\cS_p}^p \leq \normInline{A}_{\cS_p}^p + \normInline{B}_{\cS_p}^p. 
\end{align*}
Expanding right hand side of the inequality using \eqref{eq:psd}, we have that
\begin{align*}
    \tr( (A+B)^{p} ) \leq \normInline{A}_{\cS_p}^p + \normInline{B}_{\cS_p}^p = \tr((A^\top A)^{p/2}) + \tr((B^\top B)^{p/2}) = \tr(A^p) + \tr(B^p), 
\end{align*}
where the last equality used the fact that $A$ and $B$ are PSD. Thus, 
\begin{align*}
    \tr( (A+B)^{p} ) \leq \tr(A^p) + \tr(B^p).
\end{align*}
\end{proof}

Next, we inductively apply Lemma~\ref{lemma:subadditivity-intro} to prove Lemma~\ref{lemma:subadditivity}. 

\begin{lemma}\label{lemma:subadditivity} Let $A_1, A_2, ..., A_m \in \R^{n \times n}$ be PSD and $p \in (0, 1]$. Then, 
\begin{align*}
    \tr\paren{\paren{\sum_{i\in[m]} A_i}^p} \leq \sum_{i \in [m]} \tr(A_i^p). 
\end{align*} 
\end{lemma}
\begin{proof} Induct on $m$. If $m = 1$, this is trivially true. Suppose it is also true up to $m-1$. Then, by Lemma~\ref{lemma:subadditivity-intro}, 
\begin{align*}
    \tr\paren{\paren{\sum_{i\in[m]} A_i}^p} &= \tr\paren{\paren{A_m + \sum_{i\in[m-1]} A_i}^p} \leq \tr(A_m^p) + \tr\paren{\paren{\sum_{i\in[m-1]} A_i}^p} \\
    &\leq \tr(A_m^p) + \sum_{i \in [m-1]} \tr(A_i^p) = \sum_{i \in [m]} \tr(A_i^p). 
\end{align*} 
\end{proof}

Next, recall that for any $s, t \in [1, \infty]$ we have defined $ \Uq{A}{s, t} \defeq \normInline{a}_{s}, \text{ where } [a]_i = \normInline{A_{i, :}}_s.$ That is, $\Uq{A}{s, t}$ is the $\ell_t$-norm of the vector $a$, where the $i$-th entry of $A$ is the $\ell_s$-th norm of the $i$-th row of $A$. Using this definition, the following two Lemmas apply Cauchy Schwarz to generalize Lemma~\ref{lemma:conversion}.

\begin{lemma}\label{lemma:l2l1instance} Let $y \in \Delta^m_{>0}$, $p \in [1, 2]$, and $A \in \R^{m \times n}$. Then, $\normInline{\diag(y)^{1/2} A}_{\cS_p}^p \leq \Uq{A}{2, 2p/(2-p)}$.
\end{lemma}
\begin{proof} Note that 
\begin{align*}
    \normInline{\diag(y)^{1/2}A}_{\cS_p}^p = \tr\paren{ (A^\top (y) A)^{p/2} } = \tr\paren{ \paren{ \sum_{i \in [m]} [y]_i A_{i, :} A_{i, :}^\top}^{p/2} }. 
\end{align*} 
where the second equality is due to the outer-product expression for matrix multiplication. Observe that $A_{i, :} A_{i, :}^T$ is a PSD (rank-one) matrix, and $[y]_{i} > 0$. hence, each $[y]_i A_{i, :} A_{i, :}^\top$ is PSD. Further, note that $p \in [1, 2]$ ensures $p/2 \in [1/2, 1] \subset [0, 1]$. Thus, by Lemma~\ref{lemma:subadditivity}, 
\begin{align*}
      \normInline{(y)^{1/2}A}_{\cS_p}^p = \tr\paren{ \paren{ \sum_{i \in [m]} [y]_i A_{i, :} A_{i, :}^\top}^{p/2} } &\leq \sum_{i \in [m]} \tr\paren{ \paren{ [y]_i A_{i, :} A_{i, :}^\top}^{p/2} } = \sum_{i \in [m]} [y]_i^{p/2} \normInline{A_{i,:}}_2^{p}, 
\end{align*}
where the final equality is because a rank-one matrix $uu^\top$ has only one nonzero eigenvalue, equal to $\normInline{u}_2^2$. Continuing on, let $p' = \frac{1}{1-1/(2/p)} = \frac{2}{2-p}$. Now, applying Holder's inequality with the duals $(2/p, p')$ we have
\begin{align*}
       \normInline{\diag(y)^{1/2}A}_{\cS_p}^p &\leq \sum_{i \in [m]} [y]_i^{p/2} \normInline{ A_{i, :}}_{2}^p \leq \paren{\sum_{i \in [m]} \paren{[y]_i^{p/2}}^{2/p}}^{p/2} \cdot \paren{\sum_{i \in [m]} \paren{\normInline{ A_{i, :}}_{2}^p}^{p'}}^{1/p'} \\
       &= \paren{\sum_{i \in [m]} \normInline{ A_{i, :}}_{2}^{2p/(2-p)}}^{(2-p)/2}. 
\end{align*}
where the equality is because $y \in \Delta^n$. Now, the claim follows by taking the $p$-th root of both sides of the above equation. 
\end{proof}

\begin{lemma}\label{lemma:l1l1instance} Let $y \in \Delta^m_{>0}$, $p \in [1, 2]$, and $A \in \R^{m \times n}$. Then,
    \begin{align*}
        \normInline{\diag(y)^{1/2} A \cdot \diag (x)^{1/2}}_{\cS_p}^p \leq \Uq{A}{\infty, 2p/(2-p)} .
    \end{align*}
\end{lemma}
\begin{proof} Note that 
\begin{align*}
    \normInline{\diag(y)^{1/2}A \cdot \diag(x)^{1/2}}_{\cS_p}^p &= \tr\paren{ \paren{\diag(x)^{1/2}A^\top  \diag(y) A \cdot \diag(x)^{1/2}}^{p/2}} \\
    &= \tr\paren{ \paren{ \sum_{i \in [m]} [y]_i \cdot (A \cdot \diag(x)^{1/2})_{i,:} (A \cdot \diag(x)^{1/2})_{i,:}^\top}^{p/2} }. 
\end{align*} 
where the second equality is due to the outer-product expression for matrix multiplication. Observe that $(A \cdot \diag(x)^{1/2})_{i,:} (A \cdot \diag (x)^{1/2})_{i,:}^\top$ is a PSD (rank-one) matrix, and $[y]_{i} > 0$. hence, each term $[y]_i \cdot (A \cdot \diag (x)^{1/2})_{i,:} (A \cdot \diag (x)^{1/2})_{i,:}^\top$ is PSD. Further, note that $p \in [1, 2]$ ensures $p/2 \in [1/2, 1] \subset [0, 1]$. Thus, by Theorem~\ref{lemma:subadditivity}, 
\begin{align*}
      \normInline{\diag(y)^{1/2}A \cdot \diag (x)^{1/2}}_{\cS_p}^p &= \tr\paren{ \paren{ \sum_{i \in [m]} [y]_i \cdot (A \cdot \diag (x)^{1/2})_{i,:} (A \cdot \diag (x)^{1/2})_{i,:}^\top}^{p/2} } \\
      &\leq \sum_{i \in [m]} \tr\paren{ \paren{ [y]_i \cdot (A \cdot \diag(x)^{1/2})_{i,:} (A \cdot \diag(x)^{1/2})_{i,:}^\top}^{p/2} } \\
      &= \sum_{i \in [m]} [y]_i^{p/2} \normInline{ (A \cdot \diag(x)^{1/2})_{i,:} }_2^{p},
\end{align*}
where the final equality is because a rank-one matrix $uu^\top$ has only one nonzero eigenvalue, equal to $\normInline{u}_2^2$. Continuing on, let $p' = \frac{1}{1-1/(2/p)} = \frac{2}{2-p}$. Now, applying Holder's inequality with the duals $(2/p, p')$ we have
\begin{align*}
       \normInline{\diag(y)^{1/2}A \cdot \diag(x)^{1/2}}_{\cS_p}^p &\leq \sum_{i \in [m]} [y]_i^{p/2} \normInline{ (A \cdot \diag(x)^{1/2})_{i,:} }_2^{p} \\
       &=\sum_{i \in [m]} [y]_i^{p/2} \paren{\sum_{j \in [n]} [x]_j A_{ij}^2}^{p/2}  \\
       &\leq \sum_{i \in [m]} [y]_i^{p/2} \paren{\max_{j \in [n]} A_{ij}^2}^{p/2} \\
       &= \sum_{i \in [m]} [y]_i^{p/2} \max_{j \in [n]} |A_{ij}|^p \\
       &\leq \paren{\sum_{i \in [m]} \paren{[y]_i^{p/2}}^{2/p}}^{p/2} \cdot \paren{\sum_{i \in [m]} \paren{\max_{j \in [n]} |A_{ij}|}^{p p'}}^{1/p'} \\
       &= \paren{ \sum_{i \in [m]} \paren{\max_{j \in [n]} |A_{ij}|}^{2p/(2-p)} }^{(2-p)/2}. 
\end{align*}
where the second inequality is because $x \in \Delta^n$, the third inequality is Holder's inequality, and the last equality is because $y \in \Delta^n$. The claim follows by taking the $p$-th root of both sides of the above equation. 
\end{proof}

Finally, we obtain our proof of Lemma~\ref{lemma:gen-conversion}. 

\begin{proof}[Proof of Lemma~\ref{lemma:gen-conversion}] The proof is immediate from Lemma~\ref{lemma:l2l1instance} and Lemma~\ref{lemma:l1l1instance}, and Definition~\ref{def:setup-details}. 
\end{proof}
Note that when $p = 2$, $2p/(2-p) = \infty$ and Lemma~\ref{lemma:gen-conversion} reduces to Lemma~\ref{lemma:conversion}.  %
\section{Alternative implementation of a 2-smooth-guilty judge}\label{apx:other-judge}

We remark that although $\judge_\cS$ (Algorithm~\ref{alg:judge-l2l2}) is simultaneously a $p$-smooth-guilty judge for all $p \in [1, \infty)$, at least for certain values of $p$ (e.g., $p = 2$) alternative implementations of a $p$-smooth guilty judge may be possible. In the special case of $p = 2$ (the Frobenius norm), there is an alternative, simpler implementation of a $2$-smooth-guilty judge, as described in the following Algorithm~\ref{alg:frob-judge-l2l2}. The pseudocode is almost identical to that of Algorithm~\ref{alg:judge-l2l2}; however, instead of making a two-sided projection as in Line~\ref{line:two-sided}, Algorithm~\ref{alg:frob-judge-l2l2} requires just a rank-one update in Line~\ref{line:rank-one-update-apx}. We formalize this below. 

\RestyleAlgo{ruled}
\SetKwComment{Comment}{/* }{ */}
\begin{algorithm2e}[ht]
\caption{$\judge_F(z, z', B, \tau)$}
\label{alg:frob-judge-l2l2}
\KwInput{Vectors $z, z' \in \R^d$, smoothness threshold $\tau > 0$.}
\KwInput{Matrix-vector oracle for a matrix $B \in \R^{m \times n}$.} 
\tcp{Check if pair of unit vectors has large bilinear form in $B$}
\lIf{$z\y^\top B z\x > \tau \normInline{z\y}_2\normInline{z\x}_2 \label{line:if1-2-newer}$}{
    $v \gets \normalize(z\y)$ and 
    $u \gets \normalize(z\x)$ 
}
\lElseIf{${z'\y}^\top B z'\x > \tau \normInline{z'\y}_2\normInline{z'\x}_2$ \label{line:if2-2-newer}}{
    $v \gets \normalize(z'\y)$ and $u \gets \normalize(z'\x)$ 
}
\lElse{
    \Return{$(\smooth, 0)$}
}
$D \gets  (v^\top B u) \cdot vu^\top$ \label{line:rank-one-update-apx}\; 
\Return{$({\guilty},  D)$}
\end{algorithm2e}

\begin{restatable}{lemma}{smoothguiltyjudgefrob} $\judge_F$ (Algorithm~\ref{alg:frob-judge-l2l2}) is a $2$-smooth-guilty-judge. 
\end{restatable}
\begin{proof} If neither of the if statements (Lines~\ref{line:if1-2-newer} or \ref{line:if2-2-newer}) execute, then the algorithm outputs $(\smooth, 0)$ as required. Otherwise, we have that $v^\top B u > \tau$, and we need to prove that $\norm{ B - D }_{F}^2 \leq \norm{B}_{F}^2 - (v^\top B u)^2.$  First, note that $B - D = B - v^\top B u \cdot vu^\top$. Now,
\begin{align*}
    \normInline{B - D}_F^2 &= \tr(B^\top B - 2B^\top D + D^\top D) \\
    &= \normInline{B}_F^2 - 2 (v^\top B u) \cdot \tr(B^\top vu^\top) + (v^\top B u)^2 \tr(v u^\top u v^\top) \\
    &= \normInline{B}_F^2 - 2 (v^\top B u)^2 \leq \normInline{B}_F^2 - 2 (v^\top B u)^2 \\
    &\leq \normInline{B}_F^2 - \tau^2. 
\end{align*}
The algorithm also clearly requires only $O(1)$-matrix-vector queries. 
\end{proof}

As discussed in Section~\ref{sec:four-fifths}, yet another way to implement a $2$-smooth guilty judge is shown in Algorithm~\ref{alg:2-proj-judge} and this is important to our proof of Theorem~\ref{intro:l2l1}.

\end{document}